\documentclass[a4paper,12pt]{article}
\usepackage[latin1]{inputenc}
\usepackage{amsmath,amsrefs,amssymb,amsthm}
\usepackage{graphicx}
\usepackage[bookmarks=true,pdfstartview={FitH}]{hyperref}
\usepackage{enumerate,enumitem}
\usepackage{mathrsfs}
\usepackage{epstopdf}
\usepackage{fullpage}
\usepackage[small]{caption}
\newtheorem{defi}{Definition}[section]
\newtheorem{lem}[defi]{Lemma}
\newtheorem{thm}[defi]{Theorem}
\newtheorem{cor}[defi]{Corollary}
\newtheorem{prop}[defi]{Proposition}
\theoremstyle{remark}
\newtheorem{rem}[defi]{Remark}
\numberwithin{equation}{section}

\def\Ii{\ensuremath{\mathcal{I}}}

\def\Cc{\ensuremath{\mathscr{C}}}

\def\C{\ensuremath{\mathcal{C}}}
\def\Cb{\ensuremath{C_{\rm b}}}

\def\D{\ensuremath{\mathscr{D}}}
\def\Dd{\ensuremath{\mathfrak{D}}}
\def\dP{\ensuremath{d_{\rm P}}}

\def\dMGP{\ensuremath{d_{\rm mGP}}}
\def\E{\ensuremath{\mathbb{E}}}
\def\ep{\varepsilon}
\def\p{\ensuremath{\mathcal{P}}}
\def\P{\ensuremath{\mathbb{P}}}
\def\Q{\ensuremath{\mathbb{Q}}}

\def\M{\ensuremath{\mathbb{M}}}
\def\Mb{\ensuremath{M_{\rm b}}}
\def\N{\ensuremath{\mathbb{N}}}
\def\Nn{\ensuremath{\mathcal{N}}}
\def\R{\ensuremath{\mathbb{R}}}
\def\S{\ensuremath{\mathcal{S}}}

\def\U{\ensuremath{\mathbb{U}}}
\def\Uu{\ensuremath{\mathfrak{U}}}
\def\UU{\ensuremath{\mathcal{U}}}
\def\UUerg{\ensuremath{\mathcal{U}^{\rm erg}}}

\newcommand{\reff}[1]{(\ref{#1})}

\newcommand{\I}[1]{\mathop {\mathbf{1}{\left\{ #1\right\}}}}
\newcommand{\abs}[1]{\mathop {\left|{ #1}\right|}}
\newcommand{\hparagraph}[1]{\paragraph{#1}\mbox{}\\}

\begin{document}
\label{Sampl}
\title{A representation for exchangeable coalescent trees and generalized tree-valued Fleming-Viot processes}
\author{Stephan Gufler\thanks{Technion, Faculty of Industrial Engineering and Management, Haifa 3200003, Israel, \texttt{stephan.gufler@gmx.net}}
}
\maketitle
\begin{abstract}
We give a de Finetti type representation for exchangeable random coalescent trees (formally described as semi-ultrametrics) in terms of sampling iid sequences from marked metric measure spaces. We apply this representation to define versions of tree-valued Fleming-Viot processes from a $\Xi$-lookdown model. As state spaces for these processes, we use, besides the space of isomorphy classes of metric measure spaces, also the space of isomorphy classes of marked metric measure spaces and a space of distance matrix distributions. This allows to include the case with dust in which the genealogical trees have isolated leaves. 
\end{abstract}
{\small
}
{\scriptsize \tableofcontents }

\section{Introduction}
\subsection{Some background on coalescent trees, ultrametrics, and metric measure spaces}
\label{Sampl:sec:backgr}
In population genetics, coalescents are common models for the genealogy of a sample from a population.
The Kingman coalescent \cite{K82-SPA} is a partition-valued process in which each individual of the sample forms its own block at time $0$, and as we look into the past, each pair of blocks merges independently at constant rate.
These blocks stand for the families of individuals that have a common ancestor at given times in the past.
Generalizations of the Kingman coalescent include
the $\Lambda$-coalescent (Pitman \cite{Pitman99}, Sagitov \cite{Sagitov99}, Donnelly and Kurtz \cite{DK99}) where multiple blocks are allowed to merge to a single block at the same time, and the $\Xi$-coalescent (Möhle and Sagitov \cite{MS01}, Schweinsberg \cite{Schw00}) where several clusters of blocks may also merge simultaneously.

A (semi-)ultrametric $\rho$ is a (semi-)metric that satisfies the strong triangle inequality
$\max\{\rho(x,y), \rho(y,z)\} \geq\rho(x,z)$. A realization of a coalescent for an infinite sample can be expressed as a càdlàg path $(\pi_t,t\in\R_+)$ with values in the space of partitions of $\N$ such that $\pi_t$ is a coarsening of $\pi_s$ for all $s\leq t$. We assume that for each pair of integers, there is a time $t$ such that the elements of this pair are in a common block of $\pi_t$. Then $(\pi_t,t\in\R_+)$ can equivalently be expressed as a semi-ultrametric $\rho$ on $\N$ such that for all $t\in\R_+$ and $i,j\in\N$,
\begin{equation}
\label{Sampl:eq:coal-ultrametr}
\rho(i,j)\leq 2t\quad\text{if and only if $i$ and $j$ are in the same block of $\pi_t$,}
\end{equation}
and \reff{Sampl:eq:coal-ultrametr} yields a one-to-one correspondence between these càdlàg paths and the semi-ultrametrics on $\N$, cf.\ \cite{E08}*{Example 3.41} and \cite{EGW15}*{p.\,262}.

Evans \cites{E00} studies the completion of the random ultrametric space associated with the Kingman coalescent which he endows with a probability measure such that the mass on each ball is given by the asymptotic frequency of the corresponding family, and a class of more general coalescents is studied by Berestycki et al.\ \cite{BBS08}.

\begin{rem}
\label{Sampl:rem:ultrametr}
Let us briefly recall the well-known correspondence between ultrametric spaces and real trees to which we will refer to explain main concepts in this article.
A real tree is a metric space $(T,d)$ that is tree-like in the sense that
(i) no subspace is homeomorphic to the unit circle, and (ii)
for each $x,y\in T$, there exists an isometry $\iota$ from the real interval $[0,d(x,y)]$ to $T$ with $\iota(0)=x$ and $\iota(d(x,y))=y$, see e.\,g.\ Evans \cite{E08} for an overview.
An ultrametric space $(X,\rho)$ can be isometrically embedded into the real tree $(T,d)$ that is obtained by identifying the elements with distance zero of the semi-metric space $(\R_+\times X,d)$ given by $d((s,i),(t,j))=\max\{\rho(i,j)-s-t,|s-t|\}$. Then $T$ equals the set $\mathbf{U}^0$ on in \cite{EGW15}*{p.\,262} with $X=\N$ and the metrics $d$ here and in \cite{EGW15}*{p.\,262} coincide up to a factor $2$.
Clearly, $(X,\rho)$ is isometric to the subspace $\{0\}\times X$ of the leaves of $(T,d)$.
For a semi-ultrametric space $(X,\rho)$, we identify the elements with distance zero to obtain an ultrametric space which we associate with a real tree $(T,d)$ as above. A related embedding of an ultrametric space is given in \cite{Hughes}*{Section 6}.
\end{rem}

As in Remark~\ref{Sampl:rem:ultrametr}, a semi-ultrametric on $\N$ can be considered as an infinite tree whose leaves are labeled by the elements of $\N$. Often these labels are not relevant, for instance, when they only record the order in which iid samples from a population are drawn.
To remove the labels, we could pass to the isometry class. However, the asymptotic block frequencies in the coalescent given by an ultrametric on $\N$ are not determined by the isometry class, as one may apply an infinite permutation without changing the isometry class. To retain just this information besides the metric structure, we can take a measure-preserving isometry class of the completion of the ultrametric space that is endowed with a probability measure that charges each ball with the asymptotic frequency of the corresponding block, if such a probability measure exists. This probability measure can equivalently be described as the weak limit of the uniform probability measures on the individuals $1,\ldots,n$, as $n\to\infty$. Then we obtain the description by isomorphy classes of metric measure spaces of Greven, Pfaffelhuber, and Winter \cite{GPW09} that was applied to $\Lambda$-coalescents in the dust-free case. We speak of the dust-free case if the semi-ultrametric space has no isolated points, which means that the coalescent tree has no isolated leaves. Greven, Pfaffelhuber, and Winter \cite{GPW09} also show that their approach is not directly applicable to $\Lambda$-coalescents with dust. The most elementary example for the case with dust is the star-shaped coalescent which starts in the partition into singleton blocks which all merge into a single block at some instant. The associated ultrametric on $\N$ induces the discrete topology. Here the uniform probability measures on $1,\ldots,n$ do not converge weakly as they converge vaguely to the zero measure.

A triple $(X,r,\mu)$ that consists of a complete and separable metric space $(X,r)$ and a probability measure $\mu$ on the Borel sigma algebra on $X$ is called a metric measure space. For a metric measure space $(X,r,\mu)$, one can consider the matrix $(r(x(i),x(j)))_{i,j\in\N}$ of the distances between $\mu$-iid samples $(x(i))_{i\in\N}$. The distribution of $(r(x(i),x(j)))_{i,j\in\N}$ is called the distance matrix distribution\label{Sampl:not:dmd} of $(X,r,\mu)$. By the Gromov reconstruction theorem (see Theorem 4 of Vershik \cite{Vershik02}), there exists a measure-preserving isometry between the supports of the measures of any two metric measure spaces that have the same distance matrix distribution, in which case we call them isomorphic.

We view a random semi-metric $\rho$ on $\N$ as the random matrix $(\rho(i,j))_{i,j\in\N}$, and we call it exchangeable if $(\rho(i,j))_{i,j\in\N}$ is distributed as $(\rho(p(i),p(j)))_{i,j\in\N}$ for each (finite) permutation $p$ of $\N$. Under an appropriate condition which we interpret as dust-freeness in Remark \ref{rem:mm-Vershik}, Vershik \cite{Vershik02}*{Theorem 5} associates with any typical realization of an exchangeable (and ergodic) random semi-metric on $\N$ a metric measure space whose distance matrix distribution is the distribution of this semi-metric. In the next subsection, we discuss an extension of such a representation to the case with dust.

\subsection{The sampling representation}
\label{Sampl:sec:intr-sampl}
We give a representation for all exchangeable random semi-ultrametrics on $\N$ in terms of sampling from random marked metric measure spaces.
Marked metric measure spaces are introduced in Depperschmidt, Greven, and Pfaffelhuber \cite{DGP11}.
A $(\R_+)$-marked metric measure space is a triple $(X,r,m)$ that consists of a complete and separable metric space $(X,r)$ and a probability measure $m$ on the Borel sigma algebra on the product space $X\times\R_+$.
The marked distance matrix distribution\label{Sampl:not:mdm} of a marked metric measure space $(X,r,m)$ is defined as the distribution of $((r(x(i),x(j)))_{i,j\in\N},(v(i))_{i\in\N})$ where $(x(i),v(i))_{i\in\N}$ is an $m$-iid sequence in $X\times\R_+$. Marked metric measure spaces with the same marked distance matrix distribution are called isomorphic.

In the present article, we use marked metric measure spaces to obtain from a random variable $(\tilde r,\tilde v)$ that has the marked distance matrix distribution of a marked metric measure space an exchangeable semi-metric $\tilde\rho$ on $\N$ by
\begin{equation*}
\tilde\rho(i,j)=(\tilde r(i,j)+\tilde v(i)+\tilde v(j))\I{i\neq j}.
\end{equation*}
We call the distribution of $(\tilde\rho(i,j))_{i,j\in\N}$ the distance matrix distribution of the marked metric measure space.
The basic result in this article (stated in Theorem \ref{thm:repr} below) is that every exchangeable semi-ultrametric $\rho$ on $\N$ can be represented as the outcome of a two-stage random experiment, where we have the isomorphy class $\chi$ of a random marked metric measure space in the first stage, and we sample $(\rho(i,j))_{i,j\in\N}$ from this marked metric measure space according to its distance matrix distribution in the second stage.

We construct $\chi$ realization-wise from the exchangeable semi-ultrametric $\rho$: the key idea is to decompose the tree that is associated with a realization of $\rho$ into the external branches and the remaining subtree.
Here we define that an external branch consists only of the leaf if that leaf corresponds to an integer that has $\rho$-distance zero to another integer.
In the marked metric measure space, the marks encode the external branch lengths, and the metric space describes the remaining subtree.
We call the semi-ultrametric dust-free if the external branches all have length zero a.\,s. In this case, the marked metric measure space can also be replaced by a metric measure space (as in Corollary \ref{cor:repr}).
We prove Theorem~\ref{thm:repr} in Section \ref{sec:proofs}.
In Section~\ref{Sampl:sec:dist-dec}, we formulate the decomposition at the external branches in terms of semi-ultrametrics.

The representation for exchangeable semi-ultrametrics from Theorem \ref{thm:repr} can also be seen in the more general but less explicit contexts of the ergodic decomposition (Section \ref{sec:ergodic}) and the Aldous-Hoover-Kallenberg representation (see e.\,g.\ \cite{Kal05}*{Section 7}).
In the representation result outlined above, the distance matrix distribution of the isomorphy class $\chi$ of the marked metric measure space is the ergodic component in whose support the realization $\rho$ lies. The ergodic component is also characterized by $\chi$ itself, or in the dust-free case by the isomorphy class of a metric measure space. The finite analog of the aforementioned ergodic decomposition is that a (discrete) random tree whose leaves are labeled exchangeably can be obtained by first drawing the random unlabeled tree and then sampling the labels of the leaves uniformly without replacement.

We mention that Evans, Grübel, and Wakolbinger \cite{EGW15} also decompose real trees into the external branches and the remaining subtree to give a representation of the elements of the Doob-Martin boundary of Rémy's algorithm in terms of sampling from a weighted real tree and an additional structure.
In \cite{EGW15}*{Section 7}, a sampling representation for exchangeable ultrametrics is considered (see Remark~\ref{Sampl:rem:EGW}).

\subsection{Evolving genealogies}
\label{Sampl:sec:intr-evol}
In Section~\ref{Sampl:sec:tv}, we lay the foundation for our study of evolving genealogies by considering a general time-homogeneous Markov process with values in the space of semi-ultrametrics on $\N$; this process describes evolving leaf-labeled trees. Assuming that the state at each time is exchangeable, we map this process realization-wise to the processes of the ergodic components. We express these ergodic components as (isomorphy classes of) metric measure spaces and marked metric measure spaces, and as distance matrix distributions, respectively.
Here we use the representation result for exchangeable semi-ultrametrics.
This approach characterizes the processes of the ergodic components up to null events only at countably many time points, i.\,e.\ as versions, as we discuss in Remark \ref{rem:version}.
Using the criterion of Rogers and Pitman \cite{RP81}*{Theorem 2},
we deduce that these image processes are also Markovian, and we describe them by well-posed martingale problems. This is an example of Markov mapping in the sense
of Kurtz \cite{Kurtz98}, and Kurtz and Nappo \cite{KurtzNappo11}.

In Sections~\ref{Sampl:sec:ld-gen} --~\ref{Sampl:sec:dec-proc}, we study a concrete Markov process with values in the space of semi-ultrametrics, namely the process given by the evolving genealogical trees in a lookdown model with simultaneous multiple reproduction events.
Lookdown models were introduced by Donnelly and Kurtz \cites{DK96,DK99} to represent measure-valued processes along with their genealogy, see also e.\,g.\ Etheridge and Kurtz \cite{EK14} and Birkner et al.\ \cite{BBMST09}. A lookdown model can be seen as a (possibly) infinite population model in which
each individual at each time is assigned a level.
The role of this level is model-inherent, namely to order the individuals such that the restriction of the model to the first finitely many levels is well-behaved (i.\,e.\ only finitely many reproduction events are visible in bounded time intervals) and that the modeled quantity (e.\,g.\ types, genealogical distances) is exchangeable.
In \cite{DK99} and in the present article, the level is the rank among the individuals at the respective time according to the time of the latest descendant. 
Although the levels in finite restrictions of the lookdown model differ from the labels in the Moran model,
the processes of the unlabeled genealogical trees coincide which is used to study the length of the genealogical trees in Pfaffelhuber, Wakolbinger, and Weisshaupt \cite{PWW11} and Dahmer, Knobloch, and Wakolbinger \cite{DKW14}.

In Section~\ref{Sampl:sec:TVFV}, we remove the labels from the evolving genealogical trees in the infinite lookdown model by applying the result from Section~\ref{Sampl:sec:tv} to the process from Sections~\ref{Sampl:sec:ld-gen} --~\ref{Sampl:sec:dec-proc}. We call the processes of the ergodic components tree-valued Fleming-Viot processes, regardless which one of the three state spaces we use. The tree-valued Fleming-Viot process with values in the space of isomorphy classes of metric measure spaces is introduced in the case with binary reproduction events (which is associated with the Kingman coalescent) by Greven, Pfaffelhuber, and Winter \cite{GPW13} as the solution of a well-posed martingale problem that is the limit in distribution of corresponding processes read off from finite Moran models. In \cite{GPW13}*{Remark 2.20}, a construction of (a version of) this process from the lookdown model of Donnelly and Kurtz \cite{DK96} is outlined. The aim in the present article regarding tree-valued Fleming-Viot process is the generalization to the case with dust.
We remark that tree-valued Fleming-Viot processes with mutation and selection are studied in Depperschmidt, Greven, and Pfaffelhuber \cites{DGP12,DGP13} where the states are isomorphy classes of marked metric measure spaces and the marks encode allelic types. In the present article, the marks encode lengths of external branches. We consider only the neutral case, and we describe genealogies without using types.

In Section~\ref{Sampl:sec:semigroup}, we show continuity properties of the semigroups of tree-valued Fleming-Viot processes and that the domains of the martingale problems for them are cores.
In Section~\ref{Sampl:sec:equil}, we show that tree-valued Fleming-Viot processes converge in distribution to equilibrium.

While we construct versions of tree-valued Fleming-Viot processes in the present article using the representation result, the full sample paths are constructed by techniques specific to the lookdown model in the companion article \cite{Pathw}.

\subsection{Additional related literature}
Aldous \cite{Ald93} represents consistent families of finite trees that satisfy a ``leaf-tight'' property by random measures on $\ell_1$ (and random subsets of $\ell_1$).
Kingman's coalescent is given as an example in \cite{Ald93}. The ``leaf-tight'' property corresponds to the absence of dust.
A representation for exchangeable hierarchies in terms of sampling from random weighted real trees is given by Forman, Haulk, and Pitman \cite{FHP15}.
There are many other representation results for exchangeable structures in the literature.
For instance, by the Dovbysh-Sudakov theorem, see Austin \cite{Austin15} for a proof based on a representation for exchangeable random measures,
jointly exchangeable arrays that are non-negative definite can be represented in terms of sampling from the space $L_2[0,1]\times\R_+$.

The genealogy in the lookdown model is further studied in Pfaffelhuber and Wakolbinger \cite{PW06}.
Kliem and Löhr \cite{KL14} further study marked metric measure spaces. In their article, tree-valued $\Lambda$-Fleming-Viot processes in the dust-free case is also mentioned. Kliem and Winter \cite{KW17} use marked metric measure spaces to describe trait-dependent branching processes. 
In the context of measure-valued spatial $\Lambda$-Fleming-Viot processes with dust, Véber and Wakolbinger \cite{VW15} work with a skeleton structure.
Functionals of coalescents like external branch lengths have also been studied, see for example \cite{M10}. Also the time evolution of such functionals has been studied for evolving coalescents, see for example \cites{KSW14,DK14}.

Bertoin and Le Gall \cites{BLG03,BLG05,BLG06} represent $\Xi$-coalescents in terms of sampling from flows of bridges from which they also construct measure-valued Fleming-Viot processes. They also consider mass coalescents. Mass coalescents (see e.\,g.\ Chapter 4.3 in Bertoin \cite{Bertoin}) also describe genealogies without labeling individuals.
In Section \ref{Sampl:sec:bridges}, we construct the Fleming-Viot process with values in the space of distance matrix distributions from the dual flow of bridges. We also mention the work of Labbé \cite{Lab12} where relations between the lookdown model and flows of bridges are studied.

\section{Distance matrices and their decompositions}
\label{Sampl:sec:dist-dec}
We write $\N=\{1,2,3,\ldots\}$. Let $\Uu$\label{Sampl:not:Uu} denote the space of semi-ultrametrics on $\N$ and let $\Dd$\label{Sampl:not:Dd} denote the space of semimetrics on $\N$. We view $\Uu$ and $\Dd$ as subspaces of $\R^{\N^2}$ in that we do not distinguish between a semi-metric $\rho$ and the distance matrix $(\rho(i,j))_{i,j\in\N}$. We endow $\R^{\N^2}$ with a complete and separable metric that induces the product topology when $\R$ is equipped with the Euclidean topology. 
Using the map\label{Sampl:not:alpha}
\[\alpha:\R_+^{\N^2}\times\R_+^{\N}\to\R_+^{\N^2},\quad
(r,v)\mapsto((v(i)+r(i,j)+v(j))\I{i\neq j})_{i,j\in\N},\]
we define the space
\[\hat\Uu=\{(r,v)\in\Dd\times\R_+^{\N}:\alpha(r,v)\in\Uu\}\label{Sampl:not:hatUu}\]
whose elements we call decomposed semi-ultrametrics or marked distance matrices. As above, we view $\Dd\times\R_+^\N$ and $\hat\Uu$ as subspaces of $\R^{\N^2}\times\R^\N$ which we endow with a complete and separable metric that induces the product topology.

We define the function\label{Sampl:not:Upsilon}
\[\Upsilon:\Uu\to\R_+^\N,\quad\rho\mapsto(\tfrac{1}{2}\inf_{j\in\N\setminus\{i\}}\rho(i,j))_{i\in\N},\]
and we denote by $\beta$\label{Sampl:not:beta} the function that maps a semi-ultrametric $\rho\in\Uu$ to the decomposed semi-ultrametric $(r,v)\in\hat\Uu$ that is given by $v=\Upsilon(\rho)$ and
\[r(i,j)=(\rho(i,j)-v(i)-v(j))\I{i\neq j}\]
for $i,j\in\N$.
The interpretation of these functions is given in Remark~\ref{Sampl:rem:um-tree} below from which it follows that $r$ is a tree-like
semi-metric (i.\,e., $r$ is $0$-hyperbolic, see e.\,g.\ \cite{E08}). Alternatively, it can be easily checked that $r$ satisfies the triangle inequality.

The function $\alpha$ retrieves the semi-ultrametric from a decomposed semi-ultrametric. For instance, $\alpha\circ\beta$ is the identity map on $\Uu$.
\begin{rem}
Let us agree on the following notation. When we identify the elements of a semi-metric space $(X,\rho)$ that have $\rho$-distance zero to obtain a metric space $(X',\rho)$, we refer by each element $x\in X$ also to the associated element of $X'$. Furthermore, we define the metric completion of the semi-metric space $(X,\rho)$ as the metric completion of $(X',\rho)$.
\end{rem}
\begin{rem}
\label{Sampl:rem:um-tree}
Let $\rho\in\Uu$, $(r,v)=\beta(\rho)$, and let $(T,d)$ be the real tree associated with $\rho$ as in Remark~\ref{Sampl:rem:ultrametr} with $X=\N$.
Then $v(i)=\Upsilon(\rho)(i)$ can be interpreted as the length, and $(i,v(i))$ as the starting vertex of the external branch that ends in the leaf $(i,0)$ of $T$. Here we define that this external branch consists only of the leaf if there exists $k\in\N\setminus\{i\}$ with $\rho(i,k)=0$.
Furthermore, the map $\varphi(i)=(i,v(i))$ from $(\N,r)$ to $(T,d)$ is distance-preserving.

In this sense, the map $\beta:\rho\mapsto(r,v)$ decomposes the coalescent tree that is given by $\rho$ into the external branches with lengths $v$ and the subtree spanned by their starting vertices whose mutual distances are given by $r$. More generally, any element of $\hat\Uu$ can be seen as a decomposed coalescent tree.
\end{rem}

We call a semi-ultrametric $\rho\in\Uu$ dust-free if $\Upsilon(\rho)=0$, that is, if all external branches in the associated tree have length zero so that there are no isolated leaves.

\section{Sampling from marked metric measure spaces}
\label{Sampl:sec:dec-mmm}

\subsection{Preliminaries}
Recall the definitions of metric measure spaces, marked metric measure spaces, and their (marked) distance matrix distributions from Sections \ref{Sampl:sec:backgr} and \ref{Sampl:sec:intr-sampl}. Also recall that two metric measure spaces are said to be isomorphic if they have the same distance matrix distributions. We denote the set of isomorphy classes of metric measure spaces by $\M$\label{Sampl:not:M} and we endow it with the Gromov-weak topology in which metric measure spaces converge if and only if their distance matrix distributions converge. Greven, Pfaffelhuber, and Winter \cite{GPW09} showed that $\M$ is then a Polish space.

Analogously, two marked metric measure spaces are said to be isomorphic if they have the same marked distance matrix distributions.  We denote the set of isomorphy classes of marked metric measure spaces by \label{Sampl:not:hatM} $\hat\M$ and we endow it with the marked Gromov-weak topology in which marked metric measure spaces converge if and only if their marked distance matrix distributions converge weakly. This makes $\hat\M$ a Polish space, as shown by Depperschmidt, Greven, and Pfaffelhuber \cite{DGP11}.

We denote the distance matrix distribution\label{Sampl:not:nuchi} of the isomorphy class of a metric measure space $\chi\in\M$ by $\nu^\chi$. We denote the marked distance matrix distribution of $\chi'\in\hat\M$ by $\nu^{\chi'}$, so that $\alpha(\nu^{\chi'})$ is the distance matrix distribution of $\chi'$, in accordance with the definition in Section \ref{Sampl:sec:intr-sampl}. (We denote by $\varphi(\xi)=\xi\circ\varphi^{-1}$ the pushforward measure of a measure $\xi$ on a measurable space $E$ under a measurable function $\varphi$ on $E$.)\label{not:pushforward} 

\begin{rem}
\label{rem:mmm-dust-free}
We call a marked metric measure space $(X,r,m)$ dust-free if the probability measure $m$ is of the form $m=\mu\otimes\delta_0$ for a probability measure $\mu$ on the Borel sigma algebra on $X$.
Then the distance matrix distribution $\alpha(\nu^{(X,r,\mu\otimes\delta_0)})$ equals the distance matrix distribution $\nu^{(X,r,\mu)}$ of the metric measure space $(X,r,\mu)$. We call $(X,r,\mu)$ the metric measure space associated with the dust-free marked metric measure space $(X,r,\mu\otimes\delta_0)$.
\end{rem}

Let $S_\infty$ denote the group of finite permutations on $\N$. We define the action of $S_\infty$ on $\Dd$ and $\Dd\times\R_+^\N$, respectively, by
$p(\rho)=(\rho(p(i),p(j)))_{i,j\in\N}$ and
\[p(r,v)=((r(p(i),p(j)))_{i,j\in\N},(v(p(i))_{i\in\N})\]
for $p\in S_\infty$, $\rho\in\Dd$, $(r,v)\in\Dd\times\R_+^\N$.
A random variable, for instance with values in $\Dd$ or $\Dd\times\R_+^\N$, is called exchangeable if its distribution is invariant under the action of the group $S_\infty$.
\begin{rem}
Exchangeable random variables with values in $\Dd$ or $\Dd\times\R_+^\N$ can be seen as jointly exchangeable arrays, see e.\,g.\ \cite{Kal05}*{Section 7}. Also recall that the definition of exchangeability does not change when $S_\infty$ is replaced with the group of all bijections from $\N$ to itself, as the finite restrictions determine the distribution of a random variable in $\Dd$ or $\Dd\times\R_+^\N$.
\end{rem}
\begin{rem}
The coalescents associated by \reff{Sampl:eq:coal-ultrametr} with the exchangeable semi-ultrametrics on $\N$ form a larger class of processes than the so-called exchangeable coalescents defined in e.\,g.\ Section 4.2.2 of Bertoin \cite{Bertoin}. For example, the coalescent process associated with an exchangeable semi-ultrametric on $\N$ needs not be Markovian.
\end{rem}

\subsection{Tree-like marked metric measure spaces}
\label{Sampl:sec:mmm-trees}
We consider the space\label{Sampl:not:U}
\[\U=\{\chi\in\M:\nu^\chi(\Uu)=1\}\]
of ultrametric measure spaces which is a closed subspace of $\M$, as shown in \cite{GPW13}*{Lemma 2.3}.
By the same argument, the space\label{Sampl:not:hatU}
\[\hat\U=\{\chi\in\hat\M:\alpha(\nu^\chi)(\Uu)=1\}.\]
is a closed subspace of $\hat\M$. It contains the marked metric measure spaces with ultrametric distance matrix distribution.
Following e.\,g.\ \cites{GPW09,GPW13} and Remark \ref{Sampl:rem:ultrametr}, we call the elements of $\U$ trees. Also the elements of $\hat\U$ may be called trees (as in Remark \ref{Sampl:rem:mmm-tree-erg} below).

Proposition~\ref{Sampl:prop:sampl-unique} below states that a.\,e.\ realization of a $\hat\Uu$-valued random variable with the marked distance matrix distribution of a marked metric measure space in $\hat\U$ is the decomposition of a semi-ultrametric by the map $\beta$ from Section~\ref{Sampl:sec:dist-dec}. As a consequence, the isomorphy class of a marked metric measure space in $\hat\U$ is determined already by its distance matrix distribution.
\begin{prop}
\label{Sampl:prop:sampl-unique}
Let $(X,r',m)$ be a marked metric measure space with $\alpha(\nu^{(X,r',m)})(\Uu)=1$. Let $(r,v)$ be a $\hat\Uu$-valued random variable with distribution $\nu^{(X,r',m)}$. Then
\[(r,v)=\beta\circ\alpha(r,v)\quad\text{a.\,s.}\]
\end{prop}
The proof is deferred to Section \ref{sec:proof:sampl-unique}.
\begin{rem}
\label{rem:sampling-dust-free}
We call a semi-ultrametric $\rho\in\Uu$ dust-free if $\Upsilon(\rho)=0$. It can be seen as a consequence of Proposition \ref{Sampl:prop:sampl-unique} that (the isomorphy class of) a marked metric measure space $(X,r,m)$ in $\hat\U$ is dust-free (as defined in Remark \ref{rem:mmm-dust-free}) if and only if a random variable with distribution $\alpha(\nu^{(X,r,m)})$ is a.\,s.\ dust-free.
In particular, a random variable with the distance matrix distribution of a metric measure space is a.\,s.\ dust-free.
\end{rem}

\subsection{Marked metric measure spaces from marked distance matrices}
\label{Sampl:sec:mdm-mmm-new}
In this subsection, we define functions by which we construct a (marked) metric measure space from a (marked) distance matrix. An interpretation of these functions is given in Remark~\ref{Sampl:rem:psi-families} below. In Remark \ref{rem:erg-dec-map}, we state their role in the context of the ergodic decomposition.

First we define the function $\psi:\Dd\to\M$\label{Sampl:not:psi} that maps $\rho\in\Dd$ to the isomorphy class of the metric measure space $(X,\rho,\mu)$, given as follows: $(X,\rho)$ is the metric completion of $(\N,\rho)$. The probability measure $\mu$ is defined as the weak limit of the probability measures $n^{-1}\sum_{i=1}^n\delta_i$ as $n$ tends to infinity, if this weak limit exists. If the limit does not exist, we define $m$ arbitrarily, let us set $\mu=\delta_1$. Furthermore, we denote by $\Dd^*$\label{Sampl:not:Dds} the subset of distance matrices $\rho\in\Dd$ such that the weak limit in the definition above exists.

Analogously, we define the function $\hat\psi:\Dd\times\R_+^\N\to\hat\M$\label{Sampl:not:hatpsi} that maps $(r,v)$ to the isomorphy class of the marked metric measure space $(X,r,m)$, where $(X,r)$ is the metric completion of the semi-metric space $(\N,r)$ and $m$ is the weak limit of the probability measures $n^{-1}\sum_{i=1}^n\delta_{(i,v(i))}$ on $X\times\R_+$ if this weak limit exists, else we set $m=\delta_{(1,0)}$. We denote by $\hat\Dd^*$\label{Sampl:not:hatDds} the subset of marked distance matrices $(r,v)\in\Dd\times\R_+^\N$ such that the weak limit in the definition above exists.

We call $\mu$ and $m$ in the definitions of $\psi$ and $\hat\psi$ also sampling measures.

\begin{rem}
\label{Sampl:rem:psi-hatpsi}
Let $(r,v)\in\Dd\times\R_+^\N$. Then $(r,v)\in\hat\Dd^*$ implies $r\in\Dd^*$. For a representative $(X,r,m)$ of $\hat\psi(r,v)$, the isomorphy class of $(X,r,m(\cdot\times\R_+))$ equals $\psi(r)$.
\end{rem}

\begin{prop}
\label{Sampl:prop:meas}
The functions $\psi$ and $\hat\psi$ are measurable.
\end{prop}
The proof, in which we write $\psi$ and $\hat\psi$ as limits of continuous functions, is deferred to Section~\ref{Sampl:sec:proof:meas}.

\begin{rem}[An interpretation of $\psi$ and $\hat\psi$]
\label{Sampl:rem:psi-families}
For $\rho\in\Dd^*\cap\Uu$, the probability measure in the ultrametric metric measure space $\psi(\rho)$ charges each ball with the asymptotic frequency of the corresponding block of the coalescent which is associated with $\rho$ by \reff{Sampl:eq:coal-ultrametr}.

Similarly, for $(r,v)\in\hat\Dd^*\cap\hat\Uu$, let $(X,r,m)$ be the representative of  $\hat\psi(r,v)$ from the definition of $\hat\psi$. We consider the completion $(\bar T,d)$ of the real tree $(T,d)$ associated with $(r,v)$ as in Remark \ref{Sampl:rem:um-tree}, and the extension $\varphi:X\to\bar T$ of the isometry $\varphi$ from Remark \ref{Sampl:rem:um-tree}. Then the image measure $\mu:=\varphi(m(\cdot\times\R_+))$ charges each region of $\bar T$ with the asymptotic frequency of the integers that label the leaves of $T$
that are the endpoints of external branches that begin in that region.
\end{rem}

\subsection{The sampling representation}
\label{Sampl:sec:repr}
The basic result in this paper is stated in Theorem \ref{thm:repr} below. Here we consider an exchangeable random semi-ultrametric $\rho$ on $\N$, and we assert existence of a random variable $\chi$ with values in the space of isomorphy classes of marked metric measure spaces that has the following property: Let $\rho'$ be a random variable whose conditional distribution given $\chi$ is the distance matrix distribution of $\chi$. Then the random variables $\rho$ and $\rho'$ have the same (unconditional) distribution. (In the language of the theory of random measures, this means that the distribution of $\rho$ is equal to the first moment measure $\E[\alpha(\nu^\chi)]$. That is, $\E[\int\alpha(\nu^\chi)(d\rho')\phi(\rho')]=\E[\phi(\rho)]$ for each bounded measurable $\phi$.)
\begin{thm}
\label{thm:repr}
Let $\rho$ be an exchangeable $\Uu$-valued random variable. Let $\chi=\hat\psi\circ\beta(\rho)$. Let $\rho'$ be a $\Uu$-valued random variable whose conditional distribution given $\chi$ is $\alpha(\nu^\chi)$. Then:
\begin{enumerate}[label=(\roman{*}),ref=(\roman{*})]
\item\label{item:thm:repr:sampling-psi} $\beta(\rho)\in\hat\Dd^*$ a.\,s.
\item\label{item:thm:repr:sampl} $\rho$ and $\rho'$ are equal in distribution.
\item\label{item:thm:repr:det} $\chi=\hat\psi\circ\beta(\rho')$ a.\,s.
\end{enumerate}
\end{thm}
Assertion \ref{item:thm:repr:sampling-psi} above states that for a typical realization of $\rho$ and its decomposition $\beta(\rho)$, the sampling measure $m$ in the definition of $\hat\psi(\beta(\rho))$ in Subsection \ref{Sampl:sec:mdm-mmm-new} is the weak limit of the uniform probability measures therein.
Assertion \ref{item:thm:repr:det} states that the realization of $\chi$ can typically be reconstructed from the realization of $\rho'$.
We interpret the reconstruction map $\hat\psi\circ\beta$ in terms of the ergodic decomposition in Remark \ref{rem:erg-dec-map}.
We prove Theorem \ref{thm:repr} in Section~\ref{Sampl:sec:proof:sampling}. We give two proofs of Theorem \ref{thm:repr}\ref{item:thm:repr:sampling-psi}.
In one of them, the de Finetti theorem yields the aforementioned sampling measure $m$ as the directing measure of an exchangeable sequence.
\begin{rem}
\label{rem:cond-repr}
In the context of Theorem \ref{thm:repr}, $(\rho,\alpha(\nu^\chi))$ and $(\rho',\alpha(\nu^\chi))$ are equal in distribution. Hence, $\alpha(\nu^\chi)$ is a regular conditional distribution of $\rho$ given $\alpha(\nu^\chi)$.
\end{rem}

We also note the following uniqueness property which is proved in Section \ref{sec:proofs:resampling}.
\begin{prop}
\label{Sampl:prop:uniqueness}
Let $\chi$ and $\chi'$ be $\hat\U$-valued random variables. Let $\rho$ be a $\Uu$-valued random variable with conditional distribution $\alpha(\nu^\chi)$ given $\chi$, and let $\rho'$ be another $\Uu$-valued random variable with conditional distribution $\alpha(\nu^{\chi'})$ given $\chi'$. Then $\rho$ and $\rho'$ are equal in distribution if and only if $\chi$ and $\chi'$ are equal in distribution.
\end{prop}
(In terms of first-moment measures, Proposition \ref{Sampl:prop:uniqueness} says that $\chi$ and $\chi'$ are equal in distribution if and only if $\E[\alpha(\nu^\chi)]=\E[\alpha(\nu^{\chi'})]$.)

The aim of the present paper is the treatment of the case with dust. In the dust-free case, we need not decompose the semi-metric $\rho$ by the map $\beta$. Instead, we can work directly with the map $\psi$ from Subsection \ref{Sampl:sec:mdm-mmm-new}. Theorem \ref{thm:repr} then reduces to the setting of metric measure spaces as follows:
\begin{cor}
\label{cor:repr}
Let $\rho$ be an exchangeable $\Uu$-valued random variable that is a.\,s.\ dust-free. Let $\chi=\psi(\rho)$. Let $\rho'$ be a $\Uu$-valued random variable whose conditional distribution given $\chi$ is $\nu^\chi$. Then:
\begin{enumerate}[label=(\roman{*}),ref=(\roman{*})]
\item\label{item:cor:repr-nd:sampling-psi} $\rho\in\Dd^*$ a.\,s.
\item\label{item:cor:repr-nd:sampl} $\rho$ and $\rho'$ are equal in distribution.
\item\label{item:cor:repr-nd:det} $\chi=\psi(\rho')$ a.\,s.
\end{enumerate}
\end{cor}
\begin{proof}
This is immediate from Theorem \ref{thm:repr} and Remarks \ref{rem:mmm-dust-free}, \ref{rem:sampling-dust-free}, and \ref{Sampl:rem:psi-hatpsi}.
\end{proof}
\begin{rem}
\label{rem:mm-Vershik}
The assertions of Corollary \ref{cor:repr} are closely related to Vershik \cite{Vershik02}: Condition (4) in \cite{Vershik02}*{Theorem 5} is a necessary and sufficient condition for an exchangeable (and ergodic) random semi-metric to have the distance matrix distribution of a metric measure space.
By Remark \ref{rem:sampling-dust-free}, the marked metric measure space $\chi$ in Theorem \ref{thm:repr} is a.\,s.\ dust-free if and only if $\rho$ is a.\,s.\ dust-free. Hence, for a semi-ultrametric $\rho$, condition (4) in \cite{Vershik02} is equivalent to dust-freeness. In the dust-free case, the metric measure space associated with $\chi$ as in Remark \ref{rem:mmm-dust-free} is the completion of a typical realization of the semi-metric, endowed with the probability measure given by the asymptotic block frequencies of the associated coalescent (as in Remark \ref{Sampl:rem:psi-families}). This can also be deduced from \cite{Vershik02}*{Equation (9)}.
Assertion \ref{item:cor:repr-nd:det} can be proved by Proposition \ref{Sampl:prop:mmm-mdm-mmm} below which is related to \cite{Vershik02} as stated in Remark \ref{Sampl:rem:mmm-mdm-mmm}.
\end{rem}

\subsection{Interpretation as ergodic decomposition}
\label{sec:ergodic}
In this subsection, we interpret the representation from Theorem \ref{thm:repr} as the ergodic decomposition of an exchangeable distribution on the semi-ultrametrics on $\N$.

We denote by $\UU$ the space of exchangeable probability distributions on $\Uu$, and we endow $\UU$ with the Prohorov metric $\dP$ which is complete and separable.  We will also consider the subspace\label{Sampl:not:UUerg}
\[\UUerg=\{\xi\in\UU: \xi=\alpha(\nu^{(X,r,m)})\text{ for some marked metric measure space }(X,r,m)\}\]
of distance matrix distributions of marked metric measure spaces. The sets $\UUerg$ and $\hat\U$ are in one-to-one correspondence by Proposition~\ref{Sampl:prop:sampl-unique}. Hence, also the elements of $\UUerg$ can be seen as trees.

We define the invariant sigma algebra $\Ii$ on $\Uu$ as the sigma algebra that is  generated by those Borel sets $B\subset \Uu$ that satisfy $B=\{(\rho(p(i),p(j)))_{i,j\in\N}:\rho\in B\}$ for all finite permutations $p\in S_\infty$.
A distribution $\xi$ on $\Uu$ is called ergodic (with respect to the action of the group $S_\infty$ of finite permutations) if $\xi(I)\in\{0,1\}$ for all $I\in\mathcal I$.

\begin{prop}
\label{prop:mmm-ergodic}
The distance matrix distribution $\alpha(\nu^{(X,r,m)})$ of a marked metric measure space $(X,r,m)$ is invariant and ergodic with respect to the action of the group of finite permutations.
\end{prop}
\begin{proof}
This is analogous to \cite{Vershik02}*{Lemma 7}. For $I\in\mathcal I$, the Borel set $\tilde I\subset (X\times\R_+)^\N$ that given by
\[\tilde I=\{(x(i),v(i))_{i\in\N}\in(X\times\R_+)^\N:((v(i)+r(x(i),x(j))+v(j))\I{i\neq j})_{i,j\in\N}\in I\},\]
is invariant under finite permutations, that is,
\[\tilde I=\{(x(p(i)),v(p(i)))_{i\in\N}:(x,v)\in\tilde I\}\quad\text{for all }p\in S_\infty.\]
From the ergodicity of an $m$-iid sequence $(x(i),v(i))_{i\in\N}$, we obtain
\[\alpha(\nu^{(X,r,m)})(I)=\P((x,v)\in\tilde I)\in\{0,1\}.\]
\end{proof}

\begin{prop}
\label{prop:UUerg}
The subset $\UUerg\subset\UU$ consists of the ergodic distributions.
\end{prop}
\begin{proof}
By Theorem~\ref{thm:repr}\ref{item:thm:repr:sampl}, each element of $\UU$ is a mixture of elements of $\UUerg$.
The assertion follows by Proposition \ref{prop:mmm-ergodic} and as the ergodic distributions in $\UU$ are extreme in the convex set $\UU$ (see e.\,g.\ \cite{Kal05}*{Lemma A1.2}).
\end{proof}
\begin{rem}
\label{rem:erg-dec-map}
Theorem~\ref{thm:repr} decomposes the distribution of the exchangeable $\Uu$-valued random variable $\rho'$ into ergodic components in the sense of e.\,g.\ Theorem A1.4 in Kallenberg \cite{Kal05}. The function
\[\zeta:\Uu\to\UUerg,\quad\tilde\rho\mapsto\alpha(\nu^{\hat\psi\circ\beta(\tilde\rho)})\]
is a decomposition map in the sense of Varadarajan \cite{Var63}*{Section 4} so that typically, $\zeta(\rho')$ is the ergodic component in whose support the realization $\rho'$ lies. Note that this ergodic component is characterized by the isomorphy class $\chi=\hat\psi\circ\beta(\rho')$ of a marked metric measure space, and in the dust-free case also by the isomorphy class $\psi(\rho')$ of a metric measure space. Some further references on the ergodic decomposition are given e.\,g.\ in \cite{Kal05}*{p.\,475}.
\end{rem}

By the following proposition, $(\UUerg,\dP)$ is Polish which will be applied in \cite{Conv}.
\begin{prop}
\label{Sampl:prop:sampl-closed}
The subspace $\UUerg$ is closed in $(\UU,\dP)$.
\end{prop}
\begin{proof}
Let $(\rho^n,n\in\N)$ be a sequence of $\Uu$-valued random variables that converges in distribution to some $\Uu$-valued random variable $\rho$. Assume that for each $n\in\N$, the distribution of $\rho^n$ lies in $\UUerg$. Then $\rho^n$ has ergodic distribution by Proposition \ref{prop:mmm-ergodic}. Lemma 7.35 of \cite{Kal05} says that $\rho^n$ is dissociated, which means that for any disjoint $I_1,\ldots,I_k\subset\N$, the restrictions $(\rho^n(i,j))_{i,j\in I_1}$, $\ldots$, $(\rho^n(i,j))_{i,j\in I_k}$ are independent. As this property is preserved under the limit in distribution, it also holds for $\rho$, and another application of Lemma 7.35 of \cite{Kal05} and yields that $\rho$ has ergodic distribution. The assertion follows by Proposition \ref{prop:UUerg}.
\end{proof}

\section{Application to tree-valued processes}
\label{Sampl:sec:tv}
Using the function $\hat\psi$ from Section~\ref{Sampl:sec:mdm-mmm-new}, we map a Markov process $(\rho_t,t\in\R_+)$ whose states are exchangeable $\Uu$-valued random variables to a process with values in the space of isomorphy classes of marked metric measure spaces. At each time, the state of the image process is the marked metric measure space from the representation (Theorem~\ref{thm:repr}) of the state of the $\Uu$-valued process. We also consider the process of the distance matrix distributions of these marked metric measure spaces. In the dust-free case, we can also work with isomorphy classes of metric measure spaces and the map $\psi$ as in Corollary~\ref{cor:repr}. 

In the proof of Theorem \ref{Sampl:thm:tv} below, we use the criterion of Rogers and Pitman \cite{RP81}*{Theorem 2} to show that also the image processes are Markovian. A martingale problem for the $\Uu$-valued process $(\rho_t,t\in\R_+)$ or the $\hat\Uu$-valued process $(\beta(\rho_t),t\in\R_+)$ yields a martingale problem for the respective image process.

The so-called polynomials and marked polynomials, introduced in \cites{GPW09,DGP11} have been used as domains of martingale problems in e.\,g.\ \cites{GPW13,DGP12,DGP13}.
We recall them here, adapting the definition to our present
use of the marks.
The uniform continuity of the derivative in the definitions of $\C_n$ and $\hat\C_n$ below will turn out useful in \cite{Conv}.
For $n\in\N$, we write $[n]=\{1,\ldots,n\}$ for $n\in\N$, and we denote by $\gamma_n$\label{Sampl:not:gamma_n-matr} the restriction from $\R^{\N^2}\times\R^\N$ to $\R^{n^2}\times\R^n$, $\gamma_n(r,v)=((r(i,j))_{i,j\in[n]},(v(i))_{i\in[n]})$.
We denote also by $\gamma_n$\label{Sampl:not:gamma_n-mm} the restriction from $\R^{\N^2}$ to $\R^{n^2}$, $\gamma_n(\rho)=(\rho(i,j))_{i,j\in[n]}$. Let $\C_n$\label{Sampl:not:Cn} denote the set of bounded differentiable functions $\R^{n^2}\to\R$ with bounded uniformly continuous derivative. For $\phi\in\C_n$, we denote also by $\phi$ the function $\phi\circ\gamma_n:\R^{\N^2}\to\R$, and we call the function $\U\to\R$, $\chi\mapsto\nu^\chi\phi$ the polynomial associated with $\phi$. (Here and at other places, we use the notation $\xi f=\int \xi(dx)f(x)$\label{not:integral} for a measure $\xi$ and an integrable function $f$, and we view measures also as functionals on spaces of integrable functions.)
Similarly, we denote by $\hat\C_n$ the set of bounded differentiable functions $\R^{n^2}\times\R^n\to\R$ with uniformly continuous derivative. For $\phi\in\hat\C_n$, we denote also by $\phi$ the function $\phi\circ\gamma_n:\R^{\N^2}\times\R^\N\to\R$, and we call the function $\hat\U\to\R$, $\chi\mapsto\nu^\chi\phi$ the marked polynomial associated with $\phi$. (Usually, the argument $(r,v)$ of a function $\phi\in\hat\C$ will be a marked distance matrix.)
We write $\C=\bigcup_{n\in\N}\C_n$ and $\hat\C=\bigcup_n\hat\C_n$. We denote the set of polynomials by
\[\label{Sampl:not:Pi}\Pi=\{\U\to\R,\chi\mapsto\nu^\chi\phi:\phi\in\C\},\]
the set of marked polynomials by
\[\label{Sampl:not:hatPi}\hat\Pi=\{\hat\U\to\R,\chi\mapsto\nu^\chi\phi:\phi\in\hat\C\},\]
and we define the set of test functions
\[\label{Sampl:not:Cc}\Cc=\{\UUerg\to\R,\xi\mapsto\xi\phi:\phi\in\C\}.\]

For a metric space $E$, let $\Mb(E)$ denote the set of bounded measurable functions $E\to\R$. For a subset $\D\subset\Mb(E)$ and an operator $G:\D\to\Mb(E)$, we mean by a solution of the martingale problem $(G,\D)$ a progressive $E$-valued process $(X_t,t\in\R_+)$ such that for every $f\in\D$, the process
\[f(X_t)-\int_0^tGf(X_s)ds\]
is a martingale with respect to the filtration induced by $(X_t,t\in\R_+)$, cf.\ Ethier and Kurtz \cite{EK86}*{p.\,173}.

\begin{thm}
\label{Sampl:thm:tv}
Let $(\rho_t,t\in\R_+)$ be a $\Uu$-valued time-homogenous Markov process. Assume that for each $t\in\R_+$, the random variable $\rho_t$ is exchangeable.
Let $A:\C\to\Mb(\R^{\N^2})$ and $\hat A:\hat\C\to\Mb(\R^{\N^2}\times\R^{\N})$ be operators. Define the $\U$-valued process $(\chi_t,t\in\R_+):=(\psi(\rho_t),t\in\R_+)$, the $\hat\U$-valued process $(\hat\chi_t,t\in\R_+):=(\hat\psi(\beta(\rho_t)),t\in\R_+)$, and the $\UUerg$-valued process $(\xi_t,t\in\R_+):=(\alpha(\nu^{\hat\chi_t}),t\in\R_+)$. Then the following two assertions hold:
\begin{enumerate}[label=(\roman{*}),ref=(\roman{*})]
\item\label{Sampl:item:thm:tv:mmm} The process $(\hat\chi_t,t\in\R_+)$ is Markovian. If the $\hat\Uu$-valued process $(\beta(\rho_t),t\in\R_+)$ solves the martingale problem $(\hat A,\hat\C)$, then $(\hat\chi_t,t\in\R_+)$ solves the martingale problem $(\hat B,\hat\Pi)$, given by
\[\hat B\Phi(\chi)=\nu^\chi(\hat A\phi)\]
for all $\phi\in\hat\C$ with associated polynomial $\Phi$, and all $\chi\in\hat\U$.
\item\label{Sampl:item:thm:tv:dmd}
The process $(\xi_t,t\in\R_+)$ is Markovian. If  $(\rho_t,t\in\R_+)$ solves the martingale problem $(A,\C)$, then $(\xi_t,t\in\R_+)$ solves the martingale problem $(C,\Cc)$, given by
\[C\Psi(\xi)=\xi(A\phi)\]
for all $\xi\in\UUerg$ and $\phi\in\C$, and the function $\Psi\in\Cc$, $\xi'\mapsto\xi'\Psi$.
\end{enumerate}
Assertion~\ref{Sampl:item:thm:tv:mm} below holds under the additional assumption that $\rho_t$ is a.\,s.\ dust-free for each $t\in\R_+$.
\begin{enumerate}[label=(\roman{*}),ref=(\roman{*}),resume]
\item\label{Sampl:item:thm:tv:mm} The process $(\chi_t,t\in\R_+)$ is Markovian. If $(\rho_t,t\in\R_+)$ solves the martingale problem $(A,\C)$, then $(\chi_t,t\in\R_+)$ solves the martingale problem $(B,\Pi)$, given by
\[B\Phi(\chi)=\nu^\chi(A\phi)\]
for all $\phi\in\C$ with associated polynomial $\Phi$, and all $\chi\in\U$.
\end{enumerate}
\end{thm}
The proof of Theorem \ref{Sampl:thm:tv} can be found in Section \ref{sec:proof:tv}.
\begin{rem}
\label{rem:Markov-beta}
The process $(\beta(\rho_t),t\in\R_+)$ in Theorem~\ref{Sampl:thm:tv} is Markov. This follows as $(\rho_t,t\in\R_+)$ is Markov by assumption and as $\rho_t$ is determined by $\beta(\rho_t)$ via $\rho_t=\alpha(\beta(\rho_t))$ so that
\[\E[f(\beta(\rho_u))|(\beta(\rho_s))_{s\leq t}]
=\E[f(\beta(\rho_u))|(\rho_s)_{s\leq t}]
=\E[f(\beta(\rho_u))|\rho_t]
=\E[f(\beta(\rho_u))|\beta(\rho_t))]\quad\text{a.\,s.}\]
for all $s\leq t\leq u$ and bounded measurable $f:\hat\Uu\to\R$. This is an example for Dynkin's criterion \cite{Dynkin}*{Theorem 10.13} for a function of a Markov process to be Markov.
\end{rem}
\begin{rem}
\label{Sampl:rem:tv-df}
In Theorem~\ref{Sampl:thm:tv}, if $\rho_t$ is dust-free for some $t\in\R_+$, then $\hat\chi_t$ is (by Theorem \ref{thm:repr} and Remark \ref{rem:sampling-dust-free} the isomorphy class of a) dust-free marked metric measure space, $\chi_t$ is the (isomorphy class of the) metric measure space associated (as in Remark \ref{rem:mmm-dust-free}) with (any representative of) $\hat\chi_t$, and we have $\xi_t=\nu^{\chi_t}$.
The process $(\chi_t,t\in\R_+)$ is relevant only in the dust-free case: If $\rho_t$ is not dust-free, then $\psi(\rho_t)$ is just the arbitrary element of $\M$ from the definition of $\psi$ in Section~\ref{Sampl:sec:mdm-mmm-new}.
\end{rem}
\begin{rem}
\label{rem:version}
In Theorem \ref{Sampl:thm:tv}, we characterize only versions of the processes $(\hat\chi_t,t\in\R_+)$, $(\chi_t,t\in\R_+)$, and $(\xi_t,t\in\R_+)$. That is, we do not make assertions on the full sample paths but only on the states at countably many times.

From Theorem~\ref{thm:repr}, we obtain $\beta(\rho_t)\in\hat\Dd^*$ (and in the dust-free case also $\rho_t\in\Dd^*$ by Corollary~\ref{cor:repr}) only for a fixed time $t$ (or countably many $t$) on an event of probability $1$.
This means that the uniform probability measures on the starting vertices of the external branches that end in the first $n$ leaves of the tree associated with the semi-ultrametric $\rho_t$ are shown to converge only at countably many times $t$ on an event of probability $1$. For $\beta(\rho_t)\in\hat\Dd^*$, a realization $\hat\chi_t=\hat\psi(\beta(\rho_t))$ can be considered as an ergodic component. 
At the other times $t$, we do not exclude that $\hat\chi_t=\hat\psi(\beta(\rho_t))$ is just the arbitrary element of $\hat\M$ with probability measure $\delta_{(1,0)}$ in the definition of $\hat\psi$ in Section \ref{Sampl:sec:mdm-mmm-new}.

Theorem \ref{Sampl:thm:tv} yields in particular the semigroups of the processes $(\hat\chi_t,t\in\R_+)$, $(\chi_t,t\in\R_+)$, and $(\xi_t,t\in\R_+)$. Also the martingale problems in Theorem \ref{Sampl:thm:tv} characterize only versions of these processes.

For the particular example of the process $(\rho_t,t\in\R_+)$ in Sections \ref{Sampl:sec:ld-gen} --  \ref{Sampl:sec:equil}, it is shown in \cite{Pathw} that $\beta(\rho_t)\in\hat\Dd^*$ (and $\rho_t\in\Dd^*$ in the dust-free case) also holds simultaneously for all $t\in\R_+$ on an event of probability $1$ (see Theorems 3.1(i) and 3.10(i), and Remarks 4.4 and 4.13 in \cite{Pathw}). This allows to construct the full sample paths (Section 4 in \cite{Pathw}). These results are obtained in \cite{Pathw} by techniques specific to the lookdown model.
\end{rem}
\begin{rem}
Theorem~\ref{Sampl:thm:tv} is an example for Markov mapping. To show that the image processes $(\hat\psi(\beta(\rho_t)),t\in\R_+)$, $(\xi_t,t\in\R_+)$, and $(\psi(\rho_t),t\in\R_+)$ are Markovian, we use the simple criterion of Rogers and Pitman \cite{RP81}*{Theorem 2} as this criterion is formulated in terms of the abstract semigroups of the processes, which fits to our assumption that $(\rho_t,t\in\R_+)$ is a general time-homogenous Markov process whose states $\rho_t$ are exchangeable.

A criterion for the Markov property of the image processes in terms of martingale problems is given in Corollary 3.5 of Kurtz \cite{Kurtz98} which requires more assumptions, including uniqueness for the martingale problem for $(\rho_t,t\in\R_+)$ and existence of solutions of the martingale problems for the image processes.
Corollary 3.5 of \cite{Kurtz98} would also yield uniqueness for the martingale problems for the image processes.

In the present paper, we use martingale problems only to provide additional characterizations of the processes under consideration. In Proposition  \ref{Sampl:prop:mp-unique-mm}, we show uniqueness for the martingale problems for the image processes directly by duality for the concrete examples from Section \ref{Sampl:sec:TVFV}.
\end{rem}
\begin{rem}
\label{rem:conv-det}
In particular in Sections \ref{Sampl:sec:semigroup} -- \ref{Sampl:sec:equil}, \ref{sec:proof-uniqueness} and in \cite{Conv}, we need convergence determining (or at least separating) sets of test functions.
As in \cites{Lohr13,GPW09,DGP11}, the sets $\Pi$ and $\hat\Pi$ are convergence determining in $\U$ and $\hat\U$, respectively.
The argument from \cite{Lohr13}*{Corollary 2.8} also applies for $\Cc$: The algebra $\C$ generates the product topology on $\R^{\N^2}$.
By a theorem due to Le Cam, see e.\,g.\ \cite{Lohr13}*{Theorem 2.7} and the references therein, it follows that $\C$ is convergence determining in $\Uu$. Hence, $\Cc$ generates the weak topology on $\UUerg$. As $\hat\Pi$ is an algebra (see \cites{GPW09,DGP11})
and by definition of $\UUerg$, also $\Cc$ is an algebra.
Again by \cite{Lohr13}*{Theorem 2.7}, it follows that $\Cc$ is convergence determining in $\UUerg$.
\end{rem}
\begin{rem}
The set of polynomials
$\Pi'=\{\hat\U\to\R,\chi\mapsto\alpha(\nu^\chi)\phi:\phi\in\C\}$
is separating on $\hat\U$. This follows from Propositions~\ref{Sampl:prop:sampl-unique} and \ref{Sampl:prop:mmm-mdm-mmm} as in the proof of Proposition~\ref{Sampl:prop:uniqueness}. Nevertheless, we work with the space $\hat\Pi$ of test functions on $\hat\M$ as $\Pi'$ is not convergence determining, a counterexample can be constructed from \cite{GPW09}*{Example 2.12(ii)}.
\end{rem}

\section{Genealogy in the lookdown model}
\label{Sampl:sec:ld-gen}
In this section, we define a Markov process $(\rho_t,t\in\R_+)$ to which we will later apply Theorem~\ref{Sampl:thm:tv}. In Subsection \ref{sec:det-ld-gen}, we read off a realization of such a process from a population model that is driven by a deterministic point measure $ \eta$. In Subsection \ref{Sampl:sec:Xi-ld}, we let $\eta$ be a Poisson random measure, and we study further properties of $(\rho_t,t\in\R_+)$ in Subsection \ref{sec:ld-gen-t}.
We remark that for the lookdown model of Donnelly and Kurtz \cite{DK96}, the process of the evolving genealogical distances and its martingale problem are considered in Remark 2.20 of Greven, Pfaffelhuber, and Winter \cite{GPW13}.

\subsection{The deterministic construction}
\label{sec:det-ld-gen}
We denote by $\p$ the set of partitions of $\N$. We endow $\p$ with the topology in which a sequence of partitions converges if and only if the sequences of their finite restrictions converge.
For $n\in\N$, we denote by $\p_n$ the set of partitions of $[n]=\{1,\ldots,n\}$.
We denote the restriction map\label{Sampl:not:gamma_n-p} from $\p$ to $\p_n$ by $\gamma_n$, that is, $\gamma_n(\pi)=\{B\cap[n]:B\in\pi\}\setminus\{\emptyset\}$. Recall that other restriction maps, e.\,g.\ from $\R^{\N^2}\to\R^{n^2}$ are also denoted by $\gamma_n$. 
Moreover, we denote by $\mathbf{0}_n=\{\{1\},\ldots,\{n\}\}$ the partition in $\p_n$ that consists of singletons only, and by
$\p^n=\{\pi\in\p:\gamma_n(\pi)\neq\mathbf{0}_n\}$\label{Sampl:not:p^n}
the set of partitions of $\N$ in which the first $n$ integers are not all in different blocks.
Furthermore, for $\pi\in\p$,
we denote by $B_1(\pi),B_2(\pi),\ldots$\label{Sampl:not:Bi} the enumeration of the blocks of $\pi$ with $\min B_1(\pi)<\min B_2(\pi)<\ldots$.
For $i\in\N$, we denote by $\pi(i)$ the integer $j$ that satisfies $i\in B_j(\pi)$.

We use a lookdown model as the population model. In this model, there are countably infinitely many levels which are labeled by $\N$, and each level is occupied by one particle at each time $t\in\R_+$. The particles undergo reproduction events which are encoded by a simple point measure $\eta$ on $(0,\infty)\times\p$\label{Sampl:not:eta-det}. A simple point measure is a purely atomic measure whose atoms all have mass $1$.  Let us impose a further assumption on $\eta$, namely
\begin{equation}
\label{Sampl:eq:ass-eta}
\eta((0,t]\times \p^n)<\infty\quad\text{for all }t\in(0,\infty)\text{ and }n\in\N.
\end{equation}

The interpretation of a point $(t,\pi)$ of $\eta$ is that the following reproduction event occurs: At time $t-$, the particles on the levels $i\in\N$ with $i>\#\pi$ are removed. At time $t$, for each $i\in[\#\pi]$, the particle that was on level $i$ at time $t-$ assumes level $\min B_i(\pi)$ and has offspring on all other levels in $B_i(\pi)$. Thus, the level of a particle is non-decreasing as time evolves. Condition~\reff{Sampl:eq:ass-eta} means that for each $n\in\N$, only finitely many particles jump away from the first $n$ levels in bounded time intervals.

For all $0\leq s\leq t$, each particle at time $t$ has an ancestor at time $s$. We denote by $A_s(t,i)$\label{Sampl:not:A} the level of the ancestor at time $s$ of the particle on level $i$ at time $t$ such that the maps $s\mapsto A_s(t,i)$ and $t\mapsto A_s(t,i)$ are càdlàg. Then $A_s(t,i)$ is well-defined as
$s\mapsto A_{t-s}(t,i)$ is non-increasing.
\begin{rem}
\label{Sampl:rem:non-crossing}
We will use that the trajectories of the particles are non-crossing in the following sense:
For any times $s\leq t$ and particles $x,y$ on levels $i_x\leq i_y$ at time $s\in\R_+$, particle $x$ is still alive if particle $y$ is still alive, in which case the particles $x$ and $y$ occupy levels $j_x\leq j_y$.
In particular, if infinitely many particles at time $s$ survive until time $t$, then all particles at time $s$ survive until time $t$.
\end{rem}

We are interested in the process of the genealogical distances between the particles that live at the respective times. Let $\rho_0\in\R^{\N^2}$.
(We can assume $\rho_0\in\Uu$ here, but differentiability will be more elementary in the larger space, as a matter of taste.)
We interpret $\rho_0(i,j)$ as the genealogical distance between the particles on levels $i$ and $j$ at time $0$. We define the genealogical distance between the particles on levels $i$ and $j$ at time $t$ by
\[\label{Sampl:not:rhot}\rho_t(i,j)=\left\{\begin{aligned}
&2t-2\sup\{s\in[0,t]:A_s(t,i)=A_s(t,j)\}\quad\text{if }A_0(t,i)=A_0(t,j)\\
&2t+\rho_0(A_0(t,i),A_0(t,j))\quad\text{else.}
\end{aligned}\right.\]
In words, the genealogical distance between two particles at a fixed time is twice the time back to their most recent common ancestor, if such an ancestor exists, else it is given by the genealogical distance between the ancestors at time zero.
\begin{rem}
\label{Sampl:rem:rho-um}
If $\rho_0\in\Uu$, then $\rho_t\in\Uu$ for each $t\in\R_+$. Indeed, a semi-metric $\rho$ on $\N$ is a semi-ultrametric if and only if for each $s\in\R_+$, an equivalence relation $\sim$ on $\N$ is given by $i\sim j:\Leftrightarrow \rho(i,j)\leq s$. If this property holds for $\rho_0$, then the definition of $\rho_t$ readily yields that it also holds for $\rho_t$.
\end{rem}

We also describe the process $(\rho_t, t\in\R_+)$ in a more formal way which will be useful for the description by martingale problems in Section~\ref{Sampl:sec:Xi-ld}.
With each partition $\pi\in\p_n$ we associate a transformation $\R^{n^2}\to\R^{n^2}$, which we also denote by $\pi$\label{Sampl:not:pn-transf}, by
\begin{equation}
\label{Sampl:eq:pn-Un}
\pi(\rho)=(\rho(\pi(i),\pi(j)))_{i,j\in[n]}.
\end{equation}
Here $\pi(i)$ denotes the integer $k$ such that $i$ is in the $k$-th block, when blocks are ordered according to their minimal elements.
Note that for each reproduction event encoded by a point $(s,\pi)\in\eta$, the corresponding jump of the process $(\rho_t,t\in\R_+)$ can be described by
\begin{equation}
\label{Sampl:eq:jumps-pi-rhot}
\gamma_n(\pi)(\gamma_n(\rho_{s-}))=\gamma_n(\rho_s).
\end{equation}
In particular, $\gamma_n(\pi)=\mathbf{0}_n$ if $\pi\in\p\setminus\p^n$, and $\mathbf{0}_n$ acts as the identity on $\R^{n^2}$.
By assumption~\reff{Sampl:eq:ass-eta}, there are only finitely many reproduction events in bounded time intervals that result in a jump of the process $(\gamma_n(\rho_t(i,j)),t\in\R_+)$.
Between such jumps, the genealogical distances grow linearly with slope $2$, that is,
$\rho_t(i,j)+2s=\rho_{t+s}(i,j)$ for distinct $i,j\in[n]$ and $t,s\in\R_+$ with $\eta((t,t+s]\times\p^n)=0$.

\begin{rem}
\label{Sampl:rem:diff-ld}
Schweinsberg \cite{Schw00} constructs the $\Xi$-coalescent analogously from a point measure. The population model described in this section can be seen as the population model that underlies the dual flow of partitions in Foucart \cite{Foucart12}.  
A lookdown model with a reproduction mechanism that is different in the case with simultaneous multiple reproduction events is studied by Birkner et al.\ \cite{BBMST09}.
In this model, a partition $\pi\in\p$ encodes the following reproduction event: Let $i_1<i_2<\ldots$ be the increasing enumeration of the integers that either form singletons or are non-minimal elements of blocks of $\pi$. For each $j\in\N$, the particle on level $i_j$ moves to the level given by the $j$-th lowest singleton of $\pi$ if $\pi$ has at least $j$ singletons, else the particle is removed. For each non-singleton block $B\in\pi$, the particle on level $\min B$ remains on its level and has one offspring on each level in $B\setminus\{\min B\}$. Here the trajectories of the particles may cross: Consider a partition $\pi\in\p$ such that $1$ and $2$ are in the same block, $4$ forms a singleton, and $3$ is the minimal element of a non-singleton block. If the reproduction event encoded by $\pi$ occurs at time $t\in(0,\infty)$, then there exists $s\in(0,t)$ such that the particle on level $3$ at time $s$ is on level $3$ also at time $t$, and the particle on level $2$ at time $s$ jumps to level $4$ at time $t$. Such a crossing cannot occur in our population model by Remark~\ref{Sampl:rem:non-crossing}.
\end{rem}

\subsection{The \texorpdfstring{$\Xi$}{Xi}-lookdown model}
\label{Sampl:sec:Xi-ld}
The population model from the Subsection \ref{sec:det-ld-gen} will now be driven by a Poisson random measure $\eta$ on $(0,\infty)\times\p$ as in Schweinsberg \cite{Schw00}, Bertoin \cite{Bertoin}, and Foucart \cite{Foucart12}.

To define this Poisson random measure, we briefly recall Kingman's correspondence. For a full account, see e.\,g.\ \cite{Bertoin}*{Section 2.3.2}.
Kingman's correspondence is a one-to-one correspondence between the distributions of the exchangeable random partitions of $\N$ and the probability measures on the simplex
\[\Delta=\{x=(x_1,x_2,\ldots): x_1\geq x_2\geq \ldots\geq 0,\left|x\right|_1\leq 1\},\]
where $|x|_1=\sum_{i\in\N} x_i$\label{not:1-norm}. Every $x\in\Delta$ can be interpreted as a partition of $[0,1]$ into subintervals of lengths $x_1, x_2, \ldots$, and possibly another interval of length $1-\left|x\right|_1$ which may be called the dust interval. Let $U_1,U_2,\ldots$ be iid uniform random variables with values in $[0,1]$. The paintbox partition associated with $x$ is the exchangeable random partition of $\N$ where two different integers $i$ and $j$ are in the same block if and only if $U_i$ and $U_j$ fall into a common subinterval that is not the dust interval. This construction defines a probability kernel $\kappa$\label{Sampl:not:kappa} from $\Delta$ to $\p$. Conversely, every exchangeable random partition $\pi$ in $\p$ has distribution $\int_\Delta \xi(dx)\kappa(x,\cdot)$ for some distribution $\xi$ on $\Delta$. Here $x$ is the random vector in $\Delta$ of the asymptotic frequencies of the blocks of $\pi$.

Let $\Xi$ be a finite measure on $\Delta$. We decompose
\begin{equation}
\label{Sampl:eq:dec-Xi}
\Xi=\Xi_0+\Xi\{0\}\delta_0.
\end{equation}
For $i,j\in\N$ with $i\neq j$, we denote by $K_{i,j}$\label{Sampl:not:Kij} the partition in $\p$ that contains the block $\{i,j\}$ and apart from that only singleton blocks. We define a $\sigma$-finite measure $H_\Xi$ on $\p$ by
\begin{equation*}
\label{Sampl:not:HXi}
H_\Xi(d\pi)=\int_\Delta\kappa(x,d\pi)\left|x\right|_2^{-2} \Xi_0(dx) +  \Xi\{0\}\sum_{1\leq i<j}\delta_{K_{i,j}}(d\pi),
\end{equation*}
where $|x|_2=\left(\sum_{i\in\N}x_i^2\right)^{1/2}$.\label{not:2-norm}

Let $\eta$ be a Poisson random measure on $(0,\infty)\times\p$ with intensity $dt\;H_\Xi(d\pi)$\label{Sampl:not:eta-Poisson}.
Note that $\kappa(x,\p^n)\leq\binom{n}{2}|x|_2^2$ for all $x\in\Delta$ and $n\geq 2$. This follows as in the paintbox partition associated with $x$, the probability that two fixed integers belong to the same block is $|x|_2^2$. The random point measure $\eta$ thus satisfies condition \reff{Sampl:eq:ass-eta} a.\,s.\ as
\begin{equation}
\E[\eta((0,t]\times\p^n)]=\int\kappa(x,\p^n)|x|_2^{-2}\Xi_0(dx)
+\Xi\{0\}\binom{n}{2}<\infty
\end{equation}
for all $t\in\R_+$ and $n\in\N$. 
Hence, we can and will define the population model from Subsection \ref{sec:det-ld-gen} from almost every realization of $\eta$ and every $\rho_0\in\R^{\N^2}$.
We also let $\rho_0$ be a $\R^{\N^2}$-valued random variable that is independent of $\eta$.
We define the $\R^{\N^2}$-valued process $(\rho_t,t\in\R_+)$ realization-wise from the Poisson random measure $\eta$ and the random initial state $\rho_0$ as in the preceding subsection.

\begin{prop}
\label{prop:rho-Markov}
The process $(\rho_t,t\in\R_+)$ is Markov.
\end{prop}
\begin{proof}
The description around equation~\reff{Sampl:eq:jumps-pi-rhot} implies that for $0\leq t<t'$ and each $n\in\N$, the conditional expectation of $\gamma_n(\rho_{t'})$ given $(\rho_s,s\leq t)$ is measurable with respect to $\rho_t$ and the restriction of $\eta$ to $(t,t']\times\p$. The assertion follows as $n\in\N$ was arbitrary and as the restrictions of a Poisson random measure to disjoint subsets are independent.
\end{proof}

For each $n\in\N$ and $\pi\in\p_n\setminus\{\mathbf{0}_n\}$, the rate at which reproduction events encoded by partitions in
$\gamma_n^{-1}(\pi)=\{\pi'\in\p:\gamma_n(\pi')=\pi\}$
occur in the lookdown model is given by $\lambda_\pi=H_\Xi(\gamma_n^{-1}(\pi))$. The rates $\lambda_\pi$ are calculated explicitly in~\reff{Sampl:eq:lambda-P-S} and \reff{Sampl:eq:lambda-S} in Section~\ref{Sampl:sec:dec-stoch}.
\begin{rem}
\label{Sampl:rem:rates-Schw}
The quantity $\lambda_\pi$ is the coagulation rate $q_\pi$ in Section 4.2.1 of Bertoin \cite{Bertoin}.
It is related to the quantity $\lambda_{n;k_1,\ldots,k_r;s}$ from Schweinsberg \cite{Schw00} by $\lambda_\pi=\lambda_{n;k_1,\ldots,k_r;s}$, where $k_1,\ldots,k_r$ denote the sizes of the non-singleton blocks of $\pi$, and $s=n-k_1-\ldots-k_r$. This can be seen by a comparison of equations~\reff{Sampl:eq:lambda-P-S} and \reff{Sampl:eq:lambda-S} with equation (11) in \cite{Schw00}.
In particular, equation (18) in \cite{Schw00} implies that $\eta$ satisfies a.\,s.\ condition~\reff{Sampl:eq:ass-eta}.
\end{rem}

In the next proposition, we state a martingale problem for the process $(\rho_t,t\in\R_+)$.

Recall the set $\C$ from Section~\ref{Sampl:sec:tv}. For $\phi\in\C$ and $\rho\in\R^{\N^2}$, we write
\begin{equation}
\label{eq:nabla-growth}
\langle \nabla \phi, \underline{\underline 2} \rangle (\rho)= 2\sum_{\substack{i,j\in\N\\ i\neq j}} \frac{\partial}{\partial \rho(i,j)}\phi(\rho).
\end{equation}

\begin{prop}
\label{Sampl:prop:mp-rho}
Define an operator $A=A_{\rm grow}+A_{\rm repr}$ with domain $\C$ by
\[A_{\rm grow}\phi(\rho)=\langle \nabla \phi, \underline{\underline 2 }\rangle(\rho)\]
and
\[A_{\rm repr}\phi(\rho)
=\sum_{\pi\in\p_n\setminus\{\mathbf{0}_n\}}\lambda_{\pi}(\phi(\pi(\gamma_n(\rho)))-\phi(\rho))\]
for $n\in\N$, $\phi\in\C_n$, and $\rho\in\R^{\N^2}$.
Then the stochastic process
$(\rho_t,t\in\R_+)$
solves the martingale problem $(A,\C)$.
\end{prop}
Proposition~\ref{Sampl:prop:mp-rho} follows from the discussion above and the description of the process $(\gamma_n(\rho_t),t\in\R_+)$ around equation~\reff{Sampl:eq:jumps-pi-rhot}.
As in \cite{GPW13}, the operator $A_{\rm grow}$ reflects the growth of the genealogical distances between reproduction events that affects them. The operator $A_{\rm repr}$ stands for the jumps of the genealogical distances in reproduction events, as described by equation~\reff{Sampl:eq:jumps-pi-rhot}.
We omit a formal proof of Proposition~\ref{Sampl:prop:mp-rho}.
\begin{rem}
That the solutions of the martingale problems in Proposition \ref{Sampl:prop:mp-rho} and in Proposition \ref{Sampl:prop:mp-R} below are unique can be shown by the approach from Section \ref{sec:proof-uniqueness}. We do not use this assertion in the present paper.
\end{rem}

\subsection{Properties of the genealogy at a fixed time}
\label{sec:ld-gen-t}
We consider the process $(\rho_t,t\in\R_+)$ from Subsection \ref{Sampl:sec:Xi-ld}. To apply Theorem~\ref{Sampl:thm:tv}, we need exchangeability of the random variable $\rho_t$ for each $t\in\R_+$.
\begin{prop}
\label{Sampl:prop:rho-exch}
Let $t\in\R_+$ and assume that $\rho_0$ is exchangeable. Then $\rho_t$ is exchangeable.
\end{prop}
We prove Proposition \ref{Sampl:prop:rho-exch} in Section \ref{sec:proof:exch}.
\begin{rem}
\label{Sampl:rem:Xi-coal}
For $t\in\R_+$, let $(\Pi^{(t)}_s,s\in[0,t])$ be the $\p$-valued stochastic process such that two integers $i,j\in\N$ are in the same block of $\Pi^{(t)}_s$ if and only if $\rho_t(i,j)\leq 2s$.
Then a comparison of the Poisson process construction of the $\Xi$-coalescent in \cite{Schw00}*{Section 3} with the Poisson process construction from the present section shows that a $\Xi$-coalescent up to time $t$ is given by the process $(\Pi^{(t)}_s,s\in[0,t))$.
The distance matrix $\rho_t\wedge (2t)$ can be retrieved from $(\Pi^{(t)}_s,s\in[0,t))$ by
\begin{equation*}
\rho_t(i,j)\wedge(2t)=
2\inf\{s\in[0,t]:\text{ $i$ and $j$ are in the same block of $\Pi^{(t)}_s$, or $s=t$}\}
\end{equation*}
As $\Xi$-coalescents are exchangeable, it follows that the random variable $(\rho_t(i,j)\wedge(2t))_{i,j\in\N}$ is exchangeable.
We remark that the collection of partitions $(\Pi^{(t)}_{(t-s)-},0\leq s\leq t)$ is the dual flow of partitions from Foucart \cite{Foucart12} in one-sided time.
We also remark that preservation of exchangeability
in the lookdown model is studied in e.\,g.\ \cites{DK96,DK99,BBMST09}.
\end{rem}

For the application of Theorem~\ref{Sampl:thm:tv}, it is also of interest whether the states $\rho_t$ are a.\,s.\ dust-free.
Proposition~\ref{Sampl:prop:rho-df} formulates the criterion from \cite{Schw00}*{Proposition 30} in our present context. We call the finite measure $\Xi$ on $\Delta$ dust-free if
\begin{equation}
\label{Sampl:eq:dust-free}
\Xi\{0\}>0\quad\text{or}\quad\int\left|x\right|_1\left|x\right|_2^{-2}\Xi_0(dx)=\infty.
\end{equation}
\begin{prop}
\label{Sampl:prop:rho-df}
Let $t\in(0,\infty)$ and assume $\rho_0\in\Uu$. Then $\Xi$ is dust-free if and only if $\rho_t$ is a.\,s.\ dust-free.
\end{prop}
\begin{proof}
By Remark~\ref{Sampl:rem:rho-um}, $\rho_t\in\Uu$, hence $\Upsilon(\rho)$ is well-defined. Clearly, $\rho_t$ is dust-free if and only if the partition $\Pi^{(t)}_s$ from Remark~\ref{Sampl:rem:Xi-coal} contains no singletons for all $s\in(0,t)\cap\Q$. This holds a.\,s.\ if and only if $\Xi$ is dust-free by \cite{Schw00}*{Proposition 30}.
\end{proof}

\section{Decomposition of the genealogical distances}
\label{Sampl:sec:dec-proc}
To apply Theorem~\ref{Sampl:thm:tv}\ref{Sampl:item:thm:tv:mmm} to the process $(\rho_t,t\in\R_+)$ from Section \ref{Sampl:sec:Xi-ld}, we need to describe the $\hat\Uu$-valued process $(\beta(\rho_t),t\in\R_+)$ by a martingale problem.
A version of this process that readily yields a description by a martingale problem is read off from the lookdown model in this section.
We define such a process in Subsection~\ref{Sampl:sec:dec-det} for a deterministic point measure $\eta$ that drives the population model. In Subsection~\ref{Sampl:sec:dec-stoch}, we let $\eta$ again be the Poisson random measure.

\subsection{The deterministic construction}
\label{Sampl:sec:dec-det}
Let $\eta$ be a simple point measure on $(0,\infty)\times\p$ (as in Section \ref{sec:det-ld-gen}).
Let $(r_0,v_0)\in\R^{\N^2}\times\R^\N$. We interpret $(r_0,v_0)$ as a decomposition of genealogical distances at time $0$. For $i\in\N$, let
\[\p(i)=\{\pi\in\p:\{i\}\notin \pi\}\]
be the set of partitions of $\N$ in which $i$ does not form a singleton block. If $\eta(\{s\}\times\p(A_s(t,i)))>0$ for some $s\in(0,t]$, then we set
\[v_t(i)=t-\sup\{s\in(0,t]: \eta(\{s\}\times\p(A_s(t,i)))>0\},\]
else we set
\[v_t(i)=t+v_0(A_0(t,i)).\]
The quantity $v_t(i)$ is the time back until an ancestor of the particle on level $i$ at time $t$ is involved in a reproduction event in which it belongs to a non-singleton block, if there is such an event, else $v_t(i)$ is defined from $v_0$.

We let $\rho_0=\alpha(r_0,v_0)$ and define the process $(\rho_t,t\in\R_+)$ from $\eta$ and $\rho_0$ as in Subsection \ref{sec:det-ld-gen}. We set
\[\label{Sampl:not:rtvt}r_t(i,j)=(\rho_t(i,j)-v_t(i)-v_t(j))\I{i\neq j}\]
for $t\in\R_+$ and $i,j\in\N$.
Then $(r_t,v_t)$ can be thought of as a decomposition of the distance matrix $\rho_t$ in the sense of Section~\ref{Sampl:sec:dist-dec}. In this decomposition, we remove from the genealogical tree at time $t$ the part between any leaf $i$ and the most recent reproduction event on the ancestral lineage of this leaf, and we encode the length of this part as the mark $v_t(i)$.
\begin{rem}
Consider for this remark the following change (compared to our definition from Section \ref{sec:det-ld-gen}) in the definition of the reproduction event encoded by a point $(t,\pi)\in\eta$: For each non-singleton block $B_i(\pi)$, the reproducing particle on level $i$ at time $t-$ dies and is replaced at time $t$ by its offspring on all the levels in $B_i(\pi)$. Then the quantity $v_t(i)$ is the age of the particle on level $i$ at time $t$ if this holds for $t=0$.
Condition~\reff{Sampl:eq:ass-eta-hat} below ensures that the times at which the particles on a fixed level are replaced do not accumulate.
\end{rem}

Analogously to Section~\ref{sec:det-ld-gen}, we give another description of the process $((r_t,v_t),t\in\R_+)$.
Let $\S_n$\label{Sampl:not:Sn} be the set of semi-partitions of $[n]$, that is, the set of systems of nonempty disjoint subsets of $[n]$. Every partition is also a semi-partition. However, in a semi-partition, there can be missing elements, that is, elements of $[n]$ that are not contained in the union $\cup\sigma$ of the blocks of $\sigma$. By ``blocks'' we mean the subsets of $[n]$ that are the elements of $\sigma$. From every semi-partition $\sigma\in\S_n$, a partition $\pi$ is obtained by inserting a singleton block for each missing element. We call $\pi$ the partition associated with $\sigma$, and we define $\sigma(i)=\pi(i)$ for each $i\in[n]$, where $\pi(i)$ is defined in Section~\ref{sec:det-ld-gen}.
In order that equation~\reff{Sampl:eq:jumps-sigma-rvt} below hold,
we associate with each element $\sigma$ of $\S_n$ a transformation $\R^{n^2}\times\R^n\to\R^{n^2}\times\R^n$,\label{Sampl:not:Sn-transf} which we also denote by $\sigma$, by
$\sigma(r,v)=(r',v')$, where
\[v'(i)=v(\sigma(i))\I{i\notin\cup\sigma}\]
and
\[r'(i,j)=\left(v(\sigma(i))\I{i\in\cup\sigma}
+r(\sigma(i),\sigma(j))+v(\sigma(j))\I{j\in\cup\sigma}\right)\I{i\neq j}\]
for $i,j\in[n]$.

We define the function
\[\varsigma_n:\p\to\S_n, \quad \pi\mapsto
\{B\cap[n]: B\in\pi, \#B\geq 2\}\setminus\{\emptyset\}.\]
that removes all singleton blocks from a partition of $\N$ and restricts the semi-partition obtained in this way to a semi-partition of $[n]$.
For each reproduction event encoded by a point $(s,\pi)\in\eta$, the corresponding jump of the process $((r_t,v_t),t\in\R_+)$ can be described by
\begin{equation}
\label{Sampl:eq:jumps-sigma-rvt}
\varsigma_n(\pi)(\gamma_n(r_{s-},v_{s-}))=\gamma_n(r_s,v_s).
\end{equation}
Here we cannot use the restriction $\gamma_n(\pi)$ (of $\pi$ to $[n]$) instead of $\varsigma_n(\pi)$ as we cannot read off from $\gamma_n(\pi)$ which singleton blocks in $\gamma_n(\pi)$ are also singleton blocks in $\pi$.

We define the set of partitions
\[\label{Sampl:not:hatp^n}\hat\p^n=\{\pi\in\p:\varsigma_n(\pi)\neq\emptyset\}.\]
We remark that $\hat\p^n$ is the set of partitions of $\N$ in which not all of the first $n$ integers form singleton blocks, hence it is strictly larger than the set $\p^n$. 
Only reproduction events that are encoded by a partition in $\hat\p^n$ affect the decomposed genealogical distances on the first $n$ levels $(\gamma_n(r_t,v_t),t\in\R_+)$. If $\eta$ satisfies the condition
\begin{equation}
\label{Sampl:eq:ass-eta-hat}
\eta((0,t]\times\hat\p^n)<\infty\quad\text{for all }t\in(0,\infty)\text{ and }n\in\N.
\end{equation}
then there are only finitely many reproduction events in bounded time intervals that result in a jump of the process $(\gamma_n(r_t,v_t),t\in\R_+)$. Between such jumps,
the matrix $r_t$ is constant, and the entries of the vector $v_t$ grow linearly with slope $1$, that is, $v_t(i)+s=v_{t+s}(i)$ for $i\in[n]$ and $t,s\in\R_+$ with $\eta((t,t+s]\times\hat\p^n)=0$.

\subsection{Stochastic evolution}
\label{Sampl:sec:dec-stoch}
Now let $\eta$ be the Poisson random measure from Section~\ref{Sampl:sec:Xi-ld} whose distribution is characterized by some finite measure $\Xi$ on $\Delta$. Consider the population model from Subsection~\ref{Sampl:sec:dec-det} driven by the Poisson random measure $\eta$.
For each $n\in\N$ and $\sigma\in\S_n\setminus\{\emptyset\}$, the rate at which reproduction events encoded by a partition in $\varsigma_n^{-1}(\sigma)\in\p$ occur is given by
\begin{align}
\lambda_{n,\sigma} = & H_\Xi(\varsigma_n^{-1}(\sigma))\notag\\
= & \int_\Delta\kappa(x,\varsigma_n^{-1}(\sigma))\left| x\right|_2^{-2}\Xi_0(dx)
+\Xi\{0\}\sum_{1\leq i < j}\I{K_{i,j}\in\varsigma_n^{-1}(\sigma)}\notag\\
= & \int_\Delta \sum_{\substack{i_1,\ldots,i_\ell\in\N\\ \text{ pairwise distinct}}}
x_{i_1}^{k_1}\cdots x_{i_\ell}^{k_\ell}(1-\left|x\right|_1)^{n-k_1-\ldots-k_\ell}|x|_2^{-2}\Xi_0(dx)\notag\\
& + \Xi\{0\}\I{\ell=1,k_1=2}+\infty\I{\Xi\{0\}>0,\ell=1,k_1=1}\label{Sampl:eq:lambda-S}
\end{align}
where $\ell=\#\sigma$, and $k_1,\ldots,k_\ell\geq 1$ are the sizes of the subsets in $\sigma$ in arbitrary order, and $\Xi_0$ is defined as in~\reff{Sampl:eq:dec-Xi}. For the last equality, we consider the paintbox partition $\pi$ associated with $x\in\Delta$: With the notation from the beginning of Section~\ref{Sampl:sec:Xi-ld}, integers $i,j\in[n]$ are elements of a common subset in $\varsigma_n(\pi)$ if and only if $U_i$ and $U_j$ fall into a common subinterval that is not the dust interval. In particular, $i\notin\cup\varsigma_n(\pi)$ if and only if $U_i$ falls into the dust interval.

Note that the rates $\lambda_\pi$ for $\pi\in\p_n\setminus\{\mathbf{0}_n\}$, which we discussed already in Remark \ref{Sampl:rem:rates-Schw}, satisfy
\begin{equation}
\label{Sampl:eq:lambda-P-S}
\lambda_\pi=H_\Xi(\gamma_n^{-1}(\pi))=H_\Xi(\bigcup_\sigma\{\varsigma_n^{-1}(\sigma)\})=\sum_\sigma\lambda_{n,\sigma},
\end{equation}
where the union and the sum are over all semi-partitions $\sigma\in\S_n$ with the same non-singleton blocks as $\pi$. In \reff{Sampl:eq:lambda-P-S}, we also use the restriction map $\gamma_n:\p\to\p_n$.
From equations~\reff{Sampl:eq:lambda-S} and \reff{Sampl:eq:lambda-P-S}, we see that $\lambda_{\{\{1,2\}\}}=\Xi(\Delta)<\infty$ and $\lambda_\pi<\infty$ for all $\pi\in\p_n\setminus\{\mathbf{0}_n\}$, where $\mathbf{0}_n=\{\{1\},\ldots,\{n\}\}$. This implies $\eta((0,t]\times\p^n)<\infty$ a.\,s.\ for all $t\in(0,\infty)$. That is, condition~\reff{Sampl:eq:ass-eta} is a.\,s.\ satisfied, as stated in Section~\ref{Sampl:sec:Xi-ld}.
The condition~\reff{Sampl:eq:dust-free} for $\Xi$ to be dust-free is the condition that $\lambda_{1,\{\{1\}\}}=\infty$.
That is, each particle reproduces with infinite rate if and only if $\Xi$ is dust-free.
Hence, if $\Xi$ is not dust-free, then almost every realization of $\eta$ satisfies condition~\reff{Sampl:eq:ass-eta-hat}.
Moreover, if $\Xi$ is not dust-free, then $\lambda_{n,\sigma}<\infty$ for all $n\in\N$ and $\sigma\in\S_n\setminus\{\emptyset\}$ as a consequence of equation~\reff{Sampl:eq:lambda-S}.
\begin{rem}
Consider the case that $\Xi$ is concentrated on $\{(x,0,0,\ldots):x\in[0,1]\}\subset\Delta$. In this case, which corresponds to the $\Lambda$-coalescent, a.\,s.\ no simultaneous multiple reproduction events occur. The measure $\Xi_0$ is then determined by the finite measure $\Lambda_0=\varpi(\Xi_0)$, where $\varpi:\Delta\to[0,1]$, $x\mapsto x_1$. For $B\subset[n]$ and $k=\#B$, it then follows
\[\lambda_{n,\{B\}}=\int_{[0,1]}x^k(1-x)^{n-k}x^{-2}\Lambda_0(dx)+\Xi\{0\}\I{k=2}+\infty\I{\Xi\{0\}>0,k=1}.\]
The rates $\lambda_{n,\sigma}$ for $\sigma\in\S_n$ with $\#\sigma>1$ are equal to zero in this case. 
\end{rem}

Now we consider the $\R^{\N^2}\times\R^{\N}$-valued process $((r_t,v_t),t\in\R_+)$ from Subsection~\ref{Sampl:sec:dec-det}, driven by the Poisson random measure $\eta$. The initial state is defined as a $\R^{\N^2}\times\R^\N$-valued random variable $(r_0,v_0)$ that is independent of $\eta$.
\begin{prop}
\label{prop:rv-Markov}
The process $((r_t,v_t),t\in\R_+)$ is Markov.
\end{prop}
\begin{proof}
This follows by the same argument as for Proposition \ref{prop:rho-Markov}.
\end{proof}
Recall the set $\hat\C$ from Section~\ref{Sampl:sec:tv}. For $(r,v)\in\R^{\N^2}\times\R^\N$ and $\phi\in\hat\C_n$, we write
\[\langle \nabla^{v} \phi, \underline 1 \rangle (r,v) = \sum_{i\in\N} \frac{\partial}{\partial v(i)}\phi(r,v).\]
From the discussion above and the description of the process $(\gamma_n(r_t,v_t),t\in\R_+)$ around equation \reff{Sampl:eq:jumps-sigma-rvt}, we deduce the next proposition.
\begin{prop}
\label{Sampl:prop:mp-R}
Assume that $\Xi$ is not dust-free.
Define an operator $\hat A=\hat A_{\rm grow}+\hat A_{\rm repr}$ with domain $\hat\C$ by
\[\hat A_{\rm grow}\phi(r,v)=\langle \nabla^{v} \phi, \underline{1}\rangle(r,v)\]
and
\[\hat A_{\rm repr}\phi(r,v)
=\sum_{\sigma\in\S_{n}\setminus\{\emptyset\}}\lambda_{n,\sigma}(\phi(\sigma(\gamma_n(r,v)))-\phi(r,v))\]
for $n\in\N$, $\phi\in\hat\C_n$ and $(r,v)\in\R^{\N^2}\times\R^\N$.
Then the stochastic process
$((r_t,v_t),t\in\R_+)$ solves the martingale problem $(\hat A,\hat\C)$.
\end{prop}
The operator $\hat A_{\rm grow}$ accounts for the growth of the marks $v_t$ which is described in the end of Subsection~\ref{Sampl:sec:dec-det}. The operator $\hat A_{\rm repr}$ stands for the jumps of the decomposed genealogical distances in reproduction events which are described by equation~\reff{Sampl:eq:jumps-sigma-rvt}. We omit a formal proof of Proposition \ref{Sampl:prop:mp-R}.

Finally, we consider again the process $(\rho_t,t\in\R_+)$ which is defined from the Poisson random measure $\eta$ and the initial state $\rho_0=\alpha(r_0,v_0)$ as in Section \ref{sec:det-ld-gen}. We assume $\rho_0\in\Uu$. Then $\rho_t\in\Uu$ for all $t\in\R_+$ by Remark~\ref{Sampl:rem:rho-um}. Moreover, the construction in Subsection~\ref{Sampl:sec:dec-det} and the definition of the map $\alpha$ in Section~\ref{Sampl:sec:dist-dec} yield $\rho_t=\alpha(r_t,v_t)$ and $(r_t,v_t)\in\hat\Uu$ for all $t\in\R_+$. 
We further assume that
\begin{equation}
\label{Sampl:eq:ass-r0v0}
(r_0,v_0)=\beta(\rho_0).
\end{equation}
Then by the following proposition, the decomposition $(r_t,v_t)$ of the semi-ultrametric $\rho_t$ is the one given by the map $\beta$ from Section~\ref{Sampl:sec:dist-dec}, namely the decomposition into the external branches and the remaining subtree.
\begin{prop}
\label{Sampl:prop:rv-R}
Assumption~\reff{Sampl:eq:ass-r0v0} implies $(r_t,v_t)=\beta(\rho_t)$ a.\,s.\ for each $t\in\R_+$.
\end{prop}
We prove Proposition~\ref{Sampl:prop:rv-R} in Section~\ref{sec:proof:rv-R}.
\begin{cor}
\label{cor:mp-R}
The process $(\beta(\rho_t),t\in\R_+)$ is Markov and solves the martingale problem $(\hat A,\hat\C)$ from Proposition \ref{Sampl:prop:mp-R}.
\end{cor}
\begin{proof}
This is immediate from Propositions~\ref{prop:rv-Markov}, \ref{Sampl:prop:mp-R} and \ref{Sampl:prop:rv-R}. The Markov property can alternatively be seen from Proposition \ref{prop:rho-Markov} and Remark \ref{rem:Markov-beta}.
\end{proof}

\section{Tree-valued Fleming-Viot processes}
\label{Sampl:sec:TVFV}
In this section, we apply Theorem~\ref{Sampl:thm:tv} to the process $(\rho_t,t\in\R_+)$ from Section~\ref{Sampl:sec:Xi-ld}.
By Remark \ref{Sampl:rem:rho-um}, we can consider $(\rho_t,t\in\R_+)$ as an $\Uu$-valued process. 
We call all the image processes in Theorem~\ref{Sampl:thm:tv} tree-valued Fleming-Viot processes.
To distinguish them, we also call them
$\U$-, $\hat\U$, and $\UUerg$-valued $\Xi$-Fleming-Viot processes. Proposition \ref{Sampl:prop:mp-unique-mm} below states that the martingale problems for the tree-valued Fleming-Viot processes have unique solutions.

\subsection{Processes with values in the space of metric measure spaces}
\label{Sampl:sec:TVFV-mm}
In this subsection, we consider a finite measure $\Xi$ on $\Delta$ that is dust-free.
Let $\chi\in\U$, and let $(\rho_t,t\in\R_+)$ be the $\Uu$-valued Markov process from Section~\ref{Sampl:sec:Xi-ld} that is defined in terms of $\Xi$ and an initial state $\rho_0$ with distribution $\nu^\chi$.
We define a $\U$-valued $\Xi$-Fleming-Viot process $(\chi_t,t\in\R_+)$ with initial state $\chi\in\U$ by $\chi_t=\psi(\rho_t)$.
As a justification for this name, we note that $\chi_0=\psi(\rho_0)=\chi$ a.\,s.\ by Corollary \ref{cor:repr}\ref{item:cor:repr-nd:det} and Remark \ref{rem:sampling-dust-free}.
By Theorem~\ref{Sampl:thm:tv} and Propositions~\ref{prop:rho-Markov}, \ref{Sampl:prop:mp-rho}, \ref{Sampl:prop:rho-exch}, and \ref{Sampl:prop:rho-df}, the process $(\chi_t,t\in\R_+)$ is Markovian and solves the martingale problem $(B,\Pi)$, where the generator $B$ is defined by $B\Phi(\chi)=\nu^\chi(A\phi)$ for $\phi\in\C$ with associated polynomial $\Phi\in\Pi$, and $\chi\in\U$. Here $A$ is the generator defined in Proposition~\ref{Sampl:prop:mp-rho}. The martingale problem $(B,\Pi)$ is a generalization of the martingale problem in Theorem 1 of Greven, Pfaffelhuber, and Winter \cite{GPW13}.

\subsection{Processes with values in the space of marked metric measure spaces}
\label{Sampl:sec:TVFV-mmm}
Let $\Xi$ be a general finite measure on the simplex $\Delta$. Let $\hat\chi\in\hat\U$, let $\rho_0$ be a $\Uu$-valued random variable with distribution $\alpha(\nu^{\hat\chi})$, and let the $\Uu$-valued Markov process $(\rho_t,t\in\R_+)$ be defined, as in Section~\ref{Sampl:sec:Xi-ld}, from $\Xi$ and the initial state $\rho_0$.
We define a $\hat\U$-valued $\Xi$-Fleming-Viot process $(\hat\chi_t,t\in\R_+)$ with initial state $\chi\in\hat\U$ by $\hat\chi_t=\hat\psi(\beta(\rho_t))$ for $t\in\R_+$. To justify this name, we note that the initial state satisfies $\hat\chi_0=\hat\psi(\beta(\rho_0))=\chi$ a.\,s.\ by Theorem \ref{thm:repr}\ref{item:thm:repr:det}. By Theorem~\ref{Sampl:thm:tv} and Propositions~\ref{prop:rho-Markov} and \ref{Sampl:prop:rho-exch}, the process $(\hat\chi_t,t\in\R_+)$ is Markovian. 

If $\Xi$ is not dust-free, then by Theorem~\ref{Sampl:thm:tv} and Corollary~\ref{cor:mp-R}, the process $(\hat\chi_t,t\in\R_+)$ solves the martingale problem $(\hat B,\hat\Pi)$, where the generator $\hat B$ is defined by $\hat B\Phi(\chi')=\nu^{\chi'}(\hat A\phi)$ for all $\phi\in\hat\C$ with associated marked polynomial $\Phi$, and all $\chi'\in\hat\U$. Here the generator $\hat A$ is defined as in Proposition~\ref{Sampl:prop:mp-R}.

If $\Xi$ is dust-free, then for each $t\in(0,\infty)$ by Remark~\ref{Sampl:rem:tv-df} and Proposition \ref{Sampl:prop:rho-df}, the marked metric measure space $\hat\chi_t$ is a.\,s.\ dust-free, and $\hat\chi_t$ is determined a.\,s.\ by the associated metric measure space $\chi_t$.

\subsection{Processes with values in the space of distance matrix distributions}
\label{Sampl:sec:TVFV-dm}
Let $(\hat\chi_t,t\in\R_+)$ be the process from Section~\ref{Sampl:sec:TVFV-mmm}, where $\Xi$ is a general finite measure on the simplex $\Delta$. We define a $\UUerg$-valued $\Xi$-Fleming-Viot process $(\xi_t,t\in\R_+)$ with initial state $\alpha(\nu^{\hat\chi_0})\in\UUerg$ by $\xi_t=\alpha(\nu^{\hat\chi_t})$.
Again by Theorem~\ref{Sampl:thm:tv} and Propositions~\ref{prop:rho-Markov}, \ref{Sampl:prop:mp-rho} and \ref{Sampl:prop:rho-exch}, it follows that $(\xi_t,t\in\R_+)$ is Markovian and solves the martingale problem $(C,\Cc)$, where the generator $C$ is defined by $C\Psi(\xi)=\xi(A\phi)$ for all $\xi\in\UUerg$, $\phi\in\C$, and $\Psi\in\Cc:\xi'\mapsto\xi'\phi$. Here the generator $A$ is defined as in Proposition~\ref{Sampl:prop:mp-rho}.

\subsection{Well-posedness of the martingale problem}
\begin{prop}
\label{Sampl:prop:mp-unique-mm}
The martingale problems $(B,\Pi)$, $(\hat B,\hat\Pi)$, and $(C,\Cc)$ are well-posed.
\end{prop}
That a martingale problem is well-posed means that a solution exists whose finite-dimensional distributions are uniquely determined by the initial state. A proof of Proposition \ref{Sampl:prop:mp-unique-mm} by duality is given in Section \ref{sec:proof-uniqueness}.

\section{Some semigroup properties}
\label{Sampl:sec:semigroup}
In this section, we state Feller continuity of
tree-valued $\Xi$-Fleming-Viot processes, and that the domains of the martingale problems for them are cores.
We consider $\hat\U$-valued $\Xi$-Fleming-Viot processes in detail, analogous results hold for the other processes from Section~\ref{Sampl:sec:TVFV}.

Let $\Xi$ be a finite measure on the simplex $\Delta$.
For $\chi\in\hat\U$, let $(\hat\chi_t,t\in\R_+)$ under the probability measure $\P_\chi$ with associated expectation $\E_\chi$ be the $\hat\U$-valued $\Xi$-Fleming-Viot process from Section~\ref{Sampl:sec:TVFV-mmm} with initial state $\chi$. We denote by $\Cb(E)$ the set of bounded continuous $\R$-valued functions on a metric space $E$. We endow $\Cb(E)$ with the supremum norm.

The results in this section rely on the following lemma which we prove in Section \ref{sec:proof:Pi-preserve} using the lookdown construction.
\begin{lem}
\label{Sampl:lem:Pi-preserve}
For each $t\in\R_+$ and $\Phi\in\hat\Pi$, the function
$\hat\U\to\R$, $\chi\mapsto\E_\chi[\Phi(\hat\chi_t)]$ is an element of $\hat\Pi$.
\end{lem}

As a corollary, we obtain the Feller continuity of a $\hat\U$-valued $\Xi$-Fleming-Viot process, namely that its semigroup preserves the set of bounded continuous functions.
\begin{cor}
\label{Sampl:prop:feller-mmm}
For each $t\in\R_+$ and $f\in\Cb(\hat\U)$, the map $\hat\U\to\R$, $\chi\mapsto\E_\chi[f(\hat\chi_t)]$ is continuous.
\end{cor}
\begin{proof}
This follows from Lemma \ref{Sampl:lem:Pi-preserve} as the set $\hat\Pi$ of marked polynomials is convergence determining, we use the definition of convergence in distribution in $\hat\U$. 
\end{proof}

Let $\hat L$ denote the closure of $\hat\Pi$ in $\Cb(\hat\U)$ with respect to the supremum norm. For application in \cite{Conv}, we note two more corollaries of Lemma \ref{Sampl:lem:Pi-preserve}.
The first of them states that the semigroup of a $\hat\U$-valued $\Xi$-Fleming-Viot process can be restricted to a semigroup on $\hat L$ that is strongly continuous.
\begin{cor}
\label{Sampl:cor:sc}
Let $f\in \hat L$.
Then for each $t\in\R_+$, the function $\hat\U\to\R$, $\chi\mapsto\E_\chi[f(\hat\chi_t)]$ is an element of $\hat L$. Moreover,
\[\lim_{t\downarrow 0}\sup_{\chi\in\hat\U}
\abs{\E_\chi[f(\hat\chi_t)]-\E_\chi[f(\hat\chi_0)]}=0.\]
\end{cor}
\begin{proof}
The first assertion follows from Lemma \ref{Sampl:lem:Pi-preserve} and the definition of $\hat L$.
As $(\hat\chi_t,t\in\R_+)$ solves the martingale problem $(\hat B,\hat\Pi)$ from Section~\ref{Sampl:sec:TVFV-mmm},
\[\E_\chi[\Phi(\hat\chi_t)]-\E_\chi[\Phi(\hat\chi_0)]
=\E_\chi[\int_0^t\hat B\Phi(\hat\chi_s)ds]\]
for all $t\in\R_+$ and $\Phi\in\hat\Pi$. The second assertion follows as $\hat B\Phi$ is bounded and by definition of $\hat L$.
\end{proof}
The next corollary says that the semigroup on $\hat L$ of a $\hat\U$-valued $\Xi$-Fleming-Viot process is generated by the closure of the operator $\hat B$ with domain $\hat\Pi$, see \cite{EK86}*{Chapter 1} for the definitions.
\begin{cor}
The subspace $\hat\Pi\subset\Cb(\hat\U)$ is a core for the generator of the semigroup on $\hat L$ of a $\hat\U$-valued $\Xi$-Fleming-Viot process.
\end{cor}
\begin{proof}
We note that $\hat B$ is the restriction of the generator of the semigroup to $\hat\Pi$ and apply Proposition 1.3.3 and Corollary 1.1.6 of \cite{EK86}, using Lemma \ref{Sampl:lem:Pi-preserve} and Corollary \ref{Sampl:cor:sc}. 
\end{proof}

Let $L$ be the closure of $\Pi$ in $\Cb(\U)$ and let $L'$ be the closure of $\Cc$ in $\Cb(\UUerg)$, with respect to the supremum norm. In the same way as above, it can be shown: The semigroup on $L'$ of a $\UUerg$-valued $\Xi$-Fleming-Viot process is strongly continuous and generated by the closure of the operator $C$ with domain $\Cc$ from Section~\ref{Sampl:sec:TVFV-dm}. If $\Xi$ is dust-free, then the semigroup on $L$ of a $\U$-valued $\Xi$-Fleming-Viot process is strongly continuous and generated by the closure of the operator $B$ with domain $\Pi$ from Section~\ref{Sampl:sec:TVFV-mm}. Continuity properties analogous to Proposition~\ref{Sampl:prop:feller-mmm} also hold.

\section{Convergence to equilibrium}
\label{Sampl:sec:equil}
Let $\Xi$ be a finite measure on the simplex $\Delta$ with $\Xi(\Delta)>0$.
We show convergence to equilibrium for the $\hat\Uu$-valued process $(\beta(\rho_t),t\in\R_+)$ from Section~\ref{Sampl:sec:dec-stoch}.
From this, we deduce in Proposition \ref{prop:equil} that also the tree-valued $\Xi$-Fleming-Viot process
from Section~\ref{Sampl:sec:TVFV-mmm} converges to equilibrium. In the same way, it can be shown that the other processes from Section~\ref{Sampl:sec:TVFV} converge to equilibrium.

We define stationary processes and use a coupling argument. Analogously to Section \ref{Sampl:sec:Xi-ld}, let $\bar\eta$ be a Poisson random measure on $\R\times\p$ with intensity $dt\;H_\Xi(d\pi)$. This Poisson random measure drives a population model in two-sided time (with time axis $\R$) where the reproduction events and the ancestral levels $\bar A_s(t,i)$ are defined as in Section~\ref{sec:det-ld-gen}.
Then we define the stationary $\Uu$-valued process $(\bar\rho_t,t\in\R)$ of the genealogical distances by
\[\bar\rho_t(i,j)=2t-2\sup\{s\in(-\infty,t]:\bar A_s(t,i)=\bar A_s(t,j)\}\]
for $t\in\R$, $i,j\in\N$.
On an event of probability $1$, all these distances are finite. This follows from the assumption that $\Xi(\Delta)>0$. That $\bar\rho_t$ is indeed a semi-ultrametric for each $t\in\R$ can be seen as in Remark~\ref{Sampl:rem:rho-um}. Clearly, $\rho_t$ is exchangeable, which follows from exchangeability of the $\Xi$-coalescent as in Remark \ref{Sampl:rem:Xi-coal} or can be shown as in the proof of Proposition \ref{Sampl:prop:rho-exch}.

Let $\eta$ denote the restriction of $\bar\eta$ to $(0,\infty)\times\p$. Let $\chi\in\hat\U$ be arbitrary, and let $\rho_0$ be a $\Uu$-valued random variable with distribution $\alpha(\nu^\chi)$, independent of $\eta$. Let the process $(\rho_t,t\in\R_+)$ be defined from $\rho_0$ and $\eta$ as in Section~\ref{sec:det-ld-gen}.
For $n\geq 2$, on the event $\{\max_{i,j\in[n]}\bar\rho_t(i,j)<2t\}$, the marked distance matrix $\gamma_n(\beta(\rho_t))$ does not depend on $\rho_0$. This follows from the construction in Section~\ref{sec:det-ld-gen} and the definition of the map $\beta$ in Section \ref{Sampl:sec:dist-dec}.
As $\bar\rho_t$ can also be obtained from $\bar\rho_0$ and $\eta$ as in Section~\ref{sec:det-ld-gen}, it follows that $\gamma_n(\beta(\rho_t))=\gamma_n(\beta(\bar\rho_t))$ on the event $\{\max_{i,j\in[n]}\bar\rho_t(i,j)<2t\}$. By stationarity of $(\bar\rho_t,t\in\R)$, it follows that
\begin{equation}
\label{Sampl:eq:conv-rv}
\left|\E[\phi(\beta(\rho_t))]-\E[\phi(\beta(\bar\rho_0))]\right|\leq
2\sup\abs{\phi}\P(\max_{i,j\in[n]}\bar\rho_0(i,j)\geq 2t)\to 0\quad (t\to\infty)
\end{equation}
for all $\phi\in\hat\C_n$.

We call a $\hat\U$-valued random variable that is distributed as $\bar\chi_0:=\hat\psi(\beta(\bar\rho_0))$ a $\Xi$-coalescent measure tree, generalizing the $\Lambda$-coalescent measure tree from \cite{GPW13}.
A $\hat\U$-valued $\Xi$-Fleming-Viot process $(\chi_t,t\in\R_+)$ with initial state $\chi$ is given by $\chi_t=\hat\psi(\beta(\rho_t))$, as in Section~\ref{Sampl:sec:TVFV-mmm}.
\begin{prop}
\label{prop:equil}
The $\hat\U$-valued random variable $\chi_t$ converges in distribution to a $\Xi$-coalescent measure tree as $t\to\infty$.
\end{prop}
\begin{proof}
As in Proposition~\ref{Sampl:prop:eq-RP} below, we obtain $\E[\Phi(\bar\chi_0)]=\E[\phi(\bar\rho_0)]$ and $\E[\Phi(\chi_t)]=\E[\phi(\rho_t)]$. The convergence~\reff{Sampl:eq:conv-rv} then yields that $\E[\Phi(\chi_t)]$ converges to $\E[\Phi(\bar\chi_0)]$ as $t\to\infty$
for all marked polynomials $\Phi\in\hat\Pi$. The assertion follows as the set $\hat\Pi$ is convergence determining in $\hat\U$.
\end{proof}

A stationary $\hat\U$-valued $\Xi$-Fleming-Viot process can be defined by $(\hat\psi(\beta(\bar\rho_t)),t\in\R)$. In \cite{GPW13}*{Theorem 3}, duality is used to show that the tree-valued Fleming-Viot process converges to an equilibrium.
In \cite{DK99}*{Theorem 4.1}, convergence to stationarity of measure-valued Fleming-Viot processes is also proved by a coupling argument.

\section{Proofs of the general results}
\label{sec:proofs}
In Subsection \ref{sec:proof:sampl-unique}, we prove Proposition \ref{Sampl:prop:sampl-unique} which is needed for the proof of the uniqueness result (Proposition \ref{Sampl:prop:uniqueness}) in Subsection \ref{sec:proofs:resampling}.
We prove the sampling representation (Theorem \ref{thm:repr}) in Subsections \ref{Sampl:sec:proof:meas} -- \ref{Sampl:sec:proof:sampling}.
Theorem \ref{Sampl:thm:tv} gives the application to tree-valued processes and is proved in Subsection \ref{sec:proof:tv}.

\subsection{Proof of Proposition \ref{Sampl:prop:sampl-unique}}
\label{sec:proof:sampl-unique}
The proof of this result from Section \ref{Sampl:sec:mmm-trees} relies on the fact that in a separable metric space, an iid sequence with respect to a probability measure on the Borel sigma algebra has no isolated elements.
\begin{proof}[Proof of Proposition \ref{Sampl:prop:sampl-unique}]
Let $((x(i),v(i)),i\in\N)$ be an $m$-iid sequence in $X\times\R_+$. We may assume
\[r=(r(i,j))_{i,j\in\N}=(r'(x(i),x(j)))_{i,j\in\N}.\]
We write $\rho=\alpha(r,v)$. We show that $v=\Upsilon(\rho)$ a.\,s.\ from which the assertion follows by definition of the map $\beta$.

Let $\ep>0$ and $i\in\N$. By separability, $X\times\R_+$ can be covered by countably many balls of diameter $\ep$. This implies
\[m\{(x',v')\in X\times\R_+: r'(x(i),x')\vee|v(i)-v'|\leq 2\ep\}>0\quad\text{a.\,s.,}\]
and that there exists a random $j\in\N\setminus\{i\}$ with
\begin{equation}
\label{Sampl:eq:sampl-unique-j}
r'(x(i),x(j))\vee|v(i)-v(j)|\leq 2\ep\quad\text{a.\,s.}
\end{equation}
By inequality~\reff{Sampl:eq:sampl-unique-j} and the definition of $\rho$, it follows that
\[2v(i)+4\ep\geq v(i)+v(j)+r(i,j)=\rho(i,j)\quad\text{a.\,s.}\]
Using the definition of the map $\Upsilon$, we deduce
\[v(i)+2\ep\geq\tfrac{1}{2}\rho(i,j)\geq \Upsilon(\rho)(i)\quad\text{a.\,s.}\]
For the converse inequality, we first note that
\begin{equation}
\label{Sampl:eq:sampl-unique-vrho}
2v(i)\leq v(i)+v(j)+2\ep+r(i,j)=\rho(i,j)+2\ep
\end{equation}
by inequality~\reff{Sampl:eq:sampl-unique-j} and the definition of $\rho$.
Moreover, for all $k\in\N\setminus\{i,j\}$, we obtain
\begin{align*}
&2v(i)-2\ep\leq \rho(i,j)\leq\rho(i,k)\vee\rho(k,j)\\
&\leq v(k)+r(i,k)\vee r(k,j)+v(i)\vee v(j)\\
&\leq v(k)+r(i,k)+v(i)+\left|r(k,j)-r(i,k)\right|+\left|v(j)-v(i)\right|\\
&\leq\rho(i,k)+r(i,j)+2\ep\leq\rho(i,k)+4\ep\quad\text{a.\,s.}
\end{align*}
Here we use inequality~\reff{Sampl:eq:sampl-unique-vrho} for the first and inequality~\reff{Sampl:eq:sampl-unique-j} for the fifth and sixth step, the definition of $\rho$ for the third and fifth step, and ultrametricity for the second step.
By definition of the map $\Upsilon$, we obtain
\[\Upsilon(\rho)(i)=\tfrac{1}{2}\inf_{k\in\N\setminus\{i\}}\rho(i,k)
\geq v(i)-3\ep\quad\text{a.\,s.}\]
As $\ep>0$ and $i\in\N$ were arbitrary, it follows that $\Upsilon(\rho)=v$ a.\,s.
\end{proof}

\subsection{Measurability of the construction of (marked) metric measure spaces}
\label{Sampl:sec:proof:meas}
In this subsection, we show Proposition \ref{Sampl:prop:meas} from Section \ref{Sampl:sec:mdm-mmm-new}. We only discuss measurability of the map $\hat\psi:\Dd\times\R_+^{\N}\to\hat\M$ therein. Measurability of the map $\psi:\Dd\times\M$ follows along the same lines.

Recall that the Prohorov distance between two probability measures $\mu$ and $\mu'$ on the Borel sigma algebra on a metric space $(Z,d^Z)$ is given by
\begin{equation}
\label{eq:Proh}
\dP^Z(\mu,\mu')=\inf\{\ep>0: \mu(F)\leq\mu'(F^\ep)+\ep\text{ for all closed $F\subset Z$}\},
\end{equation}
where $F^\ep=\{z\in Z: d^Z(z,F)<\ep\}$. If $(Z,d^Z)$ is separable, then the coupling characterization of the Prohorov distance holds, which can be found e.\,g.\ in \cite{EK86}*{Theorem 3.1.2}:
\begin{equation}
\label{Sampl:eq:Proh-coupling}
\dP^Z(\mu,\mu')=\inf_\xi\inf\{\ep>0:\xi\{(x,y)\in Z^2:d^Z(x,y)>\ep\}<\ep\},
\end{equation}
where the first infimum is over all couplings $\xi$ of the probability measures $\mu$ and $\mu'$.

We also use the marked Gromov-Prohorov distance $\dMGP$ which metrizes the marked Gromov-weak topology on $\hat\M$, see \cite{DGP11}. It is defined by
\[\dMGP((X,r,m),(X',r',m'))=\inf_{Z,\varphi,\varphi'}\dP^Z(\hat\varphi(m),\hat\varphi'(m'))\]
for marked metric measure spaces $(X,r,m)$ and $(X',r',m')$.
Here the infimum is over all isometric embeddings $\varphi:X\to Z$ and $\varphi':X'\to Z$ into complete and separable metric spaces $(Z,d^Z)$. The space $Z\times\R_+$ is endowed with the product metric $d^{Z\times\R_+}((z,v),(z',v'))=d^Z(z,z')\vee|v-v'|$, analogously for $X\times\R_+$ and $X'\times\R_+$. The maps $\hat\varphi:X\times\R_+\to Z\times\R_+$ and $\hat\varphi':X'\times\R_+\to Z\times\R_+$ are defined by
$\hat\varphi(x,v)=(\varphi(x),v)$, $(x,v)\in X\times\R_+$ and
$\hat\varphi'(x',v)=(\varphi(x'),v)$, $(x',v)\in X'\times\R_+$.

We write $\hat\Dd=\Dd\times\R_+^\N$\label{Sampl:not:hatDd}. For $n\in\N$, we denote by
\begin{align*}
\hat\Dd_n=\{&(r,v)\in\R_+^{n^2}\times\R_+^n:r(i,i)=0,r(i,j)=r(j,i),\\
&r(i,j)+r(j,k)\geq r(i,k)\text{ for all }i,j,k\in[n]\}
\end{align*}
the space of decomposed semimetrics on $[n]$ which we view as a subspace of $\R^{n^2}\times\R^n$. We denote by $\hat\psi_n:\hat\Dd_n\to\hat\M$ the function that maps $(r,v)\in\hat\Dd_n$ to the isomorphy class of the marked metric measure space
$([n],r,n^{-1}\sum_{i=1}^n\delta_{(i,v(i))})$, here we also identify the elements of the semi-metric space $([n],r)$ with distance zero.
\begin{lem}
\label{Sampl:rem:cont-psin}
The map $\hat\psi_n:\Dd_n\to\hat\M$ is continuous.
\end{lem}
\begin{proof}
W.\,l.\,o.\,g.\ we can assume that $\hat\Dd$ is endowed with the metric $d$ that is given by
\[d((r,v),(r',v'))=\sup_{k\in\N}((\max_{i,j\in[k]}\left|r(i,j)-r'(i,j)\right|\vee\max_{i\in[k]}\left|v(i)-v'(i)\right|)\wedge (2^{-k}))\]
for all $(r,v),(r',v')\in\hat\Dd$. For $(r,v),(r',v')\in\hat\Dd_n$, we define a probability measure $\xi$ on $(\hat\Dd)^2$ as the distribution of
$((r(x_i,x_j))_{i,j\in\N},(\tilde v_i)_{i\in\N},
(r'(x'_i,x'_j))_{i,j\in\N},(\tilde v'_i)_{i\in\N})$,
where $(x_i,\tilde v_i,x'_i,\tilde v'_i)_{i\in\N}$ is an iid sequence with distribution
$n^{-1}\sum_{k=1}^n\delta_{(k,v(k),k,v'(k))}$.
Then
$\xi(\cdot\times\hat\Dd)=\nu^{\hat\psi_n(r,v)}$
and
$\xi(\hat\Dd\times\cdot)=\nu^{\hat\psi_n(r',v')}$.
For
\[c:=\max_{i,j\in[n]}\left|r(i,j)-r'(i,j)\right|\vee\max_{i\in[n]}\left|v(i)-v'(i)\right|,\]
the coupling characterization~\reff{Sampl:eq:Proh-coupling} implies
\[\dP(\nu^{\hat\psi_n(r,v)},\nu^{\hat\psi_n(r',v')})
\leq c+\xi\{(y,y')\in\hat\Dd^2:d(y,y')> c\}=c.\]
Continuity of $\hat\psi_n$ follows by definition of the marked Gromov-weak topology.
\end{proof}
\begin{proof}[Proof of Proposition~\ref{Sampl:prop:meas}]
Let $(r,v)\in\hat\Dd^*$ and let $(X,r)$ be the metric completion of $(\N,r)$. We endow the product space $X\times\R_+$ with the metric $d^{X\times\R_+}((x,v),(x',v'))=r(x,v)\vee|v-v'|$. The definition of $\hat\Dd^*$ yields $\lim_{n\to\infty}\dP^{X\times\R_+}(n^{-1}\sum_{i=1}^n\delta_{(i,v(i))},m)=0$
for a probability measure $m$ on $X\times\R_+$. As $\hat\psi(r,v)$ equals the isomorphy class of $(X,r,m)$, and as
$\hat\psi_n(\gamma_n(r,v))$ equals the isomorphy class of $(X,r,n^{-1}\sum_{i=1}^n\delta_{(i,v(i))})$ for each $n\in\N$,
the definition of the marked Gromov-Prohorov metric implies that
$\lim_{n\to\infty}\dMGP(\hat\psi(r,v),\hat\psi_n(\gamma_n(r,v)))=0$.

For $(r,v)\in\hat\Dd\setminus\hat\Dd^*$, the image $\hat\psi(r,v)$ is constant by definition.
Using Lemma~\ref{Sampl:rem:cont-psin} and Lemma \ref{lem:dstar-meas} below, we deduce measurability of $\hat\psi$.
\end{proof}
\begin{lem}
\label{lem:dstar-meas}
The subsets $\hat\Dd^*\subset\hat\Dd$ and $\Dd^*\subset\Dd$ are measurable.
\end{lem}
\begin{proof}
We represent $\hat\Dd^*$ by countable unions and intersections of measurable sets. The assertion on $\Dd$ follows along the same lines by removing the marks $v$.

For $(r,v)\in\hat\Dd$, let $(X,r)$ be the metric completion of $(X,r)$. We endow the product space $X\times\R_+$ with the metric $d^{X\times\R_+}((x,v),(x',v'))=r(x,v)\vee|v-v'|$ and define for $n\in\N$ the probability measures $m^n=n^{-1}\sum_{i=1}^n\delta_{(i,v(i))}$ on $X\times\R_+$.
The assertion $(r,v)\in\hat\Dd^*$ is equivalent to the assertion that $(m^n,n\in\N)$ is a Cauchy sequence with respect to the Prohorov metric on $X\times\R_+$.
Hence,
\[\hat\Dd^*=\bigcap_{\ep\in\Q\cap(0,\infty)}\bigcup_{k\in\N}\bigcap_{\ell\geq n\geq k}\hat\Dd_{\ep,\ell,n}\]
with
\[\hat\Dd_{\ep,\ell,n}:=\{(r,v)\in\hat\Dd:\dP(m^\ell,m^n)<\ep\},\]
where $m^\ell$ and $m^n$ are considered as probability measures on the (finite) support of $m^\ell$ in $X\times\R_+$.
Using the definition \reff{eq:Proh} of the Prohorov metric, we can write
\begin{align*}
\hat\Dd_{\ep,\ell,n}=\bigcap_{F\subset\{1,\ldots,\ell\}}\{\frac{1}{\ell}\sum_{i=1}^\ell&\I{\exists j\in F\text{ with }r(i,j)\vee|v(i)-v(j)|=0}\\
&<\frac{1}{n}\sum_{i=1}^n\I{\exists j\in F\text{ with }r(i,j)\vee|v(i)-v(j)|<\ep}+\ep\}.
\end{align*}
\end{proof}

\subsection{Resampling from marked metric measure spaces}
\label{sec:proofs:resampling}
We will use the statements from this section to prove assertions \ref{item:thm:repr:sampl} and \ref{item:thm:repr:det} of Theorem \ref{thm:repr}. In the end of this section, we also prove Proposition \ref{Sampl:prop:uniqueness} from Section \ref{Sampl:sec:repr}.

The following proposition can be compared with Lemma 8 of Vershik. We construct a marked metric measure space from a marked distance matrix. When we sample according to its marked distance matrix distribution, the assertion is that we arrive at a random variable that has the same distribution as the marked distance matrix with which we started.
Recall the functions $\hat\psi$, $\psi$ and the sets $\hat\Dd^*$, $\hat\Dd$ from Section \ref{Sampl:sec:mdm-mmm-new}.
\begin{prop}
\label{Sampl:prop:dm-mmm}
Let $(r,v)$ be an exchangeable random variable with values in $\hat\Dd^*$. Let $(r',v')$ be a random variable with values in $\Dd\times\R_+^\N$ and conditional distribution $\nu^{\hat\psi(r,v)}$ given $\hat\psi(r,v)$. Then $(r',v')$ and $(r,v)$ are equal in distribution. 
\end{prop}
\begin{rem}
\label{Sampl:rem:dm-mmm}
For an exchangeable random variable with values in $\Dd^*$ and a random variable $\rho'$ with conditional distribution $\nu^{\psi(\rho)}$ given $\psi(\rho)$, the random variables $\rho$ and $\rho'$ are equal in distribution. This follows from Proposition~\ref{Sampl:prop:dm-mmm}, we set $(r,v)=(\rho,0)$.
\end{rem}
\begin{proof}[Proof of Proposition~\ref{Sampl:prop:dm-mmm}]
Let $n\in\N$ and let $\phi:\R_+^{n^2}\times\R_+^n\to\R$ be bounded and continuous. Let $(X,r,m)$ be the representative of $\hat\psi(r,v)$ as in the definition of $\hat\psi$.
We have
\begin{align*}
\E\left[\phi\circ\gamma_n(r',v')\right]
& = \E\left[\int m^{\otimes n}(dx\;dv'')\phi((r(x(i),x(j)))_{i,j\in[n]},(v''(i))_{i\in[n]})\right]\\
& = \lim_{k\to\infty}\frac{1}{k}\sum_{\ell_1=1}^{k}
\cdots\frac{1}{k}\sum_{\ell_n=1}^{k}
\E\left[\phi((r(\ell_i,\ell_j))_{i,j\in[n]},(v(\ell_i))_{i\in[n]})\right]\\
& = \E\left[\phi\circ\gamma_n(r,v)\right].
\end{align*}
Here the assumption $(r,v)\in\hat\Dd^*$ ensures that $m$ is the weak limit of the uniform probability measures $\tfrac{1}{k}\sum_{\ell=1}^k\delta_{(\ell,v(\ell))}$ on $X\times\R_+$. This yields the second equality by dominated convergence. For the third equality, we use that summands where $\ell_1,\ldots, \ell_n$ are not pairwise distinct vanish in the limit, and that for all other summands, the expectation in the second line equals by exchangeability the expectation in the third line.
\end{proof}

In the next proposition, we start with a marked metric measure space and sample $(r,v)$ according to its marked distance matrix distribution. The marked metric measure space that we construct from any typical realization of $(r,v)$ turns out to be isomorphic to the marked metric measure space with which we started.
\begin{prop}
\label{Sampl:prop:mmm-mdm-mmm}
Let $\chi\in\hat\M$ and let $(r,v)$ be a $\Dd\times\R_+^\N$-valued random variable with distribution $\nu^{\chi}$. Then $(r,v)\in\hat\Dd^*$ a.\,s.\ and $\hat\psi(r,v)=\chi$ a.\,s.
\end{prop}
\begin{rem}
\label{Sampl:rem:mmm-mdm-mmm}
Proposition~\ref{Sampl:prop:mmm-mdm-mmm} is essentially Vershik's proof \cite{Vershik02}*{Theorem 4} of the Gromov reconstruction theorem (where metric measure spaces are considered, cf.\ also \cite{DGP11}*{Theorem 1} for marked metric measure spaces).
The present formulation focuses on the map $\hat\psi$ that will be used in the proofs of Theorems~\ref{thm:repr}\ref{item:thm:repr:det} and \ref{Sampl:thm:tv} below.
\end{rem}
\begin{rem}
\label{Sampl:rem:mm-dm-mm}
For $\chi\in\M$ and a $\Dd$-valued random variable $\rho$ with distribution $\nu^\chi$, Proposition~\ref{Sampl:prop:mmm-mdm-mmm} implies $\rho\in\Dd^*$ a.\,s.\ and $\psi(\rho)=\chi$ a.\,s. 
\end{rem}
\begin{proof}[Proof of Proposition~\ref{Sampl:prop:mmm-mdm-mmm}]
Let $(X',r',m')$ be a representative of $\chi$. W.\,l.\,o.\,g.\ we assume that the closed support of the probability measure $m'(\cdot\times\R_+)$ is the whole space $X'$, and that $(r,v)=((r'(x(i),x(j)))_{i,j\in\N},v)$ for an $m'$-iid sequence $(x,v)$. We denote by $(X,r)$ the completion of $(\N,r)$.
We endow $X'\times\R_+$ with the product metric $d^{X'\times\R_+}((x'_1,v'_1),(x'_2,v'_2))=r'(x'_1,x'_2)\vee|v'_1-v'_2|$, and analogously $X\times\R_+$.
As the sequence $(x(i))_{i\in\N}$ is a.\,s.\ dense in $X'$, the isometry that maps $x(i)$ to $i$ for all $i\in\N$ can a.\,s.\ be extended to a (surjective) isometry $\varphi$ from $X'$ to $X$. An isometry $\hat\varphi$ from $X'\times\R_+$ to $X\times\R_+$ is a.\,s.\ given by $(x,v')\mapsto(\varphi(x),v')$. By the Glivenko-Cantelli theorem, the probability measures $m'^n:=n^{-1}\sum_{i=1}^n\delta_{(x(i),v(i))}$ on $X'\times\R_+$ converge weakly to $m'$ a.\,s. As $\hat\varphi$ is continuous, the probability measures $m^n:=n^{-1}\sum_{i=1}^n\delta_{(i,v(i))}=\hat\varphi(m'^n)$ on $X\times\R_+$ converge weakly to $m:=\hat\varphi(m')$ a.\,s. This implies $(r,v)\in\Dd^*$ a.\,s.\ and that $\hat\psi(r,v)$ equals the isomorphy class of $(X,r,m)$ a.\,s. The second assertion follows as $\hat\varphi$ is a.\,s.\ a measure-preserving isometry from $X'\times\R_+$ to $X\times\R_+$, which implies that $(X',r',m')$ and $(X,r,m)$ have a.\,s.\ the same marked distance matrix distribution.
\end{proof}

\begin{rem}[Marked metric measure spaces and weighted real trees]
\label{Sampl:rem:mmm-tree-erg}
Let $\chi\in\hat\U$, and let $(r,v)$ be a $\hat\Uu$-valued random variable with the marked distance matrix distribution of $\chi$. By Proposition \ref{Sampl:prop:mmm-mdm-mmm}, we have $(r,v)\in\hat\Dd^*$ a.\,s., hence we can associate with any typical realization of $(r,v)$ a complete and separable weighted real tree $(\bar T,d,\mu)$ as in Remark \ref{Sampl:rem:psi-families}.
As in Proposition \ref{prop:mmm-ergodic}, the random marked distance matrix $(r,v)$ is ergodic with respect to the action of the group of finite permutations. This yields that the measure-preserving isometry class of the weighted real tree $(\bar T,d,\mu)$ is an a.\,s.\ constant random variable. Its typical realization can be associated with $\chi$.
\end{rem}
\begin{proof}[Proof of Proposition \ref{Sampl:prop:uniqueness}]
Let $(r,v)$ be a random variable with conditional distribution $\nu^\chi$ given $\chi$. Then we can assume $\rho=\alpha(r,v)$. Propositions~\ref{Sampl:prop:sampl-unique} and \ref{Sampl:prop:mmm-mdm-mmm} imply $\chi=\hat\psi\circ\beta(\rho)$ a.\,s. Hence, the distribution of $\rho$ determines the distribution of $\chi$ uniquely, which is the ``only if'' assertion.
The other direction clearly holds as the distribution of $\chi$ determines the distribution of $\rho$ uniquely.
\end{proof}

\subsection{Proof of the sampling representation}
\label{Sampl:sec:proof:sampling}
We give two proofs of Theorem \ref{thm:repr}\ref{item:thm:repr:sampling-psi} from Section \ref{Sampl:sec:repr} that build on a common part, namely statement \reff{Sampl:eq:Claim} below. The plan for the first proof is the following: We partition the completion of the tree $(T,d)$ associated with the semi-ultrametric $\rho$ (as in Remark \ref{Sampl:rem:ultrametr}) into small subsets. Into each of these subsets, we lay an atom whose mass is given by the asymptotic frequency of those integers that label the leaves of $T$ that are the endpoints of the external branches that begin in this subset. By exchangeability, these asymptotic frequencies exist, and \reff{Sampl:eq:Claim} yields that they add up to one.
We obtain an atomic probability measure on the product space of the metric completion of the tree and the mark space $\R_+$ by defining the $\R_+$-component as the distance to the top of the coalescent tree. Using the coupling characterization \reff{Sampl:eq:Proh-coupling} of the Prohorov metric, we show that this probability measure converges as the subsets become infinitely small, and that the limit measure coincides with the limit of the uniform measures in the definition of $\hat\Dd^*$.

As a slight difference to the description in the preceding paragraph, we will work with the space $(X,r)$ that corresponds to the completion of the space only of the starting vertices of the external branches, but we will occasionally recall the relation to the whole tree. We will use definitions also from Section \ref{Sampl:sec:dist-dec}.

\begin{proof}[Proof of Theorem \ref{thm:repr}]
Let $(r,v)=\beta(\rho)$. Then $v=\Upsilon(\rho)$ by definition of the map $\beta$.
Let $(X,r)$ be the metric completion of the semi-metric space $(\N,r)$. 

Let $\ep>0$. As the distribution of the random variable $v(i)$ has at most countably many atoms, there exists a deterministic sequence $0<h^{(\ep)}_1<h^{(\ep)}_2<\ldots$ that increases to infinity and that satisfies
\[h^{(\ep)}_1<\ep,\quad h^{(\ep)}_{n+1}-h^{(\ep)}_n<\ep,\]
and
\begin{equation}
\label{Sampl:eq:massless}
\P(v(i)=h^{(\ep)}_n)=0
\end{equation}
for all $i,j,n\in\N$. We set $h^{(\ep)}_0=0$ and we write $I^\ep_n=[h^{(\ep)}_{n-1},h^{(\ep)}_n)$ for $n\in\N$.

We define an equivalence relation $\sim^\ep$ on $\N$ such that two distinct integers $i,j$ are equivalent if and only if there exists $n\in\N$ with
\[v(i),v(j),\tfrac{1}{2}\rho(i,j)\in I^\ep_n.\]
To show transitivity, we consider $i,j,k\in\N$ with $i\neq k$, $i\sim^\ep j$, and $j\sim^\ep k$. Then there exists $n\in\N$ with $v(i),v(j),v(k),\rho(i,j)/2,\rho(j,k)/2\in I^\ep_n$. As
\[v(i)\leq\rho(i,k)/2\leq(\rho(i,j)\vee\rho(j,k))/2\]
by definition of $\Upsilon$ and ultrametricity, it follows that $i\sim^\ep k$.

Note that the definitions in Section \ref{Sampl:sec:dist-dec} imply
\begin{equation}
\label{Sampl:eq:r-ep}
r(i,j)=(\tfrac{1}{2}\rho(i,j)-v(i)+\tfrac{1}{2}\rho(i,j)-v(j))\I{i\neq j}<2\ep
\end{equation}
for $i\sim^\ep j$. (That is, in the context of Remark \ref{Sampl:rem:um-tree}, the starting points of external branches that end in leaves $(0,i)$, $(0,j)$ of $T$ with $i\sim^\ep j$ have distance smaller than $2\ep$.)

In the next two paragraphs, we prove the following claim: 
\begin{equation}
\label{Sampl:eq:Claim}
\text{A.\,s., the partition of $\N$ given by $\sim^\ep$ contains no singleton blocks.}
\end{equation}

For each $i,n\in\N$ the sequence
$(\I{v(j)\in I^\ep_n,\rho(i,j)/2\in I^\ep_n},j\in\N\setminus\{i\})$ is exchangeable. By the de Finetti theorem, it is conditionally iid.
Hence, on the event that there exists $j\in\N\setminus\{i\}$ with $v(j)\in I^\ep_n$ and $\rho(i,j)/2\in I^\ep_n$, there exists a.\,s.\ another (in fact, infinitely many) such $j$ in $\N\setminus\{i\}$.

For $j\in\N$, the definition of $\Upsilon$ and  condition~\reff{Sampl:eq:massless} imply the existence of (random) $n\in\N$ and $i\in\N\setminus\{j\}$ such that $v(j)\in I^\ep_n$ and $\rho(i,j)/2\in I^\ep_n$ a.\,s. As shown in the preceding paragraph, there exists a.\,s.\ an integer $k\in\N\setminus\{i,j\}$ with $v(k)\in I^\ep_n$ and $\rho(i,k)/2\in I^\ep_n$. From
\[v(k)\leq\rho(j,k)/2\leq(\rho(i,j)\vee\rho(i,k))/2,\]
it follows that $\rho(j,k)/2\in I^\ep_n$ a.\,s. This proves \reff{Sampl:eq:Claim}.

Now we show that the asymptotic frequencies exist and add up to one. For $A\subset\N$ and $k\in\N$, we denote the relative frequency by $|A|_k=k^{-1}\#(A\cap[k])$ and the asymptotic frequency by $|A|=\lim_{k\to\infty}|A|_k$, provided the limit exists. As the random partition given by $\sim^\ep$ is exchangeable, the asymptotic frequencies of its blocks exist a.\,s.\ by Kingman's correspondence.
Let $B^\ep(i)$ denote the equivalence class of $i\in\N$ with respect to $\sim^\ep$, and let
\[M^\ep=\{j\in\N: j=\min\,B^\ep(i)\text{ for some }i\in\N\}\]
be the set of minimal elements of the equivalence classes of $\sim^\ep$.
As the exchangeable partition given by $\sim^\ep$ has no singleton blocks a.\,s., it has proper frequencies by Kingman's correspondence, that is,
\[\sum_{i\in M^\ep}|B^\ep(i)|=1\quad\text{a.\,s.}\]
Consequently, on an event of probability $1$, a probability measure $m^\ep$ on the product sigma algebra on $X\times\R_+$ is given by
\begin{equation}
\label{Sampl:eq:def-mep}
m^\ep=\sum_{i\in M^\ep}|B^\ep(i)|\delta_{(i,v(i))}.
\end{equation}
(Into each of the subsets of $(X,r)$ given by $\sim^\ep$, the first component of the measure $m^\ep$ lays an atom with mass given by the asymptotic frequency of the integers that label the corresponding leaves in $T$.)

Let $\ep_1>\ep_2>\ldots>0$ with $\lim_{\ell\to\infty}\ep_\ell=0$.
For each $\ell\in\N$, we replace $\ep$ with $\ep_\ell$ everywhere in this proof until now, and we use the notations introduced so far.
We also assume that for $k\leq\ell$, the sequence $(h^{(\ep_k)}_n,n\in\N)$ is contained in $(h^{(\ep_\ell)}_n,n\in\N)$. That is, the partitions $\{I^{\ep_k}_n,n\in\N\}$ of $\R_+$ are nested.

For $k\leq \ell$ and $i\in M^{\ep_k}$, let $i_1,i_2,\ldots$ be an enumeration of $M^{\ep_\ell}\cap B^{\ep_k}(i)$. Then
\[B^{\ep_k}(i)=B^{\ep_\ell}(i_1)\uplus B^{\ep_\ell}(i_2)\uplus\ldots.\]
By Fatou's lemma and as a.\,s., the partition given by $\sim^{\ep_\ell}$ has proper frequencies, it follows that
\begin{equation*}
\label{eq:af-prop}
|B^{\ep_k}(i)|=|B^{\ep_\ell}(i_1)|+|B^{\ep_\ell}(i_2)|+\ldots\quad\text{a.\,s.}
\end{equation*}
Using equation \reff{Sampl:eq:def-mep}, we deduce
\begin{equation}
\label{eq:K-m-ell-k}
m^{\ep_k}\{(i,v(i))\}=\sum_{j\in B^{\ep_k}(i)}m^{\ep_\ell}\{(j,v(j))\}\quad\text{a.\,s.}
\end{equation}
A.\,s., a coupling of $m^{\ep_k}$ and $m^{\ep_\ell}$ is given by the probability measure
\begin{equation}
\label{eq:def-K}
K=\sum_{(i,j)}m^{\ep_\ell}\{(j,v(j))\}\delta_{((i,v(i)),(j,v(j)))}
\end{equation}
on $(X\times\R_+)^2$, where the sum is over all pairs $(i,j)$ with $i\in M^{\ep_k}$ and $j\in M^{\ep_\ell}\cap B^{\ep_k}(i)$.
Indeed, as equation~\reff{eq:K-m-ell-k} implies
\[K(\{(i,v(i))\}\times(X\times\R_+))=\sum_{j\in B^{\ep_k}(i)}m^{\ep_\ell}\{(j,v(j))\}=m^{\ep_k}\{(i,v(i))\}\quad\text{a.\,s.,}\]
$K$ is a.\,s.\ a coupling of $m^{\ep_k}$ and $m^{\ep_\ell}$.

In words, the probability measure $m^{\ep_\ell}$ can be obtained by splitting each atom of $m^{\ep_k}$ into fragments. Let us sample a point $(j,v(j))$ according to $m^{\ep_\ell}$, and let $(i,v(i))$ be the point such that the atom of $m^{\ep_\ell}$ at $(j,v(j))$ is one of the fragments of the atom of $m^{\ep_k}$ at $(i,v(i))$. Then the pair $((i,v(i)),(j,v(j)))$ has distribution $K$.

For every pair $(i,j)$ that appears in the sum in equation \reff{eq:def-K}, we have $i\sim^{\ep_k}j$, hence $\left|v(i)-v(j)\right|<\ep_k$ and $r(i,j)<2\ep_k$.
Hence, the coupling characterization of the Prohorov metric~\reff{Sampl:eq:Proh-coupling} yields
\begin{equation}
\label{Sampl:eq:sampl-dP-ep}
\dP^{X\times\R_+}(m^{\ep_k},m^{\ep_\ell})\leq 2\ep_k
\end{equation}
a.\,s.\ for all $k\leq \ell$, when $X\times\R_+$ is endowed with the product metric $d^{X\times\R_+}$ that is given by $d^{X\times\R_+}((x,v),(x',v'))=r(x,x')\vee|v-v'|$. As a consequence, on an event of probability $1$, the sequence $(m^{\ep_\ell},\ell\in\N)$ in the space of probability measures on the complete space $X\times\R_+$ is Cauchy, we denote its limit by $m$.

Consider for $n,\ell\in\N$ also the probability measure $m^{\ep_\ell}_n$ on $X\times\R_+$, given by
\[m^{\ep_\ell}_n
=\sum_{i\in M^{\ep_\ell}}|B^{\ep_\ell}(i)|_n\delta_{(i,v(i))}\quad\text{a.\,s.}\]
As there exists a.\,s.\ a coupling $K'$ of the probability measures $m^{\ep_\ell}_n$ and $m^{\ep_\ell}$ with
\[K'\{(y,y)\}=m^{\ep_\ell}_n\{y\}\wedge m^{\ep_\ell}\{y\}
\]
for all $y\in X\times\R_+$, the coupling characterization of the Prohorov metric~\reff{Sampl:eq:Proh-coupling} implies for each $k\in\N$
\begin{align*}
&\dP^{X\times\R_+}(m^{\ep_\ell}_n,m^{\ep_\ell})\\
&\leq K'\{(y,y')\in(X\times\R_+)^2:y\neq y'\}\\
&\leq m^{\ep_\ell}\{(j,v(j)):j\in M^{\ep_\ell},j>k\}
+K'\{((i,v(i)),(j,v(j))):i,j\in M^{\ep_\ell},i\neq j, j\leq k\}\\
&=\sum_{\substack{j\in M^{\ep_\ell}\\j>k}}|B^{\ep_\ell}(j)|
+\sum_{\substack{j\in M^{\ep_\ell}\\j\leq k}}\big||B^{\ep_\ell}(j)|_n-|B^{\ep_\ell}(j)|\big|\quad\text{a.\,s.}
\end{align*}
Letting first $n$ and then $k$ tend to infinity, we deduce
\begin{equation}
\label{Sampl:eq:sampl-conv-ep-k}
\lim_{n\to\infty}\dP^{X\times\R_+}(m^{\ep_\ell}_n,m^{\ep_\ell})=0\quad\text{a.\,s.}
\end{equation}

Moreover, we define for each $n\in\N$ the probability measure
\[m_n=n^{-1}\sum_{i=1}^n\delta_{(i,v(i))}\]
on $X\times\R_+$.
(The first component corresponds to a probability measure on the starting vertices of the external branches that end in one of the first $n$ leaves of $T$. These starting vertices are weighted according to the number of the corresponding leaves $i$, where we count leaves $i, j$ with $\rho(i,j)=0$ as separate leaves.)
By~\reff{Sampl:eq:sampl-dP-ep},
\begin{equation}
\label{Sampl:eq:sampl-dP-ep-k}
\dP^{X\times\R_+}(m^{\ep_\ell}_n,m_n)\leq 4\ep_\ell\quad\text{a.\,s.}
\end{equation}
for all $n,\ell\in\N$. From~\reff{Sampl:eq:sampl-dP-ep},~\reff{Sampl:eq:sampl-conv-ep-k}, and ~\reff{Sampl:eq:sampl-dP-ep-k}, we obtain
\begin{equation}
\label{Sampl:eq:conv-emp}
m=\text{w-}\lim_{n\to\infty}m_n\quad\text{a.\,s.}
\end{equation}
This shows assertion \ref{item:thm:repr:sampling-psi}.

Assertion \ref{item:thm:repr:sampl} follows from assertion \ref{item:thm:repr:sampling-psi} and Proposition \ref{Sampl:prop:dm-mmm}.
Proposition~\ref{Sampl:prop:mmm-mdm-mmm} implies assertion \ref{item:thm:repr:det}.
\end{proof}

The idea for the second proof of Theorem \ref{thm:repr}\ref{item:thm:repr:sampling-psi} is to construct directly by the de Finetti theorem a sampling measure on a subspace of the metric completion of the coalescent tree associated with $\rho$. To this aim, we fix by conditioning the closure of the subspace of the starting vertices of the external branches that end in the leaves labeled by the odd integers. By \reff{Sampl:eq:Claim}, this subspace contains a.\,s.\ the sequence of the starting vertices of the external branches associated with the even integers, and this sequence is exchangeable.
For a related result, see also Forman, Haulk, and Pitman \cite{FHP15}, where trees are embedded into $\ell_1$.

\begin{rem}
\label{Sampl:rem:EGW}
The second proof given below goes in a direction that is similar to the argument in Section 7 of \cite{EGW15} for the construction of the sampling measure $\mu$ on the real tree $\mathbf{S}=\Gamma(\mathbf{T})$.
That the equality $\Gamma(\mathbf{T})=\Gamma(\mathbf{T}^-)=\Gamma(\mathbf{T}^+)$ on p.\,268 in \cite{EGW15} holds for the embedding of $\Gamma(\mathbf{T}^-)$ and $\Gamma(\mathbf{T}^+)$ into $\Gamma(\mathbf{T})$ can be seen from \reff{Sampl:eq:Claim} as in the proof below as $\Gamma(\mathbf{T})$, $\Gamma(\mathbf{T}^-)$, and $\Gamma(\mathbf{T}^+)$ then correspond to $X$, $X_1$, and $X_2$ therein. The real tree $\Gamma(\mathbf{T}^-)$ can then be endowed with a measure like $X_1$ is endowed with $\mu^1$. Note that the starting vertices of the external branches and the subtree spanned by them are called the points of attachment and the core, respectively, in \cite{EGW15}.

We remark that the second last paragraph of the proof below shows that the isomorphy class of the weighted real tree $(\mathbf{S},\mu)$ is a.\,s.\ equal to $\psi(r)$ where $(r,v)=\beta(d)$ and $d$ is the exchangeable ultrametric on $\N$ from \cite{EGW15}*{Section 7}, which corresponds to $\rho$ below. This equality can also be deduced from Theorem~\ref{thm:repr}, Remark~\ref{Sampl:rem:psi-hatpsi}, as $\psi(r)$ is a.\,s.\ constant by the ergodicity assumption in \cite{EGW15}, and from the Gromov reconstruction theorem.
\end{rem}
\begin{proof}[Second proof of Theorem \ref{thm:repr}\ref{item:thm:repr:sampling-psi}]
Let $(r,v)=\beta(\rho)$. We construct the first component of the sampling measure, showing $r\in\Dd^*$ a.\,s.

We denote by $\N_1$ the odd, and by $\N_2$ the even integers. Let $(X,r)$ denote the metric completion of $(\N,r)$. A.\,s.\ by \reff{Sampl:eq:r-ep} and \reff{Sampl:eq:Claim}, there exists for each $i\in\N_2$ an integer $j\in\N_1$ with $r(i,j)<2\ep$. As $\ep$ can be chosen arbitrarily small, it follows that $i$ is a.\,s.\ contained in the closure $X_1$ of the subset $\N_1$ of $(X,r)$ a.\,s., hence $X_1=X$ a.\,s. (Recall from Remark \ref{Sampl:rem:um-tree} that $\N$ corresponds here to the set of starting vertices of the external branches in the coalescent tree $(T,d)$ associated with $\rho$.)

For $i\in\N_1$, let
\begin{equation*}
v^1(i)=\tfrac{1}{2}\inf_{j\in\N_1\setminus\{i\}}\rho(i,j).
\end{equation*}
(This is the length of the external branch that ends in the leaf $(0,i)$ in the subtree spanned by the leaves with labels in $\N_1$.)
By exchangeability of the sequence $(\rho(i,j):j\in\N\setminus\{i\})$ and by definition of $v=\Upsilon(\rho)$, it follows that $v^1(i)=v(i)$ a.\,s.
Let $\rho^1=(\rho(i,j))_{i,j\in\N_1}$ be the restriction of $\rho$ to $\N_1$. We define the random variable $r^1=(r^1(i,j))_{i,j\in\N_1}$ by
\[r^1(i,j)=(\rho^1(i,j)-v^1(i)-v^1(j))\I{i\neq j}.\]
By definition of $r$ in Section \ref{Sampl:sec:dist-dec}, it follows that $r^1=(r(i,j))_{i,j\in\N_1}$ a.\,s.

Let $\Lambda$ be a regular conditional distribution of $\rho$ given $\rho^1$. Then for a.\,a.\ $\rho^1$, under $\Lambda(\rho^1,\cdot)$, the complete and separable metric space $(X_1,r)$ is a.\,s.\ constant as $r^1$ is $\rho^1$-measurable.

Moreover, the sequence $2,4,6,\ldots$ of the even integers, viewed as a sequence in $(X_1,r)$, is exchangeable under $\Lambda(\rho^1,\cdot)$ for a.\,a.\ $\rho^1$. To see this, we use that the Borel sigma algebra on $(X_1,r)$ is generated by the balls around the elements of $\N_1\subset X_1$. Let $n\in\N$, and let $B_2,\ldots,B_{2n}$ be some finite intersections of such balls. Note that $\{2\in B_2,\ldots,2n\in B_{2n}\}$ can be written as an intersection of events of the form $\{\rho(i,j)<c\}$, where $i\in\N_2$, $j\in\N_1$ and $c\in(0,\infty)$. Using this, the uniqueness lemma, and the elementary fact that the conditional distribution of $\rho$ given its restriction $\rho^1$ is invariant under permutations that leave $\N_1$ fixed, we obtain the claimed exchangeability.

For this exchangeable sequence, the de Finetti theorem yields, $\Lambda(\rho^1,\cdot)$-a.\,s.\ for a.\,a.\ $\rho^1$, a sampling measure $\mu^1$ on $(X_1,r)$ that is the weak limit of the probability measures $\mu^1_n:=n^{-1}\sum_{i=1}^n\delta_{2i}$ on $(X_1,r)$.
By the same argument as above, also the closure $X_2$ of the subset $\N_2$ in $(X,r)$ equals $X$ a.\,s.
On the event of probability $1$ on which $\N_2$ is a dense subset of $X_2=X=X_1$, an isometry $\varphi:X_1\to X_2$ is given by $\varphi(i)=i$ for $i\in\N_2$. As also the weak limit of the image measures $\varphi(\mu^1_n)$ on $(X_2,r)$ exists a.\,s., we have shown $(r(2i,2j))_{i,j\in\N}\in\Dd^*$ a.\,s. This implies $r\in\Dd^*$ a.\,s.\ as $r$ and $(r(2i,2j))_{i,j\in\N}$ are equal in distribution by exchangeability of $r$.

That $(r,v)\in\hat\Dd^*$ can be shown analogously by considering the sequence $(i,v(i))_{i\in\N_2}$ in the space $X_1\times\R_+$ which we endow with the metric $d^{X_1\times\R_+}((x',v'),(x'',v''))=r(x',x'')\vee|v'-v''|$.
\end{proof}

\subsection{Proof of Theorem \ref{Sampl:thm:tv}}
\label{sec:proof:tv}
The following property is central in the proof of Theorem~\ref{Sampl:thm:tv}.
\begin{prop}
\label{Sampl:prop:eq-RP}
Let $t\in\R_+$, and let $f:\Uu\to\R$, $g:\hat\Uu\to\R$ be bounded measurable functions. If the assumptions of Theorem~\ref{Sampl:thm:tv} hold, then
\[\E[g(\beta(\rho_t))]=\E[\nu^{\hat\chi_t}g]\]
and
\[\E[f(\rho_t)]=\E[\xi_tf].\]
If the assumptions of Theorem~\ref{Sampl:thm:tv} hold and $\rho_t$ is a.\,s.\ dust-free, then
\[\E[f(\rho_t)]=\E[\nu^{\chi_t}f].\]
\end{prop}
\begin{proof}
This is immediate from Theorem~\ref{thm:repr}, the definition of $\xi_t$, and Corollary~\ref{cor:repr}.
\end{proof}
\begin{rem}
\label{Sampl:rem:intertwining}
In the context of Theorem~\ref{Sampl:thm:tv}\ref{Sampl:item:thm:tv:mmm}, let $(P_t,t\in\R_+)$ denote the semigroup on $\Mb(\hat\Uu)$ of the Markov process $(\beta(\rho_t),t\in\R_+)$, and let $(Q_t,t\in\R_+)$ denote the semigroup on $\Mb(\hat\U)$ of the Markov process $(\hat\chi_t,t\in\R_+)$.
Let $K$ denote the probability kernel from $\hat\U$ to $\hat\Uu$, given by $K(\chi,\cdot)=\nu^\chi$ for $\chi\in\hat\U$. Then Proposition~\ref{Sampl:prop:eq-RP} yields the intertwining relation
$Q_tK=KP_t$
which is condition (b) in \cite{RP81}*{Theorem 2}.
Many papers appeared on intertwining of Markov processes, a classical one is for instance \cite{CPY98}.
\end{rem}
\begin{proof}[Proof of Theorem~\ref{Sampl:thm:tv}]
We apply \cite{RP81}*{Theorem 2} to the semigroup of the Markov process $(\beta(\rho_t),t\in\R_+)$, the measurable map $\hat\psi:\hat\Uu\to\hat\U$, and the kernel $K$ from $\hat\U$ to $\hat\Uu$ given by $K(\chi,\cdot)=\nu^\chi$.
Clearly, Theorem 2 in \cite{RP81} also holds when the initial state $y$ therein is random. Then by Proposition~\ref{Sampl:prop:eq-RP},  condition (b) and the condition on the initial state in \cite{RP81}*{Theorem 2} are satisfied. Condition (a) in \cite{RP81}*{Theorem 2} follows from Proposition~\ref{Sampl:prop:mmm-mdm-mmm} as $f(\chi)=\nu^\chi(f\circ\hat\psi)$ for all $\chi\in\hat\U$ and all bounded measurable $f:\hat\U\to\R$. The Markov property of $(\hat\chi_t,t\in\R_+)$ now follows from \cite{RP81}*{Theorem 2}.

Now we give a proof that $(\hat\chi_t,t\in\R_+)$ solves the martingale problem $(\hat B,\hat\Pi)$.
If $(\beta(\rho_t),t\in\R_+)$ solves the martingale problem $(\hat A,\hat\C)$ in~\ref{Sampl:item:thm:tv:mmm}, then for all $\phi\in\hat\C$ with associated marked polynomial $\Phi$,
\begin{align}
0=&\E[\phi(\beta(\rho_t))-\phi(\beta(\rho_0))-\int_0^t\hat A\phi(\beta(\rho_u))du]\notag\\
=&\E[\nu^{\hat\chi_t}\phi]-\E[\nu^{\hat\chi_0}\phi]
-\int_0^t\E[\nu^{\hat\chi_u}(\hat A\phi)]du\notag\\
=&\E[\Phi(\hat\chi_t)-\Phi(\hat\chi_0)-\int_0^t \hat B\Phi(\hat\chi_u)du]\label{Sampl:eq:mp-chi-rv}
\end{align}
by Proposition~\ref{Sampl:prop:eq-RP}, Fubini, and the definitions $\Phi$ and $\hat B$. By the Markov property of $(\chi_s,s\in\R_+)$ and equation~\reff{Sampl:eq:mp-chi-rv}, it now follows for all $s\in\R_+$ and all $(\hat\chi_u,u\in[0,s])$-measurable events $A$ that
\[\E[\Phi(\hat\chi_{s+t})-\Phi(\hat\chi_s)-\int_s^{s+t}B\Phi(\hat\chi_u)du;A]
=0\]
which shows assertion~\ref{Sampl:item:thm:tv:mmm}.

The proof of~\ref{Sampl:item:thm:tv:dmd} is analogous, we apply \cite{RP81}*{Theorem 2} to the Markov process $(\rho_t,t\in\R_+)$, the measurable map $\Uu\to\UUerg$, $\rho\mapsto\alpha(\nu^{\hat\psi(\beta(\rho))})$, and the probability kernel from $\UUerg$ to $\Uu$ given by $(\nu,B)\mapsto\nu(B)$. In particular, condition (a) in \cite{RP81}*{Theorem 2} is satisfied by Propositions~\ref{Sampl:prop:mmm-mdm-mmm} and \ref{Sampl:prop:sampl-unique}, and by definition of $\UUerg$.

Also the proof of~\ref{Sampl:item:thm:tv:mm} is analogous. We apply \cite{RP81}*{Theorem 2} to the process $(\rho_t,t\in\R_+)$, the measurable map $\psi:\Uu\to\U$, and the probability kernel from $\U$ to $\Uu$ given by $(\chi,B)\mapsto\nu^\chi(B)$. We use the assumption that $\rho_t$ is a.\,s.\ dust-free
in the application of Proposition~\ref{Sampl:prop:eq-RP} and Remark~\ref{Sampl:rem:mm-dm-mm}.
\end{proof}

\section{Proofs related to the lookdown model}
This section contains the remaining proofs of the results from Sections \ref{Sampl:sec:ld-gen} -- \ref{Sampl:sec:semigroup}.
\subsection{Exchangeability in the lookdown model}
\label{sec:proof:exch}
To prove Proposition~\ref{Sampl:prop:rho-exch}, we show in Lemma~\ref{Sampl:lem:exch-pi} below that exchangeability of the genealogical distances is preserved in single reproduction events. Then we construct the genealogical distance matrix $\rho_t$ at time $t$, restricted to the first $n\in\N$ particles, from the initial state $\rho_0$ and the reproduction events before time $t$ that affect the genealogical distances between the first $n$ individuals. Here we use the description of the process $(\gamma_n(\rho_t),t\in\R_+)$ by its jumps and the evolution between the jumps from the end of Section \ref{sec:det-ld-gen}.

For $n\in\N$, we define the action of the group $S_n$ of permutations of $[n]$ on the set $\p_n$ of partitions of $[n]$, and on $\R^{n^2}$, respectively, by
\begin{equation}
\label{eq:action-p}
p(\pi)=\{p(B):B\in\pi\}\quad\text{and}\quad p(\rho)=(\rho(p(i),p(j)))_{i,j\in[n]}
\end{equation}
for each $p\in S_n$, $\pi\in\p_n$, $\rho\in\R^{n^2}$. A random variable with values for instance in $\p_n$ or in $\R^{n^2}$ is called exchangeable if its distribution is invariant under the action of $S_n$.
\begin{lem}
\label{Sampl:lem:exch-pi}
Let $n\in\N$, let $\pi$ be an exchangeable random partition of $[n]$, and let $\rho$ be an exchangeable random variable with values in $\R^{n^2}$. Assume that $\pi$ and $\rho$ are independent. Then the random variable $\pi(\rho)$ is exchangeable.
\end{lem}
Lemma~\ref{Sampl:lem:exch-pi} can be seen as a generalization of Lemma 4.3 of Bertoin \cite{Bertoin}.
\begin{proof}
Let $p\in S_n$. For each partition $\pi'\in\p_n$, the blocks of $\pi'$ are in one-to-one correspondence with the blocks of $p(\pi')$ via the bijection that maps a block $B\in\pi'$ to the block $p(B)\in p(\pi')$.
Also, the blocks of $\pi'$ are in one-to-one correspondence with the integers in $[n]$ that are the minimal elements of the blocks of $\pi'$. The same holds for the blocks of $p(\pi')$ and their minimal elements. It follows that the minimal elements of the blocks of $\pi'$ are in one-to-one correspondence with the minimal elements of the blocks of $p(\pi')$. We extend this one-to-one correspondence arbitrarily to a bijection from $[n]$ to itself which we denote by $f(\pi')$. This defines a map $f:\p_n\to S_n$ which satisfies
\begin{equation}
\label{Sampl:eq:exch-pi-alpha}
\pi'(i)=f(\pi')\left(p(\pi')(p(i))\right)
\end{equation}
for all $\pi'\in\p_n$ and $i\in[n]$. This equation holds as $\pi'(i)$, by its definition in Section \ref{sec:det-ld-gen}, is a minimal element of a block of $\pi'$ and as $p(\pi')(p(i))$ is the minimal element of the corresponding block of $\pi'$. By equation~\reff{eq:action-p} and the definition~\reff{Sampl:eq:pn-Un} of the transformation on $\R^{n^2}$ associated with each element of $\p_n$, equation~\reff{Sampl:eq:exch-pi-alpha} implies
\begin{equation}
\pi'(\rho')=p\left(p(\pi')\left(f(\pi')(\rho')\right)\right)
\label{Sampl:eq:exch-pi-rrho}
\end{equation}
for all $\pi'\in\p_n$ and $\rho'\in\R^{n^2}$.

By assumption, $p(\pi)$ and $\pi$ are equal in distribution. As the distribution of $f(\pi')(\rho)$ is the same for all $\pi'\in\p_n$, namely equal to the distribution of $\rho$, it follows that $f(\pi)(\rho)$ and $\pi$ are independent, and that $f(\pi)(\rho)$ is equal in distribution to $\rho$. This implies that $\pi(\rho)$ and $p(\pi)\left(f(\pi)(\rho)\right)$ are equal in distribution as also $\rho$ and $\pi$ are independent by assumption. By equation~\reff{Sampl:eq:exch-pi-rrho}, it follows that $\pi(\rho)$ and $p^{-1}(\pi(\rho))$ are equal in distribution, which yields the assertion.
\end{proof}
\begin{proof}[Proof of Proposition~\ref{Sampl:prop:rho-exch}]
Let $n\in\N$. For $s\in\R_+$, we define the map
\[\lambda_s:\R^{n^2}\to\R^{n^2},\quad \rho'\mapsto\rho'+\underline{\underline{2}}_n s,\]
where $\underline{\underline{2}}_n=2(\I{i\neq j})_{i,j\in[n]}$. We will use the map $\lambda_s$ to account for the linear growth of the genealogical distances between reproduction events.

On an event of probability $1$, let $(t_1,\pi_1), (t_2,\pi_2),\ldots$ be the points of $\eta$ in $(0,t]\times\p^n$. Let $L=\eta((0,t]\times\p_n)$. Conditionally given $(t_1,\ldots,t_L)$, the partitions $\pi_1,\ldots,\pi_L$ are independent and for each $k\in\N$, the restriction $\gamma_n(\pi_k)$ is exchangeable. This follows from the properties of Poisson random measures and the definition of $\eta$. From the description around equation \reff{Sampl:eq:jumps-pi-rhot}, we have
\[\gamma_{n}(\rho_t)=\lambda_{t-t_L}\circ\gamma_{n}(\pi_L)\circ\lambda_{t_L-t_{L-1}}\circ\ldots\circ\gamma_{n}(\pi_1)\circ\lambda_{t_1}(\gamma_{n}(\rho_0))\quad\text{a.\,s.}\]
on the event $\{L\geq 1\}$, and
$\gamma_n(\rho_t)=\lambda_t(\gamma_n(\rho_0))$ a.\,s.\ on $\{L=0\}$. By assumption, $\gamma_n(\rho_0)$ is exchangeable, and Lemma~\ref{Sampl:lem:exch-pi} implies that $\gamma_n(\rho_t)$ is exchangeable.
The assertion follows as $n\in\N$ was arbitrary and as the distribution of $\rho_t$ is determined by the distributions of the restrictions $\gamma_n(\rho_t)$.
\end{proof}

\subsection{Equality of decompositions}
\label{sec:proof:rv-R}
To prove Proposition~\ref{Sampl:prop:rv-R}, we use the following lemma. Its meaning is that if the ancestral lineage of an individual $i$ at time $t$ can be traced back until a most recent reproduction event on that lineage, then there exists a.\,s.\ another individual $k$ at time $t$ that descends from this reproduction event.
\begin{lem}
\label{Sampl:lem:v-rho}
Assume that $\Xi$ is not dust-free. Let $t\in(0,\infty)$ and $i\in\N$. Then a.\,s.\ on the event $\{v_t(i)<t\}$, there exists an integer $k\in\N\setminus\{i\}$ with $v_t(i)=\tfrac{1}{2}\rho_t(i,k)$.
\end{lem}
\begin{proof}
Recall the process $(\Pi^{(t)}_s,s\in\R_+)$ from Remark~\ref{Sampl:rem:Xi-coal}.
We work on the intersection of $\{v_t(i)<t\}$ with the event of probability $1$ on which condition~\reff{Sampl:eq:ass-eta-hat} is satisfied, $v_t(i)>0$, and for each $s\in(0,t)\cap\Q$, the partition $\Pi^{(t)}_s$ contains infinitely many blocks if it contains singletons. The latter event indeed has probability $1$ by Kingman's correspondence and as $t$ is a.\,s.\ not the time of a reproduction event.

At time $t-v_i(t)$, a reproduction event occurs that is encoded by a partition in which the block that contains $A_{t-v_t(i)}(t,i)$ contains some other element $j$. This follows from the definition of $v_t(i)$ in Section~\ref{Sampl:sec:dec-det} and as
$\eta((0,t]\times\hat\p^i)<\infty$ by condition~\reff{Sampl:eq:ass-eta-hat} which means that the reproduction events in which particles on levels not larger than $i$ reproduce do not accumulate. 

Moreover, by condition \reff{Sampl:eq:ass-eta-hat}, there exists a time $s\in(t-v_i(t),t)\cap\Q$ with
$\eta((t-v_t(i),s]\times\hat\p_j)=0$, which implies that the particle on level $j$ at time $t-v_t(i)$ is still on level $j$ at time $s$.

By definition of $v_t(i)$, the partition $\Pi^{(t)}_s$ contains the singleton block $\{A_s(t,i)\}$, hence $\Pi^{(t)}_s$ has infinitely many blocks. This means that infinitely many particles at time $s$ survive until time $t$. Remark~\ref{Sampl:rem:non-crossing} implies that all particles at time $s$ survive until time $t$. Therefore, the particle that was on level $j$ at the times $t-v_t(i)$ and $s$ is on some level $k$ at time $t$. The most recent common ancestor of the particles on levels $i$ and $k$ at time $t$ lives at time $t-v_t(i)$, hence $\tfrac{1}{2}\rho_t(i,k)=v_t(i)$.
\end{proof}
\begin{proof}[Proof of Proposition~\ref{Sampl:prop:rv-R}]
Let $t\in(0,\infty)$ and $i\in\N$. We have to show that $v_t(i)=\Upsilon(\rho_t)(i)$ a.\,s.

From the definitions of the reproduction events in Section~\ref{sec:det-ld-gen} and of the quantity $v_t(i)$ in Section~\ref{Sampl:sec:dec-det}, it follows that for each $s\in(t-v_t(i)\wedge t,t]$, only the particle on level $i$ at time $t$ descends from the particle on level $A_s(t,i)$ at time $s$. The definitions of $\Upsilon$ in Section~\ref{Sampl:sec:dist-dec} and of $\rho_t$ in Section~\ref{sec:det-ld-gen} imply
$0\leq v_t(i)\wedge t\leq\Upsilon(\rho_t)(i)\wedge t$.

In the case that $\Xi$ is dust-free, we have $\Upsilon(\rho_t)=0$ a.\,s.\ by Proposition~\ref{Sampl:prop:rho-df}, hence also $v_t(i)=0$ a.\,s.

Now we assume that $\Xi$ is not dust-free.
Lemma~\ref{Sampl:lem:v-rho} yields $\Upsilon(\rho_t)(i)\leq v_t(i)$ a.\,s.\ on the event $\{v_t(i)<t\}$.

We claim that on the event $\{v_t(i)\geq t\}$, all individuals at time $0$ have descendants at time $t$. This can be seen as follows: For each $s\in(0,t)$, the exchangeable partition $\Pi^{(t)}_{s}$, defined in Remark \ref{Sampl:rem:Xi-coal}, contains the singleton block $\{A_s(t,i)\}$ on the event $\{v_t(i)\geq t\}$. By Kingman's correspondence, it follows that $\Pi^{(t)}_s$ has infinitely many blocks a.\,s.\ on $\{v_t(i)\geq t\}$. Using Remark \ref{Sampl:rem:non-crossing}, we deduce that a.\,s.\ on $\{v_t(i)\geq t\}$, all particles at any time $s\in(0,t)$ survive until time $t$. As condition \reff{Sampl:eq:ass-eta} is a.\,s.\ satisfied, each individual at time $0$ retains its level for a positive time a.\,s., whence all individuals at time $0$ survive until time $t$ a.\,s.

Hence, as $\Upsilon(\rho_0)=v_0$ by assumption~\reff{Sampl:eq:ass-r0v0},
\begin{align*}
&v_t(i)=t+v_0(A_0(t,i))=
t+\tfrac{1}{2}\inf_{j\in\N\setminus\{A_0(t,i)\}}\rho_0(A_0(t,i),j)\\
&= \tfrac{1}{2}\inf_{j\in\N\setminus\{i\}}\rho_t(i,j)=\Upsilon(\rho_t)(i)\quad\text{a.\,s.\ on }\{v_t(i)\geq t\}.
\end{align*}
\end{proof}

\subsection{Uniqueness for the martingale problems for tree-valued Fleming-Viot processes}
\label{sec:proof-uniqueness}
\begin{proof}[Proof of Proposition \ref{Sampl:prop:mp-unique-mm}]
We consider the martingale problem $(B,\Pi)$, the proofs for the other martingale problems are analogous. It remains to show uniqueness of the solution discussed in Section \ref{Sampl:sec:TVFV}. We use a function-valued dual process. This method is applied in the context of tree-valued Fleming-Viot processes in \cite{DGP12}, another dual process is used in \cite{GPW13}. We fix $n\in\N$ and work with a dual process with state space $\C_n$.
With each element $\pi$ of $\p_n$, we also associate a transformation $\C_n\to\C_n$, which we also denote by $\pi$, by
\[\pi(\phi)(\rho)=\phi(\pi(\rho)),\quad\rho\in\R^{n^2},\phi\in\C_n.\]
Here $\pi(\rho)$ is defined in equation \reff{Sampl:eq:pn-Un}.
We define an independent process $(\phi_t, t\in\R_+)$ as the Markov process with càdlàg paths in $\C_n$ such that
\begin{itemize}
\item for each $\pi\in\p_n\setminus\{\mathbf{0}_n\}$ at rate $\lambda_{\pi}$, the process jumps from $\phi$ to $\pi(\phi)$,
\item and between these jumps, the process evolves deterministically according to
\[\phi_{t+s}(\rho)=\phi_t(\rho+\underline{\underline{2}}_ns)\]
for $s,t\in\R_+$ and $\rho\in\R^{n^2}$, where $\underline{\underline{2}}_n=2(\I{i\neq j})_{i,j\in[n]}$.
\end{itemize}

The process $(\phi_t,t\in\R_+)$ solves the martingale problem $(B^\downarrow,\D)$, where
\[\D=\{\C_n\to\R,\phi\mapsto\nu^{\chi'}\phi: \chi'\in\U\}\]
and an operator $B^\downarrow$ with domain $\D$ is defined by $B^\downarrow=B^\downarrow_{\rm coal}+B^{\downarrow}_{\rm shrink}$,
\[
B^\downarrow_{\rm coal}\nu^{\chi'} (\phi) 
= \sum_{\pi\in\p_n\setminus\{\mathbf{0}_n\}} \lambda_{\pi} \left( \nu^{\chi'}(\pi(\phi)) -\nu^{\chi'}\phi\right)\]
and
\[
B^\downarrow_{\rm shrink} \nu^{\chi'}(\phi) 
= \nu^{\chi'} \langle \nabla \phi, \underline{\underline{2}} \rangle\]
for $\phi\in\C_n$ and $\chi'\in\U$. Here we use the notation $\langle \nabla \phi,\underline{\underline{2}} \rangle$ from equation \reff{eq:nabla-growth}.

From this definition, we have
$B(\nu^\cdot\phi)(\chi') = B^\downarrow \nu^{\chi'}(\phi)$ for all $\phi\in\C_n$ and $\chi'\in\U$, where $\nu^\cdot\phi$ is the polynomial associated with $\phi$.
For all $t\in\R_+$ and all polynomials $\Phi\in\Pi$ of degree at most $n$, it follows from Theorem 4.4.11 in \cite{EK86} that $E[\Phi(\tilde\chi_t)]$ is equal for all solutions $((\tilde\chi_t,t\in\R_+);P)$ of the martingale problem $(B,\Pi)$ with initial state $\chi_0$. As $n\in\N$ was arbitrary and the space $\Pi$ of polynomials is separating, the uniqueness assertion follows from Theorem 4.4.2 in \cite{EK86}.
\end{proof}

\subsection{Proof of Lemma \ref{Sampl:lem:Pi-preserve}}
\label{sec:proof:Pi-preserve}
Using the lookdown construction, we show that the semigroup of an $\hat\U$-valued $\Xi$-Fleming-Viot process preserves the set of marked polynomials.
\begin{proof}[Proof of Lemma \ref{Sampl:lem:Pi-preserve}]
Let $\Nn$ denote the space of simple point measures on $(0,\infty)\times\p$.
Let $t\in\R_+$ and $n\in\N$. Note that in the construction in Sections~\ref{sec:det-ld-gen} and \ref{Sampl:sec:dec-det}, the
restriction $\gamma_n(r_t,v_t)$ depends only on the simple point measure $\eta$ and the restriction $\gamma_n(r_0,v_0)$ of the initial state. We may thus define the function $g_n:\R^{n^2}\times\R^n\times\Nn\to\R^{n^2}\times\R^n$ that maps the restriction $\gamma_n(r_0,v_0)$ of the initial state and the point measure $\eta$ to $\gamma_n(r_t,v_t)$. Note that when the simple point measure is fixed, $g_n$ is a differentiable function on $\R^{n^2}\times\R^n$ with bounded uniformly continuous derivative.

Let $\phi\in\hat\C_n$. We define the function
\[f:\R^{n^2}\times\R^n\to\R,\quad
(r,v)\mapsto\int\P(\eta\in d\eta')\phi\circ g_n((r,v),\eta'),\]
where $\eta$ is now the Poisson random measure from Section~\ref{Sampl:sec:Xi-ld}.
By dominated convergence and the mean value theorem, also the function $f$ is differentiable with bounded uniformly continuous derivative, and we obtain that $f\in\hat\C_n$.

Let $\Phi$ be the marked polynomial associated with $\phi$. For $\chi\in\hat\U$, let $(r_0,v_0)$ be a random variable with the marked distance matrix distribution of $\chi$, and let $(r_t,v_t)$ be defined from $(r_0,v_0)$ and the independent Poisson random measure $\eta$ as in Section~\ref{Sampl:sec:dec-stoch}. From Propositions~\ref{Sampl:prop:rv-R} and \ref{Sampl:prop:eq-RP}, and as we may assume that the $\hat\U$-valued $\Xi$-Fleming-Viot process $(\hat\chi_s,s\in\R_+)$ from Section \ref{Sampl:sec:semigroup} satisfies $\hat\chi_t=\hat\psi(r_t,v_t)$ a.\,s., we obtain that
$\E_\chi[\Phi(\hat\chi_t)]=\nu^\chi f$
for all $\chi\in\hat\U$.
Hence, $\chi\mapsto\E_\chi[\Phi(\hat\chi_t)]$ is in $\hat\Pi$.
\end{proof}

\section{Construction from the flow of bridges}
\label{Sampl:sec:bridges}
In this section, we construct a $\UUerg$-valued $\Xi$-Fleming-Viot process from the dual flow of bridges of Bertoin and Le Gall \cite{BLG03}.

A random non-decreasing right-continuous function $\tilde F:[0,1]\to[0,1]$ with exchangeable increments and $\tilde F(0)=0$, $\tilde F(1)=1$ is called a bridge. We view a bridge as a random variable with values in the space of càdlàg paths $[0,1]\to[0,1]$ which we endow with the Skorohod metric.
The dual flow of bridges is a collection $F=( F_{s,t},s<t)$ of bridges that satisfies the following properties (see \cite{BLG03}*{Section 5.1}):
\begin{enumerate}[label=(\roman{*}),ref=(\roman{*})]
\item\label{Sampl:item:fob:cocycle} For every $s<t<u$, $ F_{t,u}\circ F_{s,t}= F_{s,u}$ a.\,s.
\item\label{Sampl:item:fob:ind} The law of $ F_{s,t}$ depends only on $t-s$. For $s_1<s_2<\ldots<s_n$, the bridges $ F_{s_1,s_2},  F_{s_2,s_3},\ldots, F_{s_{n-1},s_n}$ are independent.
\item $ F_{0,0}$ is the identity function. For every $x\in[0,1]$, the random variable $ F_{0,t}(x)$ converges to $x$ in probability as $t$ decreases to zero.
\end{enumerate}
For each $s<t$, it is also assumed that $F_{s,t}$ is a.\,s.\ not the identity function.

The interpretation is that the individuals of a continuous population are represented by the elements of the interval $[0,1]$. For each $s\leq t$, the individuals in a subinterval $(x_1,x_2]$ at time $s$ have descendants at time $t$ that are a.\,s.\ the elements of $( F_{s,t}(x_1), F_{s,t}(x_2)]$, see \cite{BLG06}.

In \cite{BLG03}*{Section 3}, Kingman's correspondence is extended so as to represent distributions of $\Xi$-coalescents in terms of sampling from flows of bridges.
Let $F$ be a dual flow of bridges, and let $V=(V_i,i\in\N)$ be an iid sequence of uniform $[0,1]$-valued random variables, independent of $ F$. This iid sequence is interpreted as a sequence of random samples from the population at some time $t\in\R$. For each $s\in\R_+$, a partition $\tilde\pi^{(t)}_s$ is defined such that any integers $i,j\in\N$ are in the same block of $\tilde\pi^{(t)}_s$ if and only if $ F_{t-s,t}^{-1}(V_i)= F_{t-s,t}^{-1}(V_j)$ which means that these samples have the same ancestor at time $t-s$.
Here we set $f^{-1}(t)=\inf\{s\in[0,1]:f(s)>t\text{ or }s=1\}$ for $t\in[0,1]$ and a càdlàg function $f:[0,1]\to[0,1]$.
In \cite{BLG03}*{Theorem 1}, it is shown that the partition-valued process $(\tilde\pi^{(t)}_s,s\in\R_+)$ obtained in this way is a version of a $\Xi$-coalescent of Schweinsberg \cite{Schw00}.

For each $t\in\R$, there exists an event of probability $1$ on which for all $s\leq s'\in\Q_+$, the partition $\tilde\pi^{(t)}_{s'}$ can be obtained by merging blocks of the partition $\tilde\pi^{(t)}_s$.
We can thus define a.\,s.\ an ultrametric $\tilde\rho_t$ by
\[\tilde\rho_t(i,j)=2\inf\{s\in \Q_+:\text{ $i$ and $j$ are in the same block of }\tilde\pi^{(t)}_s\}.\]
The assumption that for each $r<s$, the bridge $F_{r,s}$ is a.\,s.\ not the identity function  implies that the infimum in the definition of $\tilde\rho_t(i,j)$ is a.\,s.\ not over the empty set.

Moreover, we define a.\,s.\ a random variable $\tilde\xi_t$ with values in the space $(\UU,\dP)$ of exchangeable distributions on $\Uu$ such that $\tilde\xi_t$ is a regular conditional distribution of $\tilde\rho_t$ given the collection of bridges $(F_{t-s,t},s\in \Q_+)$. For the existence of this regular conditional distribution, see e.\,g.\ \cite{Kal02}*{Theorem 6.3}.

Analogously to Sections \ref{Sampl:sec:TVFV-mmm} and \ref{Sampl:sec:TVFV-dm}, for a finite measure $\Xi$ on $\Delta$, a stationary $\UUerg$-valued $\Xi$-Fleming-Viot process $(\xi_t,t\in\R)$ is given by
$\xi_t=\alpha(\nu^{\hat\psi\circ\beta(\bar\rho_t)})$, where $(\bar\rho_t,t\in\R)$ is defined as in Section \ref{Sampl:sec:equil}.
We note that a stationary $\UUerg$-valued $\Xi$-Fleming-Viot process can be read off from the dual flow of bridges:
\begin{thm}
\label{Sampl:thm:fob}
There exists a finite measure $\Xi$ on $\Delta$ such that the process
$(\tilde\xi_t,t\in\R)$ is a version of a stationary $\UUerg$-valued $\Xi$-Fleming-Viot process.
\end{thm}
For the proof of Theorem \ref{Sampl:thm:fob}, we show that $(\tilde\xi_t,t\in\R)$ is a Markov process and has the transition kernel of a $\UUerg$-valued $\Xi$-Fleming-Viot process. In the following, we fix $u\in\R_+$.

First, we define for each finite measure $\Xi$ on $\Delta$ a probability kernel $\Lambda_\Xi$ from $\UU$ to $\Uu$ such that for each $\xi\in\UU$, the distribution $\Lambda_\Xi(\xi,\cdot)$ is the distribution of a random variable $\rho$ which we define as follows. Let $\rho'$ be a random variable with distribution $\xi$. Let $\rho''$ be an independent $\Uu$-valued random variable that is distributed as the random ultrametric associated with a $\Xi$-coalescent. That is, $\rho''$ shall be distributed as the random variable $\bar\rho_u$ mentioned above, cf.\ Remark \ref{Sampl:rem:Xi-coal}.
We define a partition $\pi$ of $\N$ such that $i$ and $j$ are in the same block of $\pi$ if and only if $\rho''(i,j)<2u$. Let $B_1(\pi),B_2(\pi),\ldots$ be the blocks of $\pi$, ordered increasingly according to their smallest element. For $i\in\N$, let $A(i)$ be the integer $j$ such that $i\in B_j(\pi)$.

Then we set for $i,j\in\N$
\[\rho(i,j)=\left\{\begin{aligned}
&\rho''(i,j)\wedge(2u)\quad\text{if }\rho''(i,j)<2u\\
&2u+\rho'(A(i),A(j))\quad\text{else.}
\end{aligned}\right.\]

In the following, we also fix $t\in\R$.
\begin{rem}
\label{Sampl:rem:tvfv-fob-trans}
Let $(\bar\rho_s,s\in\R)$ be defined as in Section \ref{Sampl:sec:equil} from a measure $\Xi$, and let $\xi_t=\alpha(\nu^{\hat\psi\circ\beta(\bar\rho_t)})$. Note that $\Lambda_\Xi$ is a regular conditional distribution of $\bar\rho_{t+u}$ given $\xi_t$. This follows as $\bar\rho_{t+u}\wedge(2u)$ is independent of $\xi_t$ and $\bar\rho_t$, as $\bar\rho_{t+u}(i,j)=2u+\bar\rho_t(A_t(t+u,i),A_t(t+u,j))$ for $i,j\in\N$ with $\bar\rho_{t+u}(i,j)\geq 2u$, as $A_t(t+u,i)$ can be read off from $\bar\rho_{t+u}\wedge(2u)$ like $A(i)$ can be read off from $\rho''$ in the definition of $\Lambda_\Xi$, and as $\xi_t$ is a regular conditional distribution of $\bar\rho_t$ given $\xi_t$ by Remark \ref{rem:cond-repr}.
\end{rem}

\begin{lem}
\label{Sampl:lem:fob-kernel}
There exists a finite measure $\Xi$ on $\Delta$ such that $\Lambda_\Xi$ is a regular conditional distribution of $\tilde\rho_{t+u}$ given $\tilde\xi_t$.
Moreover, $\tilde\rho_{t+u}$ is conditionally independent of
$(\tilde\xi_s,s\leq t)$ given $\tilde\xi_t$.
\end{lem}
\begin{proof}
We claim that given the collection of bridges $(F_{r,s}:r<s\leq t)$, the random variable $\tilde\rho_{t+u}$ has conditional distribution $\Lambda_\Xi(\tilde\xi_t,\cdot)$ for some finite measure $\Xi$ on $\Delta$. By construction of $\tilde\xi_s$, this claim implies both assertions of the lemma.

We assume that the coalescent process $(\tilde\pi^{(t+u)}_s,s\in\R_+)$ and the associated ultrametric $\tilde\rho_{t+u}$ are constructed as above from $F$ and a sequence $(V_i,i\in\N)$ of independent uniformly distributed $[0,1]$-valued random variables that is independent of $F$. By \cite{BLG03}*{Theorem 1}, there exists a finite measure $\Xi$ on $\Delta$ such that $(\tilde\pi^{(t+u)}_s,s\in\R_+)$ is a version of a $\Xi$-coalescent.

Let $B_1(\tilde\pi^{(t+u)}_u),B_2(\tilde\pi^{(t+u)}_u),\ldots$ be the blocks of $\tilde\pi^{(t+u)}_u$ in increasing order according to their respective smallest element. For $i\in\N$, we define $\tilde A(i)=j$ where $j$ is the integer such that $i\in B_j(\tilde\pi^{(t+u)}_u)$.
We define a sequence $V'=(V'_i,i\in\N)$ analogously to equation (3) of \cite{BLG03}: For $i\in\N$ with $i\leq\#\tilde\pi^{(t+u)}_u$, we set $V'_i=F^{-1}_{t,t+u}(V_j)$, where $j$ is any element of $B_i(\tilde\pi^{(t+u)}_u)$. If the number of blocks $\#\tilde\pi^{(t+u)}_u$ is finite, we extend the sequence $(V'_i,i\leq \#\tilde\pi^{(t+u)}_u)$ to $(V'_i,i\in\N)$ using an independent sequence of independent uniform random variables on $[0,1]$.

Let $\ell\in\N$ and $0\leq u_1\leq u_2\leq\ldots\leq u_\ell=u$. 
Repeated application of \cite{BLG03}*{Lemma 2} to the bridges
$F_{t+u-u_1,t+u},\ldots,F_{t+u-u_\ell,t+u-u_{\ell-1}}$
(similarly to \cite{BLG03}*{Corollary 1}) yields that $V'$ is a sequence of independent $[0,1]$-valued uniformly distributed random variables that is also independent of $\tilde\pi^{(t+u)}_{u_1},\ldots,\tilde\pi^{(t+u)}_{u_\ell}$. By construction and property \ref{Sampl:item:fob:ind} of the dual flow of bridges, $V'$ and $\tilde\pi^{(t+u)}_{u_1},\ldots,\tilde\pi^{(t+u)}_{u_\ell}$ are also independent of $(F_{r,s}:r<s\leq t)$.

We define $\tilde\rho_t$ from $F$ and the sequence $V'$. Then $\tilde\rho_{t}$ is conditionally independent of $\tilde\rho_{t+u}\wedge (2u)$ given the collection of bridges $(F_{r,s}:r<s\leq t)$.
This follows from the above by the uniqueness lemma as for $i,j\in\N$ and $m=1,\ldots,\ell$,
$\{\tilde\rho_{t+u}(i,j)\leq u_m\}$ is, up to null events, the event that $i$ and $j$ are in the same block of $\tilde\pi^{(t+u)}_{u_m}$.

By construction, $\tilde\xi_t$ is a conditional distribution of $\tilde\rho_t$ given $(F_{r,s}:r<s\leq t)$. We also define the coalescent process $(\tilde\pi^{(t)}_s,s\in\R_+)$ from $V'$ and $F$. Then $\tilde\rho_t$ is the associated ultrametric. For $i,j\in\N$ and $s\in\R_+$, the following events are equal up to null events:
\begin{align*}
&\{\tilde\rho_{t+u}(i,j)\leq 2(u+s)\}=\{i,j \text{ are in the same block of }\tilde\pi^{(t+u)}_{u+s}\}\\
&=\{\tilde A(i),\tilde A(j) \text{ are in the same block of }\tilde\pi^{(t)}_{s}\}
=\{\tilde\rho_{t}(\tilde A(i),\tilde A(j))\leq 2s\}.
\end{align*}
For the equality up to null events of the second and the third event, we use the definition of $V'$ and property \ref{Sampl:item:fob:cocycle} of the dual flow of bridges.
It follows that a.\,s.,
\[\tilde\rho_{t+u}(i,j)=\left\{\begin{aligned}
&\tilde\rho_{t+u}(i,j)\wedge(2u)\quad\text{if }\tilde\rho_{t+u}(i,j)<2u\\
&2u+\tilde\rho_t(\tilde A(i),\tilde A(j))\quad\text{else.}
\end{aligned}\right.\]
The claim follows as $\tilde A(i)$ can a.\,s.\ be read off from $\tilde\rho_{t+u}\wedge(2u)$ in the same way as $A(i)$ is read off from $\rho''$ in the definition of $\Lambda_\Xi$.
\end{proof}

To deduce Theorem \ref{Sampl:thm:fob}, we use that $\tilde\xi_{t+u}\in\UUerg$ a.\,s.
\begin{proof}[Proof of Theorem \ref{Sampl:thm:fob}]
Let $t\in\R$ and $u>0$.
By Proposition \ref{prop:UUerg} and as the sequence $V$ in the definition of $\tilde\rho_t$ is iid, $\tilde\xi_t\in\UUerg$ a.\,s. That $\tilde\xi_t$ is concentrated on the ergodic distributions can be seen directly or by an application of e.\,g.\ \cite{Kal05}*{Lemma 7.35}. By construction, $\tilde\xi_{t+u}$ is a regular conditional distribution of $\tilde\rho_{t+u}$ given $\tilde\xi_{t+u}$. By Lemma \ref{Sampl:cor:rho-nu-erg} below, $\tilde\xi_{t+u}=\zeta(\tilde\rho_{t+u})$ a.\,s., where $\zeta:\Uu\to\UUerg$,
$\rho\mapsto\nu^{\hat\psi\circ\beta(\rho)}$.
Hence, by Lemma \ref{Sampl:lem:fob-kernel}, there exists a finite measure $\Xi$ on $\Delta$ such that $\Lambda_\Xi(\cdot,\zeta^{-1}(\cdot))$ is a regular conditional distribution of $\tilde\xi_{t+u}$ given $\tilde\xi_t$, and $\tilde\xi_{t+u}$ is conditionally independent of $(\tilde\xi_s,s\leq t)$ given $\tilde\xi_t$. The latter property is the Markov property of $(\tilde\xi_s,s\in\R)$.

Let now $(\xi_s,s\in\R)$ be a $\UUerg$-valued $\Xi$-Fleming-Viot process defined from $(\bar\rho_s,s\in\R)$ as recalled in the beginning of this section.
As in Theorem \ref{Sampl:thm:tv} (or alternatively, by an extension of Remark \ref{Sampl:rem:tvfv-fob-trans}), the process $(\xi_t,t\in\R)$ is Markovian.
By Remark \ref{Sampl:rem:tvfv-fob-trans} and as $\xi_{t+u}=\zeta(\bar\rho_{t+u})$ by definition,
$\Lambda_\Xi(\cdot,\zeta^{-1}(\cdot))$ is a regular conditional distribution also of $\xi_{t+u}$ given $\xi_t$. This implies the assertion.
\end{proof}
\begin{lem}
\label{Sampl:cor:rho-nu-erg}
Let $\xi\in\UUerg$ and let $\rho$ be a random variable with distribution $\xi$. Then the distance matrix distribution of $\hat\psi\circ\beta(\rho)$ equals $\xi$ a.\,s.
\end{lem}
\begin{proof}
By definition of $\UUerg$, there exists $\chi\in\hat\U$ with distance matrix distribution $\alpha(\nu^\chi)=\xi$. Propositions \ref{Sampl:prop:sampl-unique} and \ref{Sampl:prop:mmm-mdm-mmm} imply $\hat\psi\circ\beta(\rho)=\chi$ a.\,s.
\end{proof}

\section*{List of notation}
\sectionmark{List of notation}
Here we collect notation that is used globally in the article.

\small{
\hparagraph{Miscellaneous}
$\R_+=[0,\infty)$, $\Q_+=\R_+\cap\Q$, $\N=\{1,2,3,\ldots\}$, $[n]=\{1,\ldots,n\}$ for $n\in\N$, $[0]=\emptyset$\\
$\gamma_n$: restriction map in various contexts, (p.\,\pageref{Sampl:not:gamma_n-matr}/l.\,17, p.\,\pageref{Sampl:not:gamma_n-p}/l.\,19)\\
$\xi f=\int\xi(dx)f(x)$ for a measure $\xi$ and a function $f$, p.\,\pageref{not:integral}/l.\,-17\\
$\Mb(E)$: set of bounded measurable functions $E\to\R$\\
$\varphi(\mu)=\mu\circ\varphi^{-1}$: pushforward measure under a measurable function $\varphi$ (p.\,\pageref{not:pushforward}/l.\,15)

\hparagraph{(Marked) distance matrices}
$\Uu$: space of semi-ultrametrics on $\N$, (p.\,\pageref{Sampl:not:Uu}/l.\,-2)\\
$\hat\Uu$: space of decomposed semi-ultrametrics on $\N$, (p.\,\pageref{Sampl:not:hatUu}/l.\,6)\\
$\Dd$, $\hat\Dd$: spaces of (decomposed) semimetrics on $\N$ (p.\,\pageref{Sampl:not:Dd}/l.\,-2, p.\,\pageref{Sampl:not:hatDd}/l.\,-12)\\
$\alpha$: map that retrieves the semi-ultrametric from a decomposed semi-ultrametric (p.\,\pageref{Sampl:not:alpha}/l.\,4)\\
$\beta:\Uu\to\hat\Uu$: decomposition map into the external branches and the remaining subtree (p.\,\pageref{Sampl:not:beta}/l.\,12)
\\
$\Upsilon(\rho)$: vector of the lengths of the external branches in the coalescent tree associated with $\rho$ (p.\,\pageref{Sampl:not:Upsilon}/l.\,11)

\hparagraph{(Marked) metric measure spaces}
$\M$: space of isomorphy classes of metric measure spaces (p.\,\pageref{Sampl:not:M}/l.\,2)\\
$\U$: space of isomorphy classes of ultrametric measure spaces (p.\,\pageref{Sampl:not:U}/l.\,-4)\\
$\hat\M$, $\hat\U$: spaces of isomorphy classes of marked metric measure spaces (p.\,\pageref{Sampl:not:hatM}/l.\,8, p.\,\pageref{Sampl:not:hatU}/l.\,-1)\\
$\nu^\chi$: distance matrix distribution of $\chi\in\M$ (p.\,\pageref{Sampl:not:dmd}/l.\,-8, p.\,\pageref{Sampl:not:nuchi}/l.\,13) or marked distance matrix distribution of $\chi\in\hat\M$ (p.\,\pageref{Sampl:not:mdm}/l.\,10, p.\,\pageref{Sampl:not:nuchi}/l.\,13)\\
$\UUerg$: space of distance matrix distributions (p.\,\pageref{Sampl:not:UUerg}/l.\,-8)\\
$\psi:\Dd\to\M$, $\hat\psi:\Dd\times\R_+^\N\to\hat\M$: construction of (marked) metric measure spaces (p.\,\pageref{Sampl:not:psi}/l.\,-16, p.\,\pageref{Sampl:not:hatpsi}/l.\,-10)\\
$\Dd^*$, $\hat\Dd^*$: sets of (marked) distance matrices with a good sampling measure (p.\,\pageref{Sampl:not:Dds}/l.\,-12, p.\,\pageref{Sampl:not:hatDds}/l.\,-6)\\
$\C_n$, $\C$, $\hat\C_n$, $\hat\C$: sets of bounded differentiable functions with bounded uniformly continuous derivative (p.\,\pageref{Sampl:not:Cn})\\
$\Pi$: set of polynomials on $\U$ (p.\,\pageref{Sampl:not:Pi}/l.\,-10)\\
$\hat\Pi$: set of marked polynomials on $\hat\U$ (p.\,\pageref{Sampl:not:hatPi}/l.\,-8)\\
$\Cc$: a set of test functions on $\UUerg$ (p.\,\pageref{Sampl:not:Cc}/l.\,-6)

\hparagraph{Partitions and semi-partitions}
$\p$: Set of partitions of $\N$\\
$B_i(\pi)$: $i$-th block of a partition $\pi$ (p.\,\pageref{Sampl:not:Bi}/l.\,9)\\
$\#\pi$: number of blocks of a partition $\pi$\\
$K_{i,j}$: partition of $\N$ that contains only $\{i,j\}$ and singleton blocks (p.\,\pageref{Sampl:not:Kij}/l.\,20)\\
$\p_n$: Set of partitions of $[n]$, associated transformations (equation \reff{Sampl:eq:pn-Un})\\
$\mathbf{0}_n=\{\{1\},\ldots,\{n\}\}\in\p_n$\\
$\p^n$: Set of partitions of $\N$ in which the first $n$ integers are not all in different blocks (p.\,\pageref{Sampl:not:p^n}/l.\,16)\\
$\hat\p^n$: Set of partitions of $\N$ in which the first $n$ integers are not all in singleton blocks (p.\,\pageref{Sampl:not:hatp^n}/l.\,2)\\
$\S_n$ set of semi-partitions of $[n]$, associated transformations (p.\,\pageref{Sampl:not:Sn}/l.\,-22, p.\,\pageref{Sampl:not:Sn-transf}/l.\,-15)\\
$\Delta=\{x=(x_1,x_2,\ldots):x_1\geq x_2\geq \ldots 0,|x|_1\leq 1\}$\\
$|x|_p=(\sum_i x_i^p)^{1/p}$ for $x\in\Delta$ (p.\,\pageref{not:1-norm}, l.\,9)\\
$\kappa(x,\cdot)$: paintbox distribution associated with $x\in\Delta$ (p.\,\pageref{Sampl:not:kappa}/l.\,14)

\hparagraph{Genealogy in the lookdown model}
$\eta$: point measure on $(0,\infty)\times\p$ that encodes the reproduction events, (p.\,\pageref{Sampl:not:eta-det}/l.\,23, p.\,\pageref{Sampl:not:eta-Poisson}/l.\,-15)\\
$A_s(t,i)$: level of the ancestor at time $s$ of the particle on level $i$ at time $t$ (p.\,\pageref{Sampl:not:A}/l.\,-10)\\
$\rho_t(i,j)$: genealogical distance (p.\,\pageref{Sampl:not:rhot}/l.\,5)\\
$(r_t,v_t)$: decomposed genealogical distance (p.\,\pageref{Sampl:not:rtvt})\\
$\Xi=\Xi_0+\Xi\{0\}\delta_0$, equation \reff{Sampl:eq:dec-Xi}\\
$H_\Xi$: characteristic measure of $\eta$ (p.\,\pageref{Sampl:not:HXi}/l.\,22)
}

\bibliography{C:/Users/sg/Documents/Bib/diss}
\bibliographystyle{plain}

\normalsize
\paragraph{Acknowledgments.}
This work is part of the author's PhD thesis.
The author thanks Götz Kersting, Anton Wakolbinger, and the referees for comments and suggestions that helped to improve the presentation.
Partial support from the DFG Priority Programme 1590 ``Probabilistic Structures in Evolution'' is acknowledged. In 2017/2018, the author is supported by a postdoctoral fellowship of the Minerva Foundation.
\end{document}